 \documentclass[10pt]{amsbook}
 \usepackage[all]{xy}
 	 \usepackage[stix2]{newtxmath} 	    
 	 \usepackage[cal=boondoxo]{mathalfa} 
	 \usepackage{physics} 

 \usepackage[all]{xy}
\usepackage{tikz}
	\usetikzlibrary{shapes.geometric}
	\usetikzlibrary{intersections}
	\usetikzlibrary{3d}
	 \usepackage{tqft}
 	\usetikzlibrary{quotes,angles}
	\usetikzlibrary{arrows, calc,  backgrounds, shapes.geometric, positioning, decorations.markings,fit, patterns,positioning,matrix}
\usepackage{mathtools}
\usepackage{mathrsfs}

\usepackage[pagebackref]{hyperref} 
 \usepackage{graphicx}
 \usepackage{esvect}
\usepackage{booktabs}
\usepackage[shortlabels]{enumitem} 	
 \usepackage{longtable}
\numberwithin{section}{chapter}
\numberwithin{equation}{chapter}
\numberwithin{table}{chapter}
\newtheorem{thm}{Theorem}[chapter]
	\newtheorem*{thm*}{Theorem}
	\newtheorem{lemma}[thm]{Lemma}
	\newtheorem*{lemma*}{Lemma}
	\newtheorem{crit}[thm]{Criterion}
	\newtheorem{prop}[thm]{Proposition}
	\newtheorem*{prop*}{Proposition}
	\newtheorem{corr}[thm]{Corollary}
	\newtheorem*{corr*}{Corollary}

\theoremstyle{definition} 
	\newtheorem{dfn}[thm]{Definition}
	\newtheorem{conj}[thm]{Conjecture}
\theoremstyle{remark}
	
	\newtheorem{add}[thm]{Addition}

	\newtheorem{rmk}[thm]{\textit{Remark}}
	\newtheorem{exmple}[thm]{Example}
	\newtheorem{exmples}[thm]{Examples}
\newcommand\coker{\mathop{\rm Coker}\nolimits}
\renewcommand\ker{\mathop{\rm Ker}\nolimits}

\def\endo{\mathop{\rm End}\nolimits}
\def\half{\frac 12}

 \def\ext{\mathop{\rm Ext}\nolimits}

\renewcommand\hom{\mathop{\rm Hom}\nolimits}

 \usepackage[cal=boondoxo]{mathalfa} 
\newcommand{\bA}{{\mathbb{A}}}
\newcommand{\bC}{{\mathbb{C}}}

\newcommand{\bP}{{\mathbb{P}}}
\newcommand{\bQ}{{\mathbb{Q}}}
\newcommand{\bR}{{\mathbb{R}}}
\newcommand{\bZ}{{\mathbb{Z}}}
\def\ii{{\mbold i}}
\newcommand{\mbold}[1]{\vb*{#1}} 
	
	\newcommand{\bE}{{\mbold{E}}}
	\newcommand{\bI}{{\mbold{I}}}
	\newcommand{\bF}{{\mbold{F}}}

	\newcommand{\bo}{{\mbold{0}}}
	\newcommand{\ba}{{\mbold{a}}}

	\newcommand{\be}{{\mbold{e}}}
	\newcommand{\bff}{{\mbold{f}}}
	\newcommand{\bg}{{\mbold{g}}}
	\newcommand{\bm}{{\mbold{m}}}
	\newcommand{\bp}{{\mbold{p}}}

	\newcommand{\bt}{{\mbold{t}}}
	\newcommand{\bx}{{\mbold{x}}}
	\newcommand{\by}{{\mbold{y}}}
	
	\newcommand{\bw}{{\mbold{w}}}
	\newcommand{\bz}{{\mbold{z}}}
\newcommand\cA{{\mathcal A}} 
\newcommand\cB{{\mathcal B}}
 
\newcommand\cE{{\mathcal E}}

\newcommand\cH{{\mathcal H}}

\newcommand\cL{{\mathcal L}}
\newcommand\cO{{\mathcal O}} 
\newcommand\cP{{\mathcal P}}
\newcommand\cK{{\mathcal K}}
\newcommand\cM{{\mathcal M}}
\newcommand\cN{{\mathcal N}}

\newcommand\cX{{\mathcal X}}

\newcommand\cU{{\mathcal U}}

\newcommand\germ[1]{{\mathfrak{#1}}}

	\newcommand\geg{{\germ g}}

	\newcommand\gh{{\germ h}} 
	\newcommand\gel{{\germ l}} 
	 \newcommand\m{{\germ m}}

   	\newcommand\gs{{\germ s}}

\renewcommand\sf[1]{{\mathsf{#1}}}

	  \newcommand\sfD{{\sf D}} 
\newcommand \ul[1]{{\underline{#1}}}
\def\mapright#1{\mathop{\vbox{\ialign{
                                ##\crcr
    ${\scriptstyle\hfil\;\;#1\;\;\hfil}$\crcr
 \noalign{\kern2pt\nointerlineskip}
    \rightarrowfill\crcr}}\;}}
    \def\mapleft#1{\mathop{\vbox{\ialign{
                                ##\crcr
    ${\scriptstyle\hfil\;\;#1\;\;\hfil}$\crcr
 \noalign{\kern2pt\nointerlineskip}
    \leftarrowfill\crcr}}\;}}
\def\into{\hookrightarrow}

\def\blow#1#2{\operatorname{\mathsf Bl}_{#1}({#2})}
\def\comp{\raise1pt\hbox{{$\scriptstyle\circ$}}} 
\def\del#1#2{\frac{\partial #1}{\partial #2}} 
\def\dd#1#2{\frac{d #1}{d #2}}
\def\im{\text{Im}}
\newcommand\gl[2]{\operatorname{\mathsf GL}_{#2}({#1})} 
\newcommand\spl[1]{\operatorname{\mathsf Sp}({#1})} 
\newcommand\ogr[1]{\operatorname{\mathsf O}({#1})} 
\newcommand\cz[1]{\operatorname{\mu_{\mathsf CZ}}({#1})} 
 \def\id{\text{\rm id}}  
\renewcommand\setminus{-} 
\def\trp#1{{#1}^{\sf T}} 
\def\Cl{\mathsf{Cl}} 
\def\cone#1{\mathsf{Cone}{(#1)}} 
\def\cyl#1{\mathsf{Cyl}{(#1)}}
\def\cylend#1{\mathsf{Cyl}_{\ge 0}{(#1)}}
\def\set#1{\{ #1\}} 
\def\sett#1#2{\{ #1 \mid  #2 \}}  
\def\spec#1{\mathsf{Spec}(#1)}
\newcommand{\lb}{\left(}
\newcommand{\rb}{\right)}
\def\jac#1{\mathsf{Jac}_{#1}} 
\def\hh#1#2{\mathsf{H \hspace{-1pt}H}^{#1}(#2)}
\def\fc#1#2{\mathsf{CF}_{#1}(#2)}
\def\fcc#1#2{\mathsf{CF}^{#1}(#2)}
\def\dfcc#1#2#3{\mathsf{CF}^{#1}_{#3}(#2)}
\def\fcoh#1#2{\mathsf{HF}^{#1}(#2)}
\def\ihs{\text{IHS}}
\def\matf#1{\underline{Mat\!f}_{#1}}
{
\def\sh#1#2{\mathsf{S \hspace{-1pt}H}^{#1}(#2)}
\def\dsh#1#2#3{\mathsf{S \hspace{-1pt}H}^{#1}_{#3} (#2)} 
\def\mf#1{\mathsf F_{#1}} 
\def\lnk#1{\mathsf L_{#1}} 
\newcommand\cylind[2]{%
  \begin{scope}[shift={(#1,#2)}]
  \draw (-0.5,0)--(-0.5,2); \draw (0.5,0)--(0.5,2);
  \draw[red]  (0,0) circle [x radius = 0.5, y radius = 0.25];
  \draw[red,->] (0,2)--(0,1.5); \draw[red,->] (-0.2,1.95)--(-0.2, 1.4);\draw[red,->] (0.2,1.95)--(0.2, 1.4);
  \node at (0,1.3) {$\color{red}Y $}; 
   \end{scope};
}%
\makeindex       
\begin{document}
\frontmatter

\title[Symplectic invariants related to isolated hypersurface singularities]{On Isolated Hypersurface Singularities: Algebra-geometric and symplectic aspects.\\
 Notes of the 2022--2023 Leiden seminar. \footnote{May  2024 Version} 
} 
\author{Chris Peters}
\date{}
\maketitle
 \tableofcontents 

\mainmatter
 \chapter*{Introduction}

 \subsection*{Context and origin of the notes} These notes are based on a seminar which took place in the autumn of 2022  at the Mathematical Institute of the University of Leiden.
 Its goal was to understand the recent preprint~\cite{EvansLekili} by J.\ Evans and Y.\ Lekili, a follow-up of the papers \cite{LekiliUeda,Futaki}, 
 in which the symplectic cohomology of the Milnor fiber for specific classes of isolated singularities has  been calculated.
 
 What attracted us to this paper is first of all the interplay between the  algebra-geometric and symplectic techniques,
  a  relatively new feature, perhaps going back to the article \cite{mclean} by M.~McLean.
 In \cite{EvansLekili}   the  algebraic geometry is related to  threefold singularity   theory  as  taken up  in the 1980ies and 1990ies by M. Reid, J. Koll\`ar e.a., but
 which still is an active area of research. The symplectic techniques involve quite disparate inputs, about which more later on.
The basic relation with singularities comes from the natural symplectic  structure on the Milnor fiber of an isolated hypersurface singularity
and the natural contact structure on its link. 
\par
The main motivating question is: "what implications have symplectic and contact  invariants 
for   algebra-geometric phenomena  of a singularity?"
Note that symplectic  and contact invariants are (much) finer than topological invariants, but much harder to calculate.
Only recently this has been achieved for several classes  of isolated hypersurfaces, in particular in the above mentioned papers.

One of the striking new results  of \cite{EvansLekili}  is the computation of    contact invariants for the link of  some of these  singularities. 
As a result,  contact structures for certain diffeomorphic  links  in dimension $5$  
could be distinguished  using  these invariants. 
\par
On the algebra-geometric side  there is a (largely conjectural) interplay between symplectic invariants 
and the existence of a so-called small resolution. For several threefold singularities 
a precise conjecture in this direction has been resolved,   another striking result  of \cite{EvansLekili}. 
  
\subsection*{Some historical background}
  
Singularity theory  in complex and differential geometry is a fairly old and well-established branch of mathematics.  See e.g.~\cite{singbook,milnorMorse}.
In   differential geometry  the object of study consists of   the
critical points of a function $f: M \to \bR$, where $M$ is some differential manifold. A point $m\in M$ is critical if $df(m)=0$. Considering  second order derivatives one
introduces the Hessian  at $m$, a certain real quadratic form. Then $m$ is said to be non-degenerate if  the Hessian at $f$  is non-degenerate. If all critical points
of $f$ are non-degenerate, then  $f$ is called a Morse function. Choosing a metric on $M$, one associates to  the Morse function $f$ its  gradient vector field $\nabla (f)$.
The  Betti numbers  of $M$ can now be estimated, and in some cases calculated, following the flow of $\nabla(f)$, using  the indices of the Hessians at the various 
critical points of $f$.  See for instance \cite{milnorMorse}. 
\par
Associating to an index $k$ critical point a $k$-cell, the free $\bZ$-module on these cells can be made into a homological complex by defining the boundary operators
using the flow of $\nabla (f)$.
This idea  is due to    S.\ Smale~\cite{smale1}  and  others as explained in R.\ Bott~\cite{bott}.  See  \cite{hutch} for an  introduction to these ideas.
In section~\ref{ssec:MorseHom}, the reader finds a summary of it as a warm-up for a   variant  called  symplectic cohomology.  
The essential ingredient  here is Floer (co)homology named after    A.\ Floer   \cite{floer}.
 Originally  Floer's approach played an important role for  the  understanding of the topology of $3$- and $4$-manifolds.
 A.\ Floer  and H.\ Hofer in \cite{FH1}, and   A.\ Floer,   K. Cieliebak, H.\ Hofer in  \cite{FH2}   extended these ideas   to  
 the  symplectic world. Taking limits in various ways the resulting symplectic (co)homology  groups comes  in different flavors  
  depending on the precise context of the applications. 
  Whatever version one chooses, these groups are  notably hard to calculate.

Very recently it has been realized that  for singularities defined by  functions
 $w_A:\bC^{n+1}\to \bC$ with a critical point at $\mbold 0$, 
coming from  invertible matrices $A=(a_{ij})\in \gl  \bC {n+1}$ (see Eqn.~\eqref{eqn:PolSing}),
one can define   Hochschild cohomology of the associated category  of matrix factorizations.
On the other hand,  these singularities give rise to certain  Fukaya categories  and their  mirror-duals
related to their symplectic geometry as sketched in Section~\ref{sec:SympIsHH} of these notes.  
Homological mirror symmetry  in this case  consists in replacing $A$ by its transpose and  
conjecturally the Hochschild cohomology of the category of matrix factorizations for $A$ is the same as the
symplectic cohomology for the Milnor fiber of the singularity $\set{w_{\trp A}=0}$.
This prediction  from homological mirror symmetry has been proven for several kinds of these singularities, 
cf.\ Proposition~\ref{prop:ConjABTrue}. 
Since  Hochschild cohomology  is amenable to explicit calculation,  in these cases
 symplectic   cohomology for the Milnor fiber of the corresponding \ihs\ can be calculated as well.   
Moreover, there is an extra algebraic structure present on Hochschild cohomology, that of a Gerstenhaber algebra.
 One of the main results of \cite{EvansLekili} states that this leads to a contact invariant  for the link of large classes of  such singularities.
 
\subsection*{About the seminar} Special attention was given to   so called  small resolutions of special singularities. 
See Section~\ref{sec:smallres} for the algebra-geometric background and
and \ref{sec:SmallResSH} for  the above mentioned  (conjectural) relation with symplectic geometry.
It turns out that this area presents a fascinating source of examples for the interplay of algebraic geometry and symplectic geometry.

  Evans and Lekili  use the above discussed recent techniques from homological mirror-symmetry in their paper.
   The participants in the seminar have  various backgrounds and specializations
 in algebraic geometry and/or  symplectic geometry but   were not familiar with all of these techniques. 
 In the seminar the required results  from these fields were then treated as a black box, 
 with the exception of the elaborate input from matrix factorizations.

\par
Such  an ambitious program with inputs from rather disparate field makes access difficult.
So the idea arose to work out  the talks    to make the  results  from \cite{EvansLekili} more accessible to both algebraic and symplectic geometers. 
This   unavoidably implies that some chapters  might be well known to either one of these groups, but the participants of the
lectures all felt that such a text would serve the greater goal of introducing the mathematical community to this exiting and challenging
intersection of two fields   dealing with singularities from totally different angles.

The resulting notes presented here  entirely reflect my view as an algebraic geometer 
well versed in the older differential geometric literature, 
but a  dilettant in matters of symplectic geometry and
 the finer points of matrix factorizations, especially their categorical aspects. 
\par
 
 The writing up of these notes proceeded 
 in parallel  with  a  project on cDV-type singularities related to small resolutions by three of 
 the speakers of  the seminar, 
and whose outcome is  the recent preprint~\cite{APZ}. I could not resist explaining
 (at the end of  Section~\ref{sec:SmallResSH}) some of the enticing  new results  they obtained.

\subsection*{Acknowledgement} 
\begin{small}
In writing this extended version I  have  had 
several long explanatory discussions with the participants of the seminar, N.\ Adaglou,  F.~Pasquotto, A.\ Sauvaget 
and  A.\ Zanardini for which I want to thank them.
I also want to thank Thomas Dyckerhoff for explaining some points of \cite{Dyck} and M. Hablicsek for help with Hochschild cohomology of dg-categories.
\end{small}

\newpage
\clearpage
\section*{List of Notation} 

\begin{small} 
\begin{longtable}{| l | l | l |} 
 \hline
\emph{\textbf{Symbol}}  & \emph{\textbf{meaning}}  & \emph{\textbf{page}} \\
\hline
\endhead
\ihs\ & isolated hypersurface singularity & \pageref{page:ihs}\\
$w_A=0$& invertible polynomial \ihs\ & \pageref{page:wA}\\
$\lnk {X,x}$ &link of the singularity germ $(X,x)$ & \pageref{page:link}\\
$\mf{f}=\mf{X,x}$&Milnor fiber of & \\
& \hfill the singularity germ $(X,x)$, $X=\set{f=0}$&\pageref{page:minorf}\\
 $\mu(w)=\mu(X,x)$& Milnor number of &\\
 & \hfill the singularity germ $(X,x)$, $X=\set{f=0}$& \pageref{page:milnn} \\
 $ \jac w$  &Jacobian ring of $w\in \bC[x_1,\dots,x_{m+1}]$&\pageref{page:JacW} \\
 $T^*U$ & total space of&\\
 & \hfill  the cotangent bundle of  the manifold $U$ & \pageref{page:cotbndl}\\
 $ \lambda_{\rm can}, \omega_{\rm can}$& canonical $1$ and $2$ form on $T^*U$& \pageref{page:canforms} \\
 $\omega_{C^n}$& canonical symplectic form on $\bC^n$&\pageref{page:canformCn}\\
 $\sh * {\mf  {X,x}}$& symplectic cohomology of the germ $(X,x)$& \pageref{page:scoh}  
 \\
 $\hh * {A,\Gamma_A}$ & Hochschild cohomology associated to &
 \\
 &\hfill  the matrix $A$ and  group $\Gamma_A$& \pageref{page:HHcoh}\\
 $A_{1,2k}$, $\alpha_{1,k}$& cDV-singularity $A_1(2k)$ and its link  & \pageref{page:Alpha1k}\\
 $\blow V {W} $  & blow up of smooth variety $W$ in subvariety $V$ & \pageref{page:BlowUp}\\
$T_U,T^*_U$ & tangent, resp. cotangent bundle of $U$ &\pageref{page:CanDiv} \\
 $\omega_Z$, $K_X$ & canonical sheaf, canonical divisor  of $X$  &   \pageref{page:CanDiv}\\
 $\Cl_x(X)$, $\rho(x)$&local class group of $(X,x)$ and its rank & \pageref{page:lcg}\\
 $\iota_Y$ & contraction against vector field $Y$& \pageref{page:contract}\\
 $\cL_Y$ &Lie derivative in direction of vector field $Y$& \pageref{page:Lie} \\
 $R_\alpha$ &Reeb vector field for contact form $\alpha$ &\pageref{page:Reeb} \\
 $\cyl {M_\alpha} $ & symplectization of contact manifold $(M,\alpha)$& \pageref{page:symplectize}\\
 $\widehat W$& symplectic completion of Liouville domain $W$&\pageref{page:SympComp}
 \\
 $\spl V,\,  \spl {2n}$&symplectic group of $V\simeq \bR^{2n}$& \pageref{page:SympGr}\\
 $\mu(\psi)$ &Maslov index of path $\psi$ in $\spl {2n}$&\pageref{page:MaslovInd}\\
 $\cz {H,\bx}$ &Conley--Zehnder index&\\
 & \hfill  of smooth curve $\bx$ of Hamiltonian flow of $H$&\pageref{page:CZIndex} \\
 $C_*^{\rm Morse} M$& Morse chain groups of manifold $M$&\pageref{page:Mcomplex} \\
 $H_*^{\rm Morse}(M), H^*_{\rm Morse(M)}$&Morse (co)homology of manifold $M$ &\pageref{page:Morse} \\
 $\fcc  *H$ & Floer  chain groups of Hamiltonian $H$&\pageref{page:floer} \\
 $ \fcoh * {\widehat W,H} $ & Floer  cohomology  of Hamiltonian $H$ &\\
 & \hfill   on completion of Liouville domain $W$   & \pageref{page:FloerCoh}\\
$  \sh * {\widehat W}^{<a}$  & symplectic cohomology of $\widehat W$ &\\
&\hfill  w.r. periodic orbits of periods $<a$&\pageref{page:SH<} \\
 $ \sh * W$ & symplectic cohomology of  $ W$ & \pageref{page:SH} \\
 $\dfcc  \pm  H *$   & positive/negative Floer& \\ 
 &\hfill   chain groups of Hamiltonian $H$ & \pageref{page:shPM} \\
 $\dsh {k}  W\pm $  &positive/negative  Floer &\\
 &\hspace{6em}   cohomology of   $W$ &  \pageref{page:shPM} \\
 $md(X,x)$&minimal discrepancy &\\
 &  the singularity germ $(X,x)$ &\pageref{page:md} \\
  $\text{hmi}(\lnk {X,x}, \xi), \text{hmi}(\lnk {X,x}, \xi)$  &(highest) minimal index of the  link & \\
 & \hfill  of the  singularity germ $(X,x)$  & \pageref{page:mi}\\
$N^\bullet(\bff)$&Koszul sequence for $R$-regular sequence $\bf$ & \pageref{page:Koszul}\\
$ \set{\bff,\bg}$ & Koszul matrix factorization w.r. to $\bff$, $\bg$ & \pageref{page:KoszMf}\\

$\underline{C}(R)$& category of complexes over $R$& \pageref{page:complexes}\\
$ \underline{C}_{dg}(R)$ & dg-category of complexes over $R$&\pageref{page:dg} \\
\hline
$[\underline{A}]$ &homotopy category of $\underline{A}$ & \pageref{page:homcat}\\
$\matf {R,w},\matf{R,w}^\infty  $&category of matrix factorisations of $w\in R$&\pageref{page:mtf}  \\
$M^{\rm stab}$ &stabilization of $M$&\pageref{page:Stab} \\
$X^o$-module $M$& $M$ with action of $X$ from the right&\pageref{page:Xo} \\
$ \widehat {\underline{C} }$ & "completion" of category $\underline C$&\pageref{page:CHat} \\
$\Delta$ & diagonal of $R$ in $R\otimes _k R$, $R=k [\! [ x_1,\dots, x_m]\!]$ &\pageref{page:Delta}  \\
$\hh * A$& Hochschild cohomology of algebra $A$& \pageref{page:HHA}\\
$  A^o$, $A^e$ & opposite and enveloping algebra of $A$& \pageref{page:envelope} \\
$C^{\rm bar}_*(A)$ & bar-complex of algebra $A$& \pageref{page:BarCmplx}\\
$  \underline{A}^o$, $\underline{A}^e$ & opposite and enveloping dg-category of $\underline A$&\pageref{page:OppAndExtCat} \\
$\Delta_{\underline{A}}$ & diagonal/identity functor of  dg-category  $\underline A$ &\pageref{page:IdFunct}\\
$\matf{w,\Gamma,\chi}$ & category of equivariant matrix&\\
 & \hspace{5em}  factorizations of $w$ w.r. to $(\Gamma,\chi)$  &\pageref{page:EqvMatf}  \\
$G_A$& kernel of character $\chi_A: \Gamma_A\to \bC^*$&\pageref{page:GA} \\
$\chi_j$&  character  of $\Gamma_A$& \pageref{page:ChiJ}\\
$\chi_\bm$&character of $(\bC^*)^{n+2}$& \pageref{page:ChiM}\\
$A_\gamma$, $B_\gamma$, $C_\gamma$&   sets of $\gamma$-monomials &\pageref{page:AGamma}, \pageref{page:CGamma}\\
$\widetilde {G_A}$ &  auxiliary group &\pageref{page:AuxGroup}\\
\hline
\end{longtable}
  \end{small}


\chapter{Overall view} 
	\label{lect:overview}

{\small I shall  work over the complex numbers and use the analytic topology unless   mentioned  otherwise.}

\section{The protagonists} 
\label{sec:protagonists}

A complex variety $V$ is a subset of $\bC^N$ given by the vanishing of a finite number of holomorphic functions.
One assumes throughout that  $V$ has  just one singularity at the origin $\mbold 0 \in \bC^N$.
Such isolated singularities  are investigated in appropriately small balls centered at $\mbold  0$.
Special attention is given to   isolated hypersurface singularities, 
 abbreviated  \emph{\ihs}  in what follows. Explicitly,\index{singularity!isolated hypersurface ---}
\label{page:ihs} 
\begin{dfn} An $m$-dimensional analytic hypersurface 
\[
V(f):= \sett{\bx=(x_1,\dots,x_{n+1})\subset \bC^{m+1}}{f(x_1,\dots,x_{m+1})=0}
,\quad f(\mbold 0)=0
\]
has an \textbf{isolated singularity} at $\mbold 0$ if
 there is an open neighborhood $U\subset \bC^{m+1} $ 
 of $\mbold 0$ such that $\mbold  0 $ is the only zero  of  $\nabla f $ in $U$  along $V(f)$.
\end{dfn}

\begin{exmples} \label{ex:ihs} \textbf{1.} The type $A_n$ singularities or double point singularities on curves: $x^2-y^{n+1}=0 $, $n\ge 1$. For odd $n$ these have two branches $x=\pm y^{\half(n+1)}$ (the red curve) and for $n$ odd
only one (the black curve). The latter  are the so-called cusps.  
\begin{figure}[htpb]
     \includegraphics[scale=0.31]{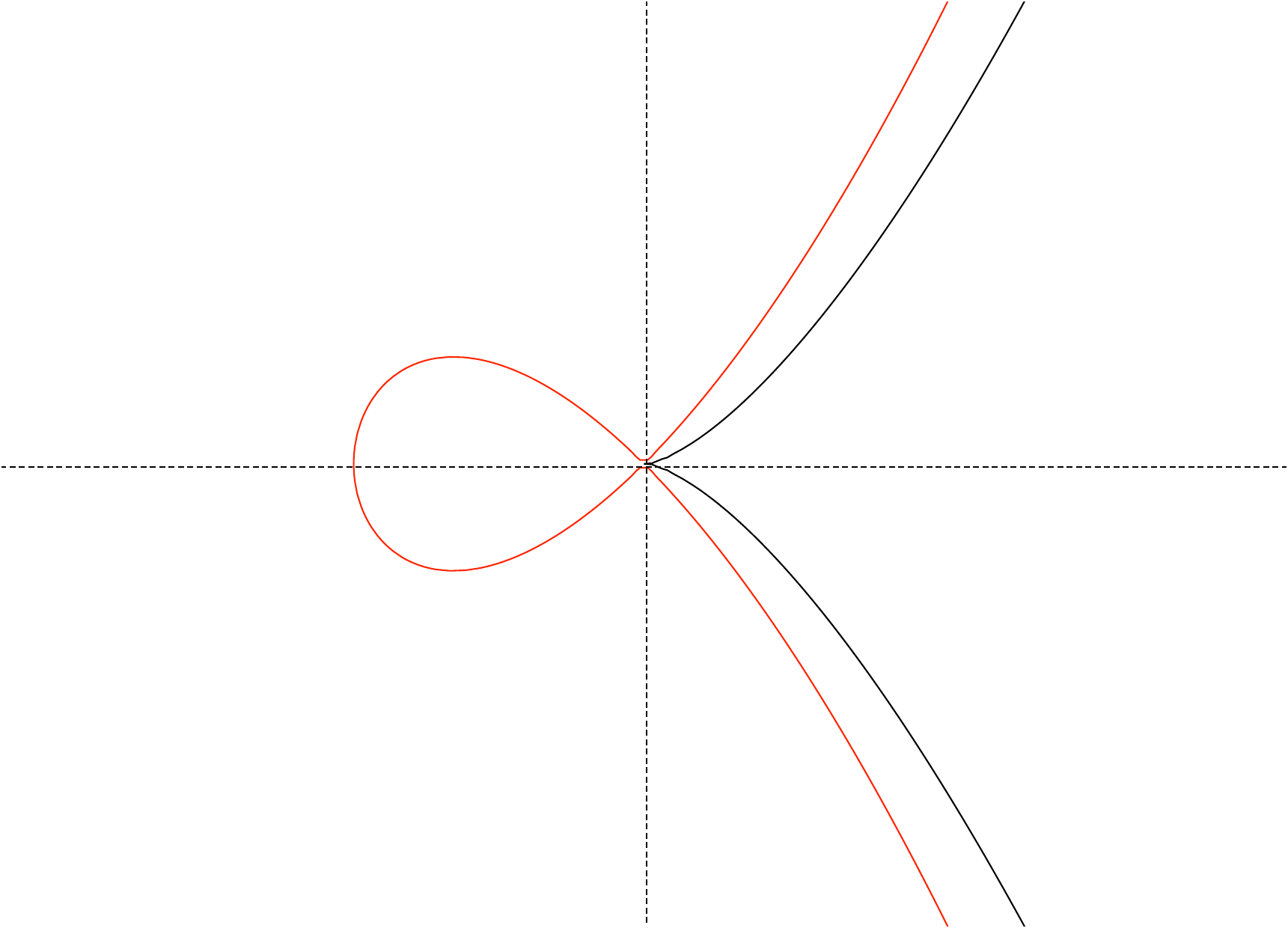}
    \caption{Double points.}
    \label{fig:curvsings}
\end{figure}

\noindent 
\textbf{2.} The triple point curve singularities $D_n$ (given by  $y(x^2-y^{n-2})=0$, $n\ge 4$),\index{singularity!type $A$-$D$-$E$}
 and the three $E$-types $x^3+y^4=0$, $x(x^2+y^3)=0$ and $x^3+y^5=0$, respectively $E_6,E_7$ and $E_8$.
\\
\textbf{3.} The \textbf{\emph{du Val surface singularities}} are obtained from the  $A$-$D$-$E$-types by adding the square of a new variable. Adding more squares of new variables,\index{du Val singularity}
  then  gives  \ihs\ of higher dimension, likewise called "of $A$-$D$-$E$-type".\index{singularity!du Val ---}
\\
\textbf{4.}  \index{singularity!compound du Val (cDV) ---}The compound du Val threefold singularities, abbreviated \emph{\textbf{cDV-singularities}}: $g(x,y,z)+ t h(x,y,z,t)$
by definition are such that   the hyperplane $t=0$ gives  a du Val surface singularity.
\\
\textbf{5.} Singularities associated to invertible matrices. Consider the polynomial in $\bC[x_1,\dots,x_{m+1}]$ given by \label{page:wA}
\begin{equation}
\label{eqn:PolSing} w_A(\mbold x):=\sum_k   x_1^{a_{k,1}}  x_2^{a_{k,2 }} \cdots  x_{m+1}^{a_{k,m+1}} , \quad A=(a_{ij})\in \gl  \bC {m+1}.
\end{equation} 
If   $w_A=0$ has an \ihs\ at $\mbold 0$, one speaks of an  \textbf{\emph{invertible polynomial \ihs}}. This is the case for instance if $A$ is diagonal with exponents $\ge 2$.
Note  that the equations $\sum_j a_{ij}  d_i= d$, $d\in \bQ$, $i=1,\dots, m+1$, have a unique solution over  the rationals.
Clearing denominators, there is a unique solution $(d_1,\dots,d_{m+1}, d)$ with $d$ a positive integer and with $\gcd( d_1,\dots,d_{m+1}, d)=1$.\index{singularity!invertible polynomial ---}
If $A$ is diagonal, the \ihs\ is called a  \textbf{\emph{Brieskorn--Pham singularity}}.\index{Brieskorn-Pham singularity}\index{singularity!Brieskorn-Pham ---}

The polynomial  $w_A$ is  a  so-called  \textbf{\emph{weighted homogeneous polynomial}} of type $(d,[d_1,\dots,d_{m+1}])$,
 which means that if $t\in \bC^*$ acts on $\bC^{m+1}$ by multiplying $x_j$ by $t^{d_j}$, the induced action
on polynomials sends $w_A$ to $t^d w_A$.  \index{weighted homogeneous polynomial}The associated integer $\alpha(w_A)=d-\sum d_j  $ is 
called the \textbf{\emph{amplitude}} of $w_A$. \index{amplitude}\index{log-Fano}\index{log-Calabi--Yau}\index{log-general}Its sign plays
an important role in the theory:  If $\alpha(w_A)<0$ one calls $w_A$ a \textbf{\emph{log-Fano type}} polynomial, if $\alpha(w_A)=0$, it is of   
\textbf{\emph{log-Calabi--Yau type}}, while
 a \textbf{\emph{log-general type}} polynomial  has $\alpha(w_A)>0$.

\end{exmples}

\section{Links and Milnor fibrations}
\label{topology}

For an $m$-dimensional isolated singularity $(X,x) \subset (\bC^{m+1} ,\bo)$, 
    not necessarily a hypersurface singularity,  the  \textbf{\emph{link}} is defined as the $(2m-1)$-dimensional manifold
which is obtained by intersecting $X$ with  a small enough sphere centered at $\bo$:\index{link of singularity}\label{page:link} 
\[
\lnk {X,x}: =  X \cap S^{2m+1}(\bo,  \epsilon),\quad 0<\epsilon  \ll  1.
\]
For all small enough $\epsilon$ the oriented diffeomorphism type of this manifold  does not change. 
\par
In the \emph{\textbf{hypersurface case}}  $X=\set{f=0} \subset \bC^{m+1}$  the  map
\[
 S^{2m+1}(\mbold  0,\epsilon)\setminus \lnk {X,x}  \mapright{\phi_f} S^1,\quad  \phi_f(\mbold x) = f(\bx)/ |f(\bx)|
 \]
 is well defined.    J. Milnor shows  \cite[Thm. 4.8]{milnorbook} that $\phi_f$  is a differentiable locally trivial fiber bundle with smooth
 $2m $-dimensional  fibers. The general fiber is  called the
  \emph{\textbf{Milnor fiber}} $\mf  f=\mf {X,x}$ of the germ $(X,x)$ given by $f=0$. The  topology of this fibration is well-understood especially if 
  $f$  is a polynomial.\index{Milnor!fiber (of singularity)}\label{page:minorf} 
\begin{center}  \textbf{\emph{In    these notes, I shall mostly use the letter $w$ for polynomial singularities.}}
\end{center}

  \begin{thm}[\protect{\cite[\S 5]{milnorbook}}] \label{thm:milnor}The Milnor fiber $\mf  w$ and the link $\lnk w$ of an isolated  $m$-dimensional singularity of a polynomial singularity
  $w=0$ have the following properties:
  \begin{enumerate}[ \bf 1.]
\item $\mf  w $ is (orientably) parallelizable, i.e its tangent bundle admits   an orientation preserving trivialization;
\item $\mf  w $ has the homotopy type of a wedge of $m$-spheres; in particular, its middle homology $H_m(\mf  w)$, $2m=\dim \mf  w$,  is free of rank $\mu(w)$;
\item Each  fiber  of the Milnor fibration  $\phi_w$ has the link $\lnk w$ as its boundary;
\item If $m\ge 2$, then $\lnk w$ is $(m-2)$-connected, that is, it is connected and  its homotopy groups $\pi_k(\lnk w)$ vanish for $k=1,\dots,m-2$; 
\item  $\mf  w $ is $(m-1)$-connected.
\end{enumerate}

   \end{thm}
   
  The number $\mu(w)=\mu(X,x)$ of $m$-spheres in this theorem  is called the \emph{\textbf{Milnor number}} of the singularity  germ $(X,x)$ given by $w=0$. \label{page:milnn}
  It can also be calculated\index{Milnor!number} 
algebraically as the dimension of
  the \emph{\textbf{Jacobian ring }}$\jac w$ of $w\in \bC[x_1,\dots,x_{m+1}]$: \label{page:JacW}
  \begin{equation}
  \mu(w) =\dim \jac w ,\, \jac w = \bC[x_1,\dots,x_{m+1}]/ J(w),\quad J(w)= \left(\del w{x_1},\dots,\del w{x_{m+1}}\right).
  \label{eqn:JacRing}
  \end{equation} 
 
 Using S.\ Smale's technique of surgery  there is  a  sharper statement   in the case $m\not=2$. This sharper statement implies the existence of
 certain types of Morse functions on the Milnor fiber which play a crucial role later (cf. the  statement of Corollary~\ref{cor:contactinv}).
 \par
 Let me first give a short explanation of this technique.  One starts with  
 an $m$-dimensional manifold with boundary $(W, \partial W)$ 
 and such  that $\partial W$ contains $S^{k-1}\times B^{m-k}$. More precisely, one assumes that there exists a 
  smooth map 
  \[
  S^{k-1}\times B^{m-k} \mapright{\phi} \partial W.
  \]
 Here $B^{s}\subset \bR^s$ denotes the  unit  ball in $\bR^{s}$.
One attaches   a $k$-handle  $H^k= B^k\times B^{m-k}$ to    $W$ by taking first the disjoint
union of $W$ and $H^k$ and then glues $S^{k-1}\times  B^{m-k}\subset \partial H^k$ to $\partial W$ using $\phi$.
At the same time  $B^k\times S^{m-k-1}$,  the other part of the boundary of $H^k$, replaces the image of 
$\phi$ in $\partial W$:
\[
(W' ,\partial W') =\left(  W \cup_\phi  H^k   , 
(\partial W \setminus  \im(\phi))   \cup  B^k\times S^{m-k-1}  \right).
\]
Then $W'$ said to be obtained from $W$ by attaching the  $k$-handle $H^k$ 
and $\partial W'$ is obtained from $ \partial  W$ by an \textbf{\emph{elementary  surgery}} of type $(k,m-k)$. \index{surgery!elementary --}
An  $m$-manifold with boundary obtained by successively attaching handles (possibly of varying types) is called
a \textbf{\emph{handlebody}}.\index{handlebody} 
   
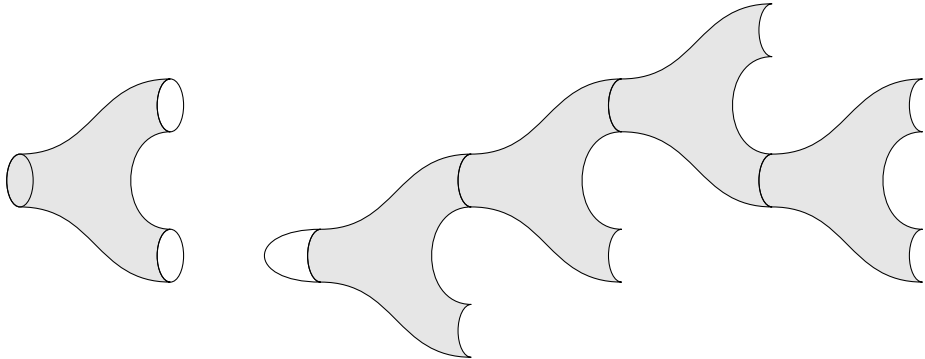
\begin{figure}[h] 
\begin{center}

\begin{tikzpicture}
  \begin{scope}[tqft/flow=east]
\node[draw,tqft/pair of pants,fill=black,fill opacity=0.1,tqft/boundary style={draw , black , thin,fill opacity=0}]   {};
 \end{scope}  
 
 \begin{scope}[shift={( 4,-1)},tqft/flow=east] 
	\node[draw,tqft/cap,fill=black,fill opacity=0] at (-2,0)  {};
	\node[draw,tqft/pair of pants,fill=black,fill opacity=0.1]     {};
	\node[tqft/pair of pants,fill=black,fill opacity=0.1,draw]   at (2,1) {};
	\node[draw,tqft/pair of pants,fill=black,fill opacity=0.1]   at (4,2)  {};
	\node[draw,tqft/pair of pants,fill=black,fill opacity=0.1]   at (6,1)  {};
	 \end{scope}

\end{tikzpicture}
\caption{Elementary surgery on $S^1$ of type $(1,1)$ and a typical Milnor fiber in (complex) dimension $1$ viewed as a handlebody.\label{fig:elsurg}}
\end{center}
\end{figure}

In the present  setting  the Milnor fiber is $(m-1)$-connected and the link $(m-2)$-connected and in this situation
 Smale's result    \cite[Theorem 1.2]{smale}  applies, yielding:
 \begin{thm}[\protect{\cite[Thm. 6.6]{milnorbook}}] If $m\not=2$ the Milnor fiber of an isolated $m$-dimensional singularity having Milnor number $\mu$
is obtained from the $m $-ball by attaching $\mu$ disjoint $m$-handles.
\end{thm}
  
This indeed refines Milnor's result: 
Attaching  an $m$-handle   to the $2m$-ball  gives a manifold which is homeomorphic to the product $S^m\times D^m$
which has the $m$-sphere as a deformation retract and
 attaching $\mu$ disjoint $m$-handles   has a wedge of $\mu$ such $m$-spheres as deformation retract. In Figure~\ref{fig:elsurg} the right hand
 picture is supposed to have  connected boundary (the   handles are all open at the back, making a large knot-like boundary).
 It  represents the Milnor fiber of an irreducible $1$-dimensional singularity (see Section~\ref{sec:plane} for some background). However,  in case the singularity has $k$ branches the boundary has $k$ connected components.

 Before explaining the consequence for Morse functions, let me first recall the definition:
 
 \begin{dfn} Let $M$ be a  smooth $n$-dimensional manifold, and $f:M\to \bR$ a smooth function. A point $p\in M$ is a 
 \textbf{\emph{critical point}} if $df(p)=0$ and it is a \index{critical point!non-degenerate --- of index $r$}
 \textbf{\emph{non-degenerate critical point (of index $r$)}} if locally in a neighbourhood $U$  of $p$ coordinates
  $ x_1,\dots,x_n $ can be found centered at  $p$  so that  $f(x)=- \sum_{j=1}^r x_j^2+\sum_{j=r+1}^n x_j^2$ in $U$.
The function $f$ is a \textbf{\emph{Morse function}}\index{Morse!function}  if  all critical points of $f$ are non-degenerate.
The flow lines of the gradient vector field of $f$ is called the \textbf{\emph{flow}} associated to   the Morse function $f$.
 \end{dfn}
 
 In a precise sense "most" smooth functions are More functions so that small perturbations of any smooth function gives a  Morse function.
 The critical points of Morse functions and their  indices  can be used to describe a manifold  by means of attaching handles which gives some information 
 about the topology of the manifold. See \cite{milnorMorse} for more information and many examples.
 Sometimes the information is ideal: there exists a so-called \textbf{\emph{perfect Morse function}} with\index{Morse!perfect --- function!}   precisely\footnote{$b_r(M)=\rank{H_r(M)}$ is the $r$-th  Betti number of $M$.} $b_r(M)$ critical points of index $r$ for all $r$. This is the case for Milnor fibers of isolated hypersurface singularities:
 
 \begin{corr} \label{cor:OnMorseFuncts}
 In case the complex dimension $m$ is different from $2$, there exists a Morse function on  the Milnor fiber   with a minimum (index $0$) and $\mu$ non-degenerate critical points of index $m$.
 \end{corr}
 \begin{proof} One starts off with the $2m$ ball  $W_0$ with Morse-function $f_0:W_0\to [-3,-1]$ given by 
 $f_0=-1- 2\sum_{j=1}^{2m} x_j^2$   and then one  consecutively attaches  
 $ m$-handles as follows.
   \begin{figure}[h]
     \includegraphics[scale=0.07]{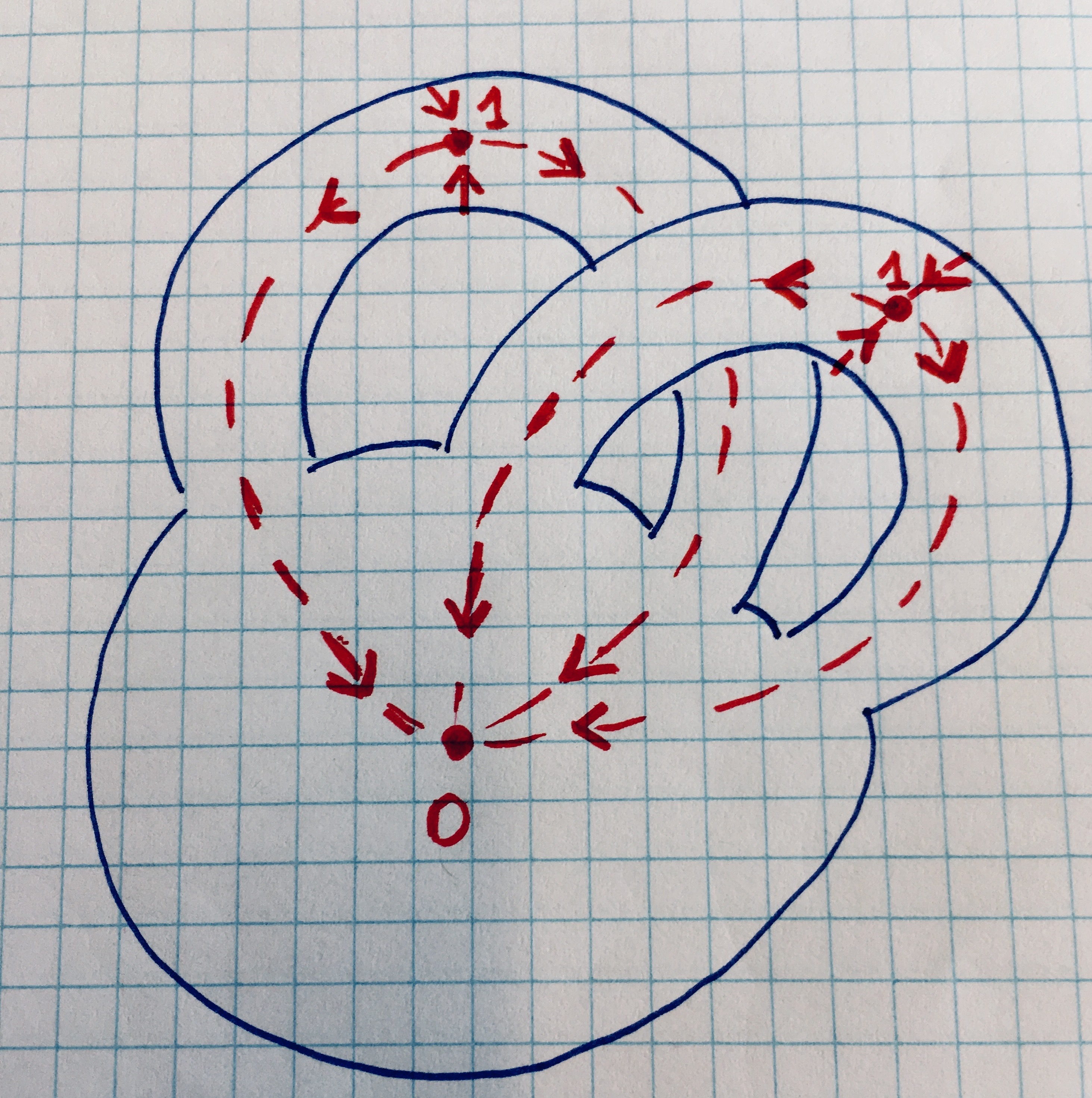}
    \caption{Handlebody with flow associated to  the Morse function.}
    \end{figure}

 By \cite[Thm. 3.12]{milnorcobord} an elementary surgery of type $(m,m)$  
 applied to $\partial W_0$  is given by  a manifold $W_{0\to 1}$ 
 with "lower" boundary $\partial W_0 $
 and  Morse  function $f_{01}:W_{0\to 1} \to [-1,1]$
 having one non-degenerate critical point of index $m$ and being $-1$ on   $\partial W_0 $ and $1$ on the "upper" boundary of \index{surgery!elementary --}
 $W_{0\to 1}$. Let $W_{1}=W_0\cup W_{0\to 1}$ be the   $(2m)$-fold obtained from $W_0$ by attaching the $m$-handle produced by the 
 elementary surgery. The functions $f_0$ and $f_1$ are both $0$ on   $\partial W_0 $  and
by construction (see the proof of \cite[Thm. 3.12]{milnorcobord})
  $f_1$ glues differentiably to $f_0$ without critical points near this boundary  and so gives a Morse function $f_{1}: W_{1} \to [-1,1]$.
 A further elementary surgery of type $(m,m)$  applied to the upper boundary of $W_{0\to 1}$ gives the  manifold $W_ 2 $  which is obtained from 
 $W_{1}$ by attaching an  $m$-handle and which has a Morse function with one more critical point of index $m$ on the attached $m$-handle. 
  Continuing in this manner one  obtains a Morse function $f_{\mu}$ on the Milnor fiber with one critical point of index $0$ 
  and $\mu$ critical points of index $m$, one for  each attached handle.
  \end{proof}

 \begin{exmple} \label{exm:MfandLnk} \textbf{1.}  One can show  that the link of the cusp $x^2-y^3=0$ is 
 homeomorphic to a \textbf{\emph{trefoil-knot}}   \index{knot!trefoil --}given parametrically by 
 $   x =(2+\cos 3t)\cos 2t $,  $y =(2+\cos 3t)\sin 2t $, $z =\sin 3t $.
 \begin{figure}[h]
     \includegraphics[scale=0.07]{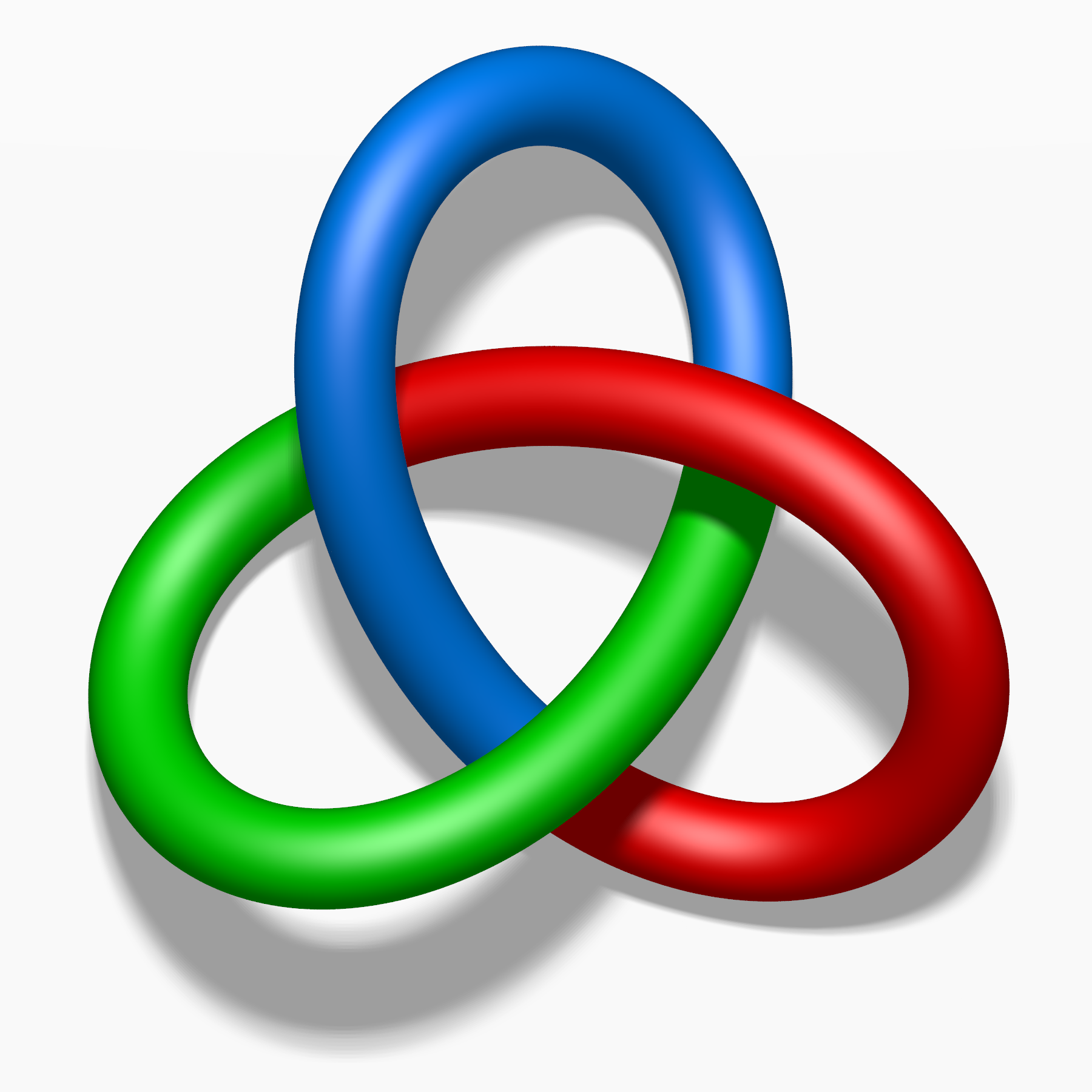}
    \caption{Disc bundle on trefoil knot (By Jim Belk - Own work, Public Domain, \\\texttt{https://commons.wikimedia.org/w/index.php?curid=7903214})}
    \end{figure}

\noindent 
 The Jacobian ring is spanned by $1$ and $y $ and so 
 the Milnor number equals $2$. The Milnor fiber is diffeomorphic to a torus minus a disc spanning the trefoil knot. Hence  the  Milnor fiber  contracts to the union of a  latitudinal   and longitudinal circle, i.e. a wedge of two circles.
 \\
 \textbf{2.} Consider the singularity $z_1^2+z_2^2 + z_3^2=0$. Its link can be calculated in real coordinates 
 $ x_1,y_1,x_2,y_2,x_3,y_3$ given by $z_j= x_j+\ii y_j$. Indeed the equation gives $\sum_j x_j^2- y_j^2=0$,
 $ \sum_j x_j y_j=0$ and the sphere condition gives $\sum x^2_j+y^2_j =\epsilon^2$. Hence the link is given in $\bR^3\times \bR^3$ by
 the equations  $\sum x_j^2= \sum y_j^2=\half \epsilon^2$, $\bx \cdot \by=0$ which gives the subset of the tangent bundle to $S^2$ consisting of tangent vectors of 
 fixed length $\half\sqrt{2} \epsilon$. This is also called the \textbf{\emph{Stiefel manifold}} $\mathsf {St} (3,2)$. Similarly, for $\sum_{j=1}^{n+1}\bz_j^2=0$\index{Stiefel manifold}
 one obtains  the bundle of tangent vectors to $S^n$ of fixed length, the Stiefel manifold $\mathsf {St} (n+1,2)$, an $S^{n-1}$-bundle over $S^n$. This is not always a product,
 but for $n=1,3,7$ it is (see \cite[\S 8.5 and \S 27]{steenTFB}, see also \cite{adams}).
  \end{exmple}
 
 There is another way to view the Milnor fibration by considering a complex valued  function $ f $ defining an \ihs\ at $\bo$ on a small enough ball $B=B(\mbold  0, \epsilon)$.
 In a small enough disc $\Delta(0,r)$ the sets 
 $F(\bx,t):=\set{f(\bx)=t }\cap B$
 for $t\not=0$ are  open subsets of  $n$-dimensional algebraic varieties without singularities while $F(\bx, 0)$ is of course  the defining  \ihs. \index{Milnor!fiber (L\^e style)}
  Milnor shows in \cite[\S 5]{milnorbook}  that this yields an alternative incarnation of the Milnor fibration over the punctured disc. A more precise  result of  L\^e  states:

\begin{thm}[\protect{\cite{le}}] For $0< r \ll \epsilon \ll 1$, the family of (open) complex manifolds $   \set{f(\bx)=t}\cap \bar B(\mbold  0, \epsilon) $ over the punctured disc
 $\Delta(0,r) \setminus \set {0}$ is  a locally trivial fiber bundle  with  fibers  diffeomorphic to the Milnor fiber.
 Their   boundaries    are   all  diffeomorphic to the  link $\lnk f$.\label{thm:AltMF}
 \end{thm}
 
 \section{Enter: the associated symplectic   and  contact structure}
 \label{sec:introsymgeom}
 
 \begin{small}
 The Milnor fiber of an \ihs\ has a canonical symplectic structure and its boundary, the link,
  inherits a canonical contact structure. This is briefly explained in this  section.
 For more on these notions I refer to \chaptername~\ref{lect:symp&contact} and to the book~\cite{SYmpTop} by D.\ McDuff and D.\ Salamon.  
  \end{small}
 \begin{dfn} \textbf{1.}  A real $2$-form $\omega$ on a smooth manifold is \textbf{\emph{non-degenerate}} if the skew form it defines on every tangent space is non-degenerate.
 \\
 \textbf{2.} A \textbf{\emph{symplectic structure}} \index{symplectic!structure}on a smooth manifold $N$ is given by a closed non-degenerate real $2$-form $\omega$. 
 A symplectic manifold is\index{symplectic!manifold}
 a smooth manifold equipped with a symplectic structure. A symplectomorphism $f: (N_1,\omega_1) \to (N_2,\omega_2)$  between symplectic  manifolds $N_1, N_2$,
 is a diffeomorphism such that $f^*\omega_2=\omega_1$. Two manifolds are called \textbf{\emph{symplectomorphic}} if there is a symplectomorphism between them.\index{symplectomorphism}
 \\
 \textbf{3.}  A \textbf{\emph{contact structure}} \index{contact!structure}on a smooth manifold $M$ of odd dimension $2n-1$ is a field  $\xi$ of codimension $1$ hyperplanes in the tangent bundle  of $M$ which defines a maximally non-integrable 
 distribution. This means  that at every point $ m\in M$  two vector fields  tangent to $\xi$ can be found whose Lie bracket is non-zero. A contact manifold \index{contact!manifold}is an odd-dimensional manifold admitting a contact structure. A contactomorphism is a diffeomorphism between contact manifolds preserving the contact structures. Two manifolds are called \textbf{\emph{contactomorphic}} if there is a contactomorphism between them.\index{contactomorphism}

A  one-form $\alpha$   such that $\ker(\alpha)= \xi$ is called a \textbf{\emph{contact form}}.\index{contact!form} 
Locally in a chart with coordinates $ x_1,\dots, x_{2n-1}$ such a form exists: if
$\displaystyle \sum a_j(x_1,\dots,x_{2n-1}) \frac{d } {d x_j} =0$ is a field of tangent hyperplanes, $\alpha= \sum_j  a_j dx_j$ is a corresponding
contact form. It is clearly unique up to multiplying with a function.
This can be done globally, if   $\xi^\perp$ can be oriented, i.e., if the normal vector field $\displaystyle \sum a_j \frac{d } {dx_j} $  can be scaled locally so that it glues to give a global normal
vector to the field  $\xi$ of hyperplanes. This is called a  \textbf{\emph{co-orientation}}.
The corresponding form $\alpha$  is said to be induced by the chosen  co-orientation. 
 Note that if $\alpha$   is induced   from a  co-orientation, then   $-\alpha$ is induced from the opposite co-orientation.
 
\par 
The contact structure being maximally non-integrable is equivalent to $\alpha\wedge (d\alpha)^n\not= 0$.  The non-uniqueness of the contact form of course  holds  globally:  if  two forms $\alpha,\alpha'$ have  $\xi$ as its kernel, $\alpha'=f\cdot \alpha$ where $f$ is a nowhere zero function. Then $f>0$ precisely when both forms come from the same co-orientation. 
  
  \end{dfn}
  
 \begin{exmples}[Symplectic manifolds]  \label{ex:BasicExs} 
 \begin{enumerate}[\bf 1.,nosep,left  = 0.3\parindent,itemindent=2em] \index{cotangent bundle}
 \item The \textbf{\emph{total space $T^*U$ of the cotangent bundle}}   \label{page:cotbndl}
 \[
 \pi: T^*U \to U,\quad U \text{   a smooth manifold}.
 \]   
 Since  the $1$-forms on $U$ are the sections of the cotangent bundle, there is  a canonical $1$-form $\lambda_{\rm can}$ 
 on $T^*U$ defined by \label{page:canforms}
\begin{equation}
\label{eqn:LambdaCotBndle}
 \lambda_{\rm can}(u, \alpha_u ) =\alpha_u, \quad u\in U, \alpha_u\in T^*_u(U),
\end{equation} 
 and hence a canonical exact $2$ form $\omega_{\rm can}:= d \lambda_{\rm can}$. In a chart $V$ with coordinates $x_1,\dots,x_n$, their differentials at $u$
 form a basis for the cotangent space and so, if $\alpha_u$ is a cotangent vector, $\alpha_u = \sum_j y_j dx_j$. The $y_j$ give coordinate functions on each cotangent space $T_p^*U, p\in V$,
 and so, together with the $x_j$ give a chart on $\pi^{-1}V$. The canonical $1$-form in this chart is    given by $\sum_j  y_jdx_j$ and so 
 \[
 \omega_{\rm can}:= \sum_j dy_j\wedge dx_j
 \]
  is non-degenerate and defines a symplectic  structure.
  \par
 As a special case,  consider $T^*\bR^n$.
If one identifies  $T^*\bR^n$ with $\bC^n$ by sending $(x_j,y_j)$ to $z_j= x_j+\ii y_j$, the canonical two-form  
 $\sum_j dy_j\wedge dx_j$ to   becomes 
\begin{equation}
\label{eqn:StFormCn}
 \omega_{\bC^n}:= \half \ii \sum_j dz_j\wedge 
 d\bar z_j= d \big(  \underbrace{\half \ii \sum z_j \wedge d\bar z_j}_{\lambda}  \big).
\end{equation}
 In other words $(T^*\bR^n,\omega_{\rm can})\simeq (\bC^n,  \omega_{\bC^n})$. \label{page:canformCn}
 
 \item  \textbf{\emph{K\"ahler manifolds}} \index{K\"ahler!manifold}
 A  K\"ahler form on a complex manifold   is a closed   real $2$-form  of type $(1,1)$ and a pair $(X,\omega)$
 consisting of a  complex manifold $X$ equipped with a K\"ahler form $\omega$ is a K\"ahler manifold.
   In local (complex) 
 coordinates $z_1,\dots,z_n$ one has  $\omega = \half \ii  \sum_{i,\bar j} h_{i,\bar j} dz_i\wedge  \overline{dz}_j$ with 
 $ h:=(h_{i,\bar j}) $ a hermitian matrix (since the form is real).
 The non-degeneracy of $\omega$ is equivalent to $\det (h)\not=0$, i.e. $h$ is a metric. This metric is called the associated \textbf{\emph{K\"ahler metric}}. Here are some concrete examples:\index{K\"ahler!metric}
  \begin{itemize} 
  \item  $\bC^n$ with its standard hermitian metric. This is the symplectic manifold $(\bC^n,  \omega_{\bC^n})$ of example 1.
   \item  $\bP^n$ with the so-called Fubini--Study metric
 \[
 \omega_{\rm FS} : \frac{\ii}{2\pi} \partial\overline\partial \log \| \bz \|^2, \quad \bz= (z_0:\dots:z_n),\, \| \bz \|^2=\sum_j |z_j|^2.
 \]
  \end{itemize} 
 Since a submanifold $Y\subset  X$ of a K\"ahler manifold $(X,\omega)$ inherits a K\"ahler structure $\omega|_Y$ from the the one on $X$, all open or closed submanifolds
of a K\"ahler manifold are K\"ahler. In particular this holds for projective manifolds, that is, submanifolds of $\bP^n$.
 
\item \textbf{\emph{Milnor fibers.}}
The Milnor fiber   of an \ihs\  with equation $\set{f(\bz)=0}$ with an  isolated singularity at $\mbold 0\in  \bC^{n+1}$ 
carries a symplectic structure coming from the K\"ahler form on $\bC^{n+1}$.  Specifically, consider the  the Milnor fibration  
$f: B\setminus B\cap\set{f^{-1}0}  \to  \Delta^*(r)$ as in  Theorem~\ref{thm:AltMF}. 
If $B^o$ is the interior of $B$, the  manifold $ B^0 \setminus B^0\cap\set{f^{-1}0}$ 
as an open subset of   the K\"ahler manifold  $\bC^{n+1}$ is K\"ahler and so are 
 the submanifolds $f^{-1}t$ which are  copies  of the Milnor fiber $\mf f$. 
\end{enumerate}
 \end{exmples}
 
\begin{exmples}[Contact manifolds]  \label{exmpl:BasSymp}
\begin{enumerate}[\bf 1.,nosep,left  = 0.3\parindent,itemindent=2em]  
\item \textbf{\emph{Odd dimensional unit spheres.}} \index{unit sphere}As above, identify  the symplectic manifold $T^*\bR^n$ 
with $(\bC^n,\omega_{\bC^n})$. The real unit $(2n-1)$-sphere can be identified with $\|\bz\|=1$. The canonical $1$-form on
$T^*\bR^n$  in complex coordinates is given by
\[
\alpha:= \frac \ii 2 (\sum _j z_jd\bar z_j- \bar z_j dz_j)=  \sum_{j} x_j dy_j-y_jdx_j ,\quad z_j=x_j+\ii y_j.
 \]
At a point $\bp\in S^{2n-1}$ the (real) tangent space $T_pS^{2n-1}$ can be viewed as $\bp^\perp$. The (almost) complex structure
on $\bR^{2n}$ (which gives the identification with $\bC^n$)  is the operator $J: T_p\bR^{2n}  \to T_p\bR^{2n} $
sending $(\cdots,x_j,y_j,\cdots)$ to  $(\cdots,-y_j,x_j,\cdots)$. The smallest complex subspace of $\bp^\perp$ is given by
\[
 \xi_\bp:= T_\bp^\bC \bR^{2n} =\sett{X\in T_p\bR^{2n} }{X\cdot \bp=JX\cdot \bp=0}.
 \]
Almost by definition,  $X\in \xi_p$ belongs to $\ker \alpha_p$ and conversely. Indeed, writing $X=(\cdots,X_j,Y_j,\dots)$, 
and $\alpha_\bp=   \sum_{j} x_j dY_j-y_jdX_j $,  one has $\alpha_\bp (X)= \bp\cdot X=0$, while 
$\alpha_\bp (JX)= \bp\cdot JX=  0$.  The converse is   clear because of dimension reasons.
Using that $d(\bz\cdot \overline{\bz})= \sum  dz_j  \bar {z_j}+ 
\bar z_j  dz_j=0$ on the sphere, one verifies easily that  $\alpha\wedge (d\alpha)^n|_{S^{2n-1}} \not= 0$ and so $\xi$ is
a contact structure on the sphere.

The unit sphere can have other contact structures. See for instance \cite{eliash2}. 
The contact structures on $S^3$ have been classified. Up to isotopy there is
exactly one contact structure on the boundary $\partial M=S^3$ of  a
symplectic  $4$-manifold $M$ with an exact symplectic form $
d\alpha$, $\alpha|_{\partial M}$ the standard contact form on $S^3$.
Such a  contact manifold  is said to have a symplectic filling (see Definition~\ref{dfn:SympFill} below). 
All others, called \textbf{\emph{overtwisted}},  are classified by $\pi_2(S^3)\simeq \bZ$.

\item  As a generalization of the foregoing, the \textbf{\emph{unit sphere  bundle $ S(T^*U)$ in the total space of the cotangent bundle}} 
$T^*U$ of a Riemannian manifold $(U,g)$ is a contact manifold
with  the form $\lambda_{\rm can}$ (see \eqref{eqn:LambdaCotBndle})
restricted to the unit sphere bundle as its contact form. 

In the local coordinates on $V$ of the first example above of a symplectic manifold,
\index{cotangent bundle} the form $\sum_j y_j  dx_j$ has as its kernel  at $(\bx,\by)$ the collection of tangent vectors
$\displaystyle \sum u_j \del{} {y_j}+ \sum v_j \del {} {x_j}$   for which  $\sum y_j v_j=0$. This gives a hyperplane   in the tangent space at $(\bx,\by)$ of  $ S(T^*U) \subset T^*U$.
  In particular, for $U=\bR^n$ the product  $\bR^n\times S^{n-1}$ receives a contact structure.
  
  \item \textbf{\emph{Links.}}\index{link}
  Up to diffeomorphism the link  of an isolated hypersurface singularity $(X,x)$ is the boundary of its Milnor fiber.
  As in the first  example above,   a complex structure on an even dimensional differentiable
manifold $M$ gives an almost complex structure $J$ on the tangent bundle of $ M$, i.e., a bundle morphism $J$ on  
$TM$ for which $J^2=-\id$. As in that  example, one 
can give $\lnk {X,x} =\partial\mf  {X,x} \cap S(0,\epsilon) $  the    contact structure  
$ T(\lnk  {X,x}) \cap\, J(T(\lnk  {X,x}) )$. Alternatively,
  the contact form is the restriction to the link of the form
   $\lambda_{\bC^{n+1}}:=\half (\sum_{j=1}^{n+1}  x_j dy_j -y_jdx_j)$ for which $d\lambda_{\bC^{n+1}} = \sum_j  dx_j\wedge dy_j =  \omega_{\bC^{n+1}}$.
\end{enumerate} 
\end{exmples}

  The last example exhibits   a so-called symplectic filling of a contact structure:

\begin{dfn} \label{dfn:SympFill} A contact structure $(M,\xi)$ admits a \textbf{\emph{symplectic filling}} $(N,\alpha)$ if the following conditions hold simultaneously:\index{symplectic!filling}\index{filling!symplectic ---}
\begin{enumerate}
\item $\partial  N=M$;
\item a contact form $\lambda$ for $\xi$ exists such that $d\lambda = \alpha|_M$.
\end{enumerate}
\end{dfn}

Summarizing, I  have now shown:

\begin{prop} The Milnor fiber of an \ihs\  $(X,x)$ carries a symplectic structure which gives a symplectic filling of the   contact structure  on  $(\lnk  {X,x},\lambda_{\bC^{n+1}}|_{\lnk  {X,x} })$,
where  $\lambda_{\bC^{n+1}}=\half (\sum_{j=1}^{n+1}  x_j dy_j -y_jdx_j)$, 
and  $z_j=x_j+\ii y_j$, $j=1,\dots,n+1$ are the standard coordinates on 
$\bC^{n+1}$.
\end{prop}

The Milnor fiber as a symplectic filling of the link $\lnk  {X,x}$ is also called a \textbf{\emph{Milnor filling}} of $\lnk  {X,x}$.
 \index{Milnor!filling}\index{filling!Milnor ---}
The classical invariants of   the Milnor fiber  $\mf   {X,x}$ of $f$ are of topological and differentiable nature and except for low dimensions
do not  in general  give information on  
the symplectic structure. An invariant which does is   the so-called \textbf{\emph{symplectic 
cohomology algebra}}  $\sh * {\mf  {X,x}}$ \label{page:scoh}
treated in more detail in \S~\ref{sec:scoh}. This algebra has a rich structure -- as will be   shown  later --,\index{symplectic!cohomology} 
but in general is very hard to calculate. 
In  Section~\ref{sec:newresults}    classes of \ihs\ will be given where one -- thanks to a flurry   of recent activities --
 does have a detailed knowledge of this algebra.

\section{When are \ihs s considered equal?}

In complex algebraic geometry two \ihs\ s  given by holomorphic functions $f.f': U \to \bC$, $U\subset \bC^{n+1}$ with an isolated 
critical point at $\bo$  are considered to be equivalent if $f$ transforms to $f'$ after a biholomorphic coordinate 
change valid in a small enough neighborhood of $\bo\in U$.  Clearly, such singularities have symplectomorphic Milnor fibers and
and  contactomorphic  links. 

There are weaker equivalences which play a role in symplectic geometry, notably the one induced by deformations of singularities:

\begin{dfn}  A function  $f(\bx,t): \bC^{n+1} \times  U \to \bC$, $U\in \bC$ open,   polynomial in $\bx$ and depending smoothly  on 
$t$ is called a \textbf{\emph{smooth deformation}} of \ihs s  if \index{equivalence of \ihs}\index{Arnold's unimodal hyperbolic  singularity}
\begin{itemize}
\item for all $t\in U $ the hypersurface $f(\bx,t)=0$ 
has an isolated singularity at $\bo\in \bC^{n+1}$;
\item the Milnor fibers of $f(\bx,t)=0$ and  their links deform smoothly with $t$.
\end{itemize}
The singularities are then called deformation equivalent. 
\end{dfn}

What this concept makes interesting is that there are smooth  deformations of  \ihs s which  are not equivalent..

\begin{exmple}
Consider the so-called  $2$-dimensional hyperbolic $T_{p,q,r}$-singularity  given by:
  \[
   f_a=x^p+y^q+z^r +a xyz,  \quad a\in \bC^\times , \quad \frac 1p+\frac 1q+\frac 1r <1.
  \]
  For all  non-zero $a\in \bC$ the polynomial   $f_a=0$  has  an isolated singularity at the origin with 
  Milnor number   equal to     $\mu(f)=p+q+r-1$. This is the case  since the   Jacobian ring   
  is spanned  by $1$ together with the monomials $x^k$, $k=1,\dots, p-1$, $y^\ell$, $\ell=1,\dots, q-1$, $z^m$, $m=1,\dots, r-1$,
together with $xyz$. Note that the family is smooth in $a\in\bC$ but for $a=0$ the Milnor number is always different: 
its is equal to $(p-1)(q-1)(r-1)$ (a basis of the Jacobian ring is given by the monomials $x^\alpha y^\beta z^\gamma$, 
$0\le \alpha\le p-2$, $0\le \beta \le q-2$, $0\le \gamma\le r-2$). 
For $a\not=0$ the Milnor fibers and links deform smoothly with $a$.
  It is classical that the parameter $a$ 
  is a modulus, i.e. the complex structure of the singularity varies with $a$; 
  the example is one of Arnold's  unimodal singularities \index{singularity!unimodal ---}as discussed in \cite[Ch. 2.3]{singbook}.
\end{exmple}

 \section{Symplectic invariants for isolated normal singularities}
 \label{sec:newresults}
 
\subsection{Using Hochschild cohomology}\index{Hochschild!cohomology}
  I shall now discuss briefly  very recent results concerning  the \ihs\ given by     Eqn.~\ref{eqn:PolSing}.   
The symplectic  cohomology  $\sh *  {\mf  {w_{\trp A}}}$ of the Milnor fiber of the  "mirror" $w_{\trp A}$  
    is   conjecturally  equal to a certain algebra  which is in an explicit way  associated to the 
    pair $(A,\Gamma_A)$ where   $\Gamma_A$ is the   finite extension   of $\bC^*$   given by
 \[
\Gamma_A:= \sett{\bt:=(t_0,\dots,t_{n+1})\in (\bC^\times)^{n+2}}{t_1^{a_k,1}\cdots t_{n+1}^{a_k,n+1} = t_0\cdots t_{n+1}, k=1,\dots,n+1 }.
\] 
In these notes this algebra will be denoted $\hh * {A,\Gamma_A}$. \label{page:HHcoh} It is a so-called Hochschild
algebra, the definition of which  will be given  in  \chaptername~\ref{lect:mf}
after having explained the required techniques from the theory of  matrix factorizations. 
  
Here I just explain some of the crucial features and ingredients. The character  
\[
\chi_A  :   \quad \Gamma_A \mapright{\quad}  \bC^*,\quad 
 		  \quad \bt \mapsto t_0 \cdots t_{n+1}
\]
has a finite kernel,  showing that  $\Gamma_A$ is indeed a finite extension of $\bC^*$.  
The invertible matrix $A$ is similar to a diagonal matrix $\text{diag}(d_1,\dots,d_{n+1})$ with $d_1|d_2|\cdots|d_{n+1}$. These positive integers are the elementary divisors
of the finite abelian group $ \ker (\chi_A)$.

 The group  $\Gamma_A$ acts on the polynomial $w_A$ by multiplying $x_j$ by $t_j$. Then $\bt (w_A)= \chi_A (\bt)  \cdot w_A$. So $w_A$ is a  
 semi-invariant for the  $\Gamma_A$-action with character $\chi_A$.
 This set-up makes it possible to apply the theory of so-called $\Gamma_A$-equivariant matrix factorizations
 explained in Section~\ref{sec:EqMatFac} . It turns out  that this theory yields  the  algebra $\hh{*}{A,\Gamma_A}$ that I mentioned, and,   as will be detailed below  has been calculated for several   classes of matrices $A$.

As I noted before, conjecturally  $\sh *  {\mf {w_{\trp A}} }$ and  $\hh{*}{A,\Gamma_A}$  are   isomorphic, and so, if this is the case,
 the latter    gives  computable symplectic  invariants.  
 For the present status of the conjectural isomorphism I refer to Section~\ref{sec:SympIsHH}, 
 especially Proposition~\ref{prop:ConjABTrue}.
 For now it suffices to mention that it holds in all cases treated in \cite{EvansLekili} and so in particular for the
 diagonal cDV-singularities
  which  are treated in more detail in these notes  (here $A =\trp A$ so here the situation is self-mirrored).


\subsection{Contact invariants} 
 \index{contact!invariants}
One calls an isolated normal  singularity  \textbf{\emph{topologically smooth}} if its link is diffeomorphic to the standard sphere.\index{singularity!topologically smooth ---}
\index{topologically smooth singularity}\index{Mumford's theorem!topologically smooth=smooth for surfaces}
 In dimension $2$ a renown result of D.\ Mumford  implies that  an isolated  normal surface singularity is topologically smooth if and only if
 it is smooth. See \S~\ref{sec:surfsings}. 
 This ceases to be true in higher dimensions as shown by E.\ Brieskorn~\cite{brieskorn}, e.g. the singularity $x^2+y^2+z^2+w^3=0$ is topological smooth
 but not smooth.
 
  There are only a few results about  the contact structure on the link of an isolated singularity in higher dimension:
 \begin{enumerate}
\item A result of I. Ustilovski~\cite{usti} states that for  each $m>0$ there are infinitely many isolated singularities  for which its link is diffeomeorphic to $S^{4m+1}$
 but which are not mutually contactomorphic.\index{contact structure!on $S^{4m+1}$}
 \item   M.\ Kwon and O.\ van Koert \cite{kk} have shown that the contact structure of  the     Brieskorn--Pham singularities
 $ \sum_j z ^{a_j}=0$ determines whether  the singularity is canonical in the sense of M. Reid (see Definition~\ref{dfn:can}). In other words,    a Brieskorn--Pham singularity  presenting  a  canonical singularity is a  property of the canonical contact structure of the  link.\index{Brieskorn-Pham singularity}\index{singularity!Brieskorn-Pham ---}\index{contact structure!on link of Brieskorn singularities}
 \item Work of M. Mclean~\cite{mclean}    characterizing isolated normal Gorenstein singularities   $ \set{w=0}$ for which  
 $H^1(\lnk w)$ is   torsion,  in terms of contact invariants. He also has shown that Mumford'd theorem can be extended to  isolated normal singularities
   of dimension $3$ if one replaces  "topologically smooth" by   "contactomorphic to the standard $5$-sphere".
 See \S~\ref{sec:Mclean} for an exposition of his results. \index{contact structure!on $S^5$}
\end{enumerate}

The Hochschild algebra    $\hh{*}{A,\Gamma_A}$  discussed in the previous subsection
is generated by  certain monomials in the polynomial ring   $ \bC[x_0,\dots,x_{n+1}, x_0^{-1},\dots,x^{-1}_{n+1}]$, 
where  the $x_j^{-1}$ are  given degree $-1$.\footnote{This is slightly imprecise since there is no cancellation between the $x_J$-variables and the $x_j^{-1}$-variables. In Section~\ref{sec:EqMatFac} this will be remedied.}
These degrees determine the cohomological degree. Unlike ordinary cohomology,
this  will be seen to imply that  Hochschild cohomology can have   (even infinitely many)  negative degrees.

The $\bC^\ast$-action on $\bC[x_0,\dots,x_{n+1}]$  which for $t\in \bC^\times $ multiplies only $x_0$ by $t$ does  not affect  the 
polynomial $w_A$ but gives a second grading on  $\hh{*}{A,\Gamma_A}$. As mentioned above,    the Hochschild 
cohomology $\hh{*}{A,\Gamma_A}$ is spanned by
certain monomials $x_0^{b_0}\cdots x^{b_{n+1}}_{n+1}  (x^{-1}_0)^{c_0}\cdots (x^{-1}_{n+1})^{c_{n+1}}$.
The  second grading on  $\hh d {A,\Gamma_A}$  is then given by the total degree $a_0=b_0-c_0$ of $x_0$ of such a monomial.
Conventionally this gives a class in $\hh{d- na_0, na_0}{A,\Gamma_A}$, but sometimes a different scaling
is preferable, changing $na_0 $ to $ma_0$ for some $m\in \bZ$. 
This torus-action on $\hh{*}{A,\Gamma_A}$,  yielding the second grading, has a counterpart on symplectic cohomology which   
under  rather restrictive conditions is shown to be
 a contact invariant for the contact structure on the link, as will be  explained in Section~\ref{sec:SympIsHH}. 
 The basic underlying structure which makes this possible is that of a Gerstenhaber algebra,
whose definition is  given  in Section~\ref{sec:Gerst}.

\begin{exmple} \label{exm:A1k}  Consider  the first non-trivial example of a cDV-singularity 
\[
A_1(2k): \quad x^2+y^2+z^2+w^{2k}=0.
\]  
This example has Milnor number $2k-1$ since the Jacobian ring is generated by $1,w,\dots,w^{2k-2}$. Hence the topological structure depends on $k$.
In   Example~\ref{exm:MfandLnk}.2  one saw that for $k=1$ the link is  diffeomorphic to $S^2\times S^3$. Below it will be shown
that   this is  also true for  $k\ge 2$. See example~\ref{ex:gendpeven}.\index{contact structure!on $S^2\times S^3$}
The contact structures turn out to depend on $k$. The dimensions of  the symplectic cohomology groups are given by
\[
\dim (\sh d {A_1(2k)})= \begin{cases}
				0 & d\ge4\\
				2k-1 & d=3\\
				0 & d=2\\
				1 & d\le 1.
  			\end{cases} 
\]
The induced contact structure on the link of $A_1(2k)$ will be denoted  $\alpha_{1,k}$. \label{page:Alpha1k}Using monomials  representing  the generators,  one calculates the second grading which shows that the links  are  mutually not contactomorphic. This is explained
in  Section~\ref{ssec:bigrading}.  The result is  summarized in Table~\ref{tab:A1k}.
  \end{exmple}

 \begin{rmk} The link of the singularity of $A_1(2k)$ is the Brieskorn manifold $\Sigma(2,2,2,2k)$, equipped with the contact structure defined by the contact form
    \[
    \alpha_k=\frac{\ii}{4}\sum_{j=0}^2 (z_jd\bar{z}_j-\bar{z}_j dz_j) + \frac{\ii k}{4} (z_3d\bar{z}_3-\bar{z}_3 dz_3).
    \]
    That the  contact structures $\alpha_{1,k}$ (on $S^2\times S^3$) are all pairwise non-isomorphic, was already shown in \cite{Uebele} using positive symplectic cohomology. 
\end{rmk}

\subsection{Relation with small resolutions}

\index{small resolution}The kind  of singularities  coming up in these notes  are also investigated in algebraic
geometry. The hypersurface singularities from  Section~\ref{sec:protagonists}  
 are examples of singular points on an  affine variety.
The main tool from algebraic geometry to study singularities  is called desingularization. This  is discussed in some detail in
\chaptername~\ref{lect:cDV}  with an eye towards the class of the cDV-singularities from Example~\ref{ex:ihs}.\textbf{4}.
As will be explained there, although one generally needs to replace  a singularity by a divisor in order to
obtain a smooth variety, sometimes  glueing in lower-dimensional varieties already yield smooth varieties.
The process leading to it is then called a small resolution.

It is a natural  question whether this can be detected on the level of symplectic geometry. For several  examples of cDV-singularities this has been affirmed and has led to a precise conjecture, stated and explained in Section~\ref{sec:SmallResSH}.

 \chapter[Classical results on the topology of isolated singularities]{Classical results on the topology of isolated singularities}
 	\label{lect:classic}

\section*{Introduction}
In this  chapter  classical topological concepts related to isolated singularities will be reviewed: 
\begin{itemize}
\item the monodromy operator for the Milnor fibration,
\item  knots and $1$-dimensional singularities, 
\end{itemize}
Furthermore some basic results are reviewed 
\begin{itemize}
\item Mumford's result  implying that smoothness of a normal surface singularity can be phrased in terms of its link and so 
it is a purely topological property,

\item  Milnor's characterization of the link in terms of the monodromy operator,
 
\item  The implication of Smale's differential topological classification of $5$-manifolds    for links of $3$-dimensional \ihs.
\end{itemize}
 	 
\section {Central notions}  \label{sec:classcentral}

\begin{dfn}  \label{dfn:Sings} Let $(W,x)$ be a germ of a complex analytic variety. \\
1. The point $x\in W$ is \emph{\textbf{normal}} \index{normal point}\index{singularity!normal---}if the local ring $\cO_x(W)$ of germs of holomorphic functions at $x$  is integrally closed in its quotient ring.  A smooth point is always normal, but singular points may or may not be normal. 
For instance isolated curve singularities are not normal. Reducible  surfaces are singular 
in non-normal points, forming  the intersection  of  two of their
components.
Isolated surface singularities need not be normal.
\\
2. A \textbf{\emph{resolution}}  of $(W,x)$ is a proper  morphism $\sigma: (\widetilde W,E)\to (X,x)$, $E$ a subvariety of $\widetilde W$,\index{resolution}
  such that  $\widetilde W$ is non-singular and $\sigma:  \widetilde W\setminus E \mapright{\sim}   W\setminus x$ is   biholomorphic.  $E $ is   is called the 
  \emph{\textbf{exceptional locus}}.\index{exceptional locus}
  \\
  3. If the exceptional locus has  codimension $\geq 2$, i.e. if  it is not a divisor, the resolution is called a \emph{\textbf{small resolution}}.\index{small resolution}\index{resolution!small ---}
  These do exist: see \S~\ref{sec:smallres}.
  \\
  4. A singularity $(W,x)$ is \emph{\textbf{rational}} if for one (and hence for every)  resolution $\sigma:Y\to W$, the  higher
  direct images $R^k\sigma_*(\cO_Y)$, $k\ge 1$,  vanish (this only affects their stalks at $x$).  \index{singularity!rational ---}
\end{dfn}

\begin{exmple} \label{exm:BlSm}The simplest example of a resolution is  a blow up of $\bC^n$ 
  in a  dimension  $m$  subspace  $V$. \label{page:BlowUp}
 To define the blow up in $V$  one can use the smooth variety $G_m V $ of  
 all linear subspaces of $\bC^n$ of dimension $m+1$ passing through $V$,
 a variety isomorphic to $  \bP^{n-m-1}$:
\begin{equation*}
q: \blow V {\bC^n} \to  \bC^n, \quad  \blow   V {\bC^n}  = \sett{(W , x )\in G_m V  \times \bC^n }{ x \in  V},
\end{equation*}
where $q$ is the projection onto the second factor.
The exceptional divisor $q^{-1}V$ in this case is isomorphic to  $\bP^{n-m-1} \times V$. 
If $X$ is a   complex  manifold, \textbf{\emph{the blow up}}  $\blow ZX$\index{blow up} 
in a smooth subvariety $Z\subset X$   can be defined locally just as in the linear setting, and then glue the results.

\end{exmple}

A special case of  H. Hironaka's desingularization theorem~\cite{Hironaka1,Hironaka2}, reads as follows:\index{Hironaka!resolution theorem of ---}
\begin{thm}[Hironaka]
Let $X$ be an algebraic subvariety of $\bC^N$ having an isolated singularity at $p$. Then there is a sequence of  successive blow ups 
$\sigma_j, j=1,\dots,r$,  along smooth subvarieties, say  
 $\sigma=\sigma_r\comp\cdots\comp\ \sigma_1: \widetilde \bC^N \to \bC^N$, such that
\begin{itemize}
\item The proper transform $W$ of $X$ in $\widetilde \bC^N$ is smooth;
\item the  exceptional locus   $E'\subset  \widetilde C^N $  is a hypersurface with strict normal crossings, that is,
 the irreducible components of $E'$ are  smooth and either do not  intersect  or  cross normally
 i.e., there are local coordinates $(z_1,\dots,z_n)$    on $\widetilde \bC^N  $ so that $E$ is  given in this coordinate patch by $z_1\cdots z_k=0$;
 \item The irreducible components of $E'$ meet $W$ transversally so that the exceptional locus $E=E'\cap W$ is a
 hypersurface in $W$ with strict normal crossing.
\end{itemize}
\end{thm}

\begin{exmple} \label{ex:embres} By Example~\ref{exm:BlSm}, the blow up of $ \bC^n$  at $\bo$ is defined as 
 $ \blow \bo  {\bC^n}  =\sett{(\ell , x)\in \bP^{n-1} \times \bC^n}{x\in \ell}$. If $X\subset \bC^n$, $\blow \bo X$ is
   the  closure of $X\setminus \bo$ in  $ \blow \bo  {\bC^n}$.  
  If $\bo\in X$ is a singularity and $\blow \bo X$ is smooth, this gives an   embedded resolution.
 
 As an example, consider the threefold    $X\subset  \bC^4$ with equation $x_1x_4-x_2x_3=0$   which 
 is singular at the origin $\bo$. It is  the cone $\cone Q$ over a quadratic surface  $Q\subset \bP^3$ with the 
same homogeneous equation. 
Now  $Q \simeq \bP^1\times \bP^1$ as one can see as follows.  In  inhomogeneous coordinates  the point 
$(\lambda,\lambda')\in \bP^1\times\bP^1$ can be identified with $ p_{\lambda,\lambda'}=(\lambda\lambda': \lambda:\lambda':1)\in Q\subset\bP^3$.
With this description one easily sees that 
\[
 \blow{\bo}{ \cone Q} =\sett{(\lambda,\lambda',x)\in \bP^1\times\bP^1\times \bC^4}{x\in  \bo  p_{\lambda,\lambda'}}\to \cone Q 
 \]
  is an embedded  resolution of  $ \cone  Q$ with exceptional  divisor $E= Q\times \bo \subset \bP^3\times\bC^4$. 
See also Atiyah's example in \S~\ref{ssec:atiyah} where a different kind of resolution is given.
\end{exmple}

It is important to realize  that \textbf{\emph{Milnor fibrations only arise for  hypersurfaces}} 
$\set{f=0}$ of  a  (germ of a ) smooth variety $X$ where the  function $f:X \to D$, $D=\sett{z\in \bC}{|z|< \delta}$ has  an isolated critical point at $x$. The 
nearby fibers $f^{-1} t$, $t\not=0$ are smooth and therefore one calls such singularities \textbf{\emph{smoothable}}.
Mumford~\cite{Mumford1} was the first to give an example of a non-smoothable isolated singularity.
See also   \cite{greuel}  for a nice overview. 
\begin{exmple} Start with a smooth  elliptic  curve $C\subset \bP^n$ of degree $\ge 10$.
Suppose also that the embedding of $C$ is projectively normal, e.g. the restriction homomorphism $H^0(\bP^n, \cO(k))\to H^0(C,
\cO(k))$ is surjective for all $k\ge 0$. If
$p: \bC^{n+1}\setminus \set{0}\to \bP^n$ is the defining projection, the closure of $p^{-1} C$ in
$\bC^{n+1}$, the so-called cone on $C$, has an isolated singular point in $0$ which is not smoothable. This is the first and easiest example from H. Pinkham's thesis~ \cite{pinkham}.
\end{exmple}

\section{Monodromy}
For  a  locally trivial fiber bundle over the circle, say $\pi:E\to S^1\subset \bC$,   
 a topological monodromy operator can be defined on  any given fiber of $\pi$, say on $F=\pi^{-1}(1)$.
 This can be done by lifting the    loop $t\mapsto \exp(2\pi\ii t)$ on the base $S^1$
 to a path starting at a given point $e\in F$ and letting $h(e)\in F$ be the endpoint.
  Doing this in a coherent way defines a self-homeomorphism of $F$, the 
  \textbf{\emph{topological monodromy operator}}.
It  induces a linear isomorphism $h_*$ on $H_*( F)$ and on $H^*(F)$, the associated  monodromy-operator.
\index{monodromy}\index{Wang sequence}
An important tool in this regard is the so-called \textbf{\emph{Wang sequence}}
(cf. \cite[p. 67]{milnorbook}):
\begin{equation}
\label{eqn:wang}
\cdots\to H_{j+1}(E) \to H_j(F) \mapright{h_*-\id} H_j(F) \to H_j(E)\to \cdots
\end{equation} 
For the Milnor fibration of an $m$-dimensional singularity one has $E=  S^{2m+1} \setminus\lnk w $, and so the Wang sequence  is useful to calculate the homology of the Milnor fiber.

\subsection*{The case $n=1$}
In the curve case   $\lnk w$ is a true link, so homeomorphic
to a disjoint union of $r$ circles.  Hence $H_0(\lnk w)\simeq  H_1(\lnk w)\simeq \bZ^r$.
By the Alexander duality theorem $H_j(S^{2m+1}\setminus\lnk w)
\simeq H^{2m+1-j}(S^{2m+1},\lnk w)$ (see e.g. \cite[Thm. 3.46]{hatcher}), and then the long  
exact sequence for the pair $(S^{3},\lnk w)$
gives 
\begin{equation}
\label{eqn:homlink1}
H_j(S^3\setminus\lnk w)= \begin{cases}
		\bZ& \text{for } j=0,\\
		\bZ^r  &\text{for } j=1,\\
		 \bZ^{r-1}  & \text{for } j=2,\\
			 0 &\text{for } j\ge 3 .
\end{cases} 
\end{equation}
Since the Milnor fiber only has homology in ranks $0$ and $m=1$, and since $h_*=\id$ on $H_0(\mf w)$, in this case the Wang sequence reduces to
\[
\xymatrix@C=16pt{
0 \ar[r] &  H_2(S^3\setminus\lnk w) \ar[r] \ar@{=}[d] &  H_1(\mf w ) \ar[r] _{h_*-\id}\ar@{=}[d]  & H_1(\mf w )\ar[r] \ar@{=}[d] &  H_1(S^3\setminus\lnk w)
\ar[r] \ar@{=}[d] &    H_0(\mf w)\ar[r] \ar@{=}[d] &   0\\
0 \ar[r] & \bZ^{r-1}  \ar[r] & \bZ^\mu\ar[r] & \bZ^\mu  \ar[r]& \bZ^r  \ar[r]& \bZ\ar[r] & 0
}
\]
From this, one sees:
\begin{lemma} For an isolated plane curve singularity having $r$ branches, one has
$\dim(\ker(h_*-\id ))=\dim(\coker(h_*-\id))=r-1$.
\end{lemma}

\subsection*{The case $m\ge 2$} Here  a similar approach as for $m=1$  together with Poincar\'e duality  gives an isomorphism
$ 
H_{j+1}(S^{2m+1} \setminus\lnk w )\simeq H_j(\lnk w)$.
Hence  the Wang sequence becomes:  
\[
 0\to H_{m  } (\lnk w)\to H_m(\mf{w}) \mapright{h_*-\id} H_m(\mf{w})  \to  H_{m-1} (\lnk w) \to 0,
\]
and  one deduces: 
\begin{prop} \label{prop:homlink}  
The monodromy of the Milnor fiber of $w$ relates to the  homology  of the link   as follows:
\[
 H_j(\lnk w)= \begin{cases}   	\ker(h_*-\id)  & \text{ for }  j= m \\
						 \coker(h_*-\id)  & \text{ for }  j=  m-1.
                            \end{cases}
                            \]
  \end{prop}
\begin{rmk} \label{rmk:homlink}
\textbf{1.} Since $H_m(\mf w)$ has no torsion, also   $H_m(\lnk w)$ is without torsion.
\\
\textbf{2.} Since $\lnk w$ is  $(m-2)$-connected, for $m\ge 2$ it is connected, and for $ m\ge 3$ simply connected. In that case,
by the Hurewicz theorem (\cite[Thm. 4.32]{hatcher}), $\widetilde H_j(\lnk w)=0$ for $0 \le  j\le m-2$
and hence also for $m+2\le j \le 2m-2$. So the only interesting homology then is in
the "middle" ranks $m-1,m$.
\end{rmk}

\section{Singularities of plane curves}\label{sec:plane}

\begin{small}    The topology of the Milnor fiber and of  the link of isolated curve singularities
has been widely studied by L.\ Neuwirth \cite{neuwirth} and J.\ Stallings~\cite{stallings}.  
  For a treatment of plane curve singularities and their topology using Puiseux-expansions, see e.g. the book \cite{planecurves} of E.\ Brieskorn and H.  Kn\"orrer.\index{singularity!plane curve ---}
 \end{small}

\medskip
A curve singularity $(W,\mbold 0)\subset (\bC^3,\mbold 0)$  given by  one equation $f(x,y)=0$  gives a normal singularity if and only if  $f$ is irreducible
 in the local ring $\cO_{\mbold 0}(W)$  of holomorphic functions at $\mbold 0$.
 For example, the double points $A_{2n-1}$, given by $y^2- x^{2n}= (y-x^n)(y+x^n)$ are non-normal while the cuspidal points $A_{2n}$ are normal.
 The resolution of $A_{2n-1}$ is given by the disjoint union of the two branches $y-x^n=0$ and $y+x^n=0$. The resolution of the cusps are  irreducible
 smooth curves, e.g. $x^2-y^3=0$ becomes smooth after  blowing up the origin.
 
 \medskip 
 The main results concerning \emph{\textbf{irreducible}} curve singularities can be summarized as follows:
 \begin{prop} Let $(W,\mbold 0)$  be  an irreducible   local curve singularity. Then
 \begin{enumerate}
\item its link  is a knot $\mathsf k$ embedded in $S^3$;
\item the commutator of $\pi_1(S^3\setminus \mathsf k)$ is a finitely generated free group of rank $\mu$, the Milnor number of the singularity;
\item $\mu$ is even and the Milnor fiber of the singularity is a once-punctured orientable surface of genus $\mu/2$.
\end{enumerate}
  \end{prop}

 \begin{exmples} \textbf{1.} Coming back to Example~\ref{exm:MfandLnk}.1, the cusp singularity, we see that indeed $\mu=2$ is even and the Milnor fiber
 which is a torus minus a $2$-disc  is homeomorphic to a once-punctured torus.
 The higher cusp singularities $y^2-x^{2g+1}$ have Milnor number $2g$ which gives a once-punctured oriented surface of genus $g$.
 \\
 \textbf{ 2.} The curve $0=x^p-y^{pq}=\prod_{\omega,\omega^p=-1}(x- \omega y^q)$ having $p$ branches has as its link  $p$ unknotted circles in a torus.
 See \cite[p. 82]{milnorbook}.
 \end{exmples}

\section{Surface singularities}  
\label{sec:surfsings}
Suppose one has an isolated  surface singularity $(W,x)$. \index{singularity!surface ---}
D. Mumford \cite{mumford} has  shown that being singular at $x$ is   a  purely topological property:

\begin{thm} \label{MumfThm} \label{thm:Mumford} \textbf{1.}  Suppose  that   $x$ is a normal singularity,
then the link of $x $ is simply connected if and only if $x$ is a smooth point of $W$. 
\\
\textbf{2.}  If    a neighborhood   of $x$ in $W$  is homeomorphic to
an open  $4$-ball, then $(W,x)$ is smooth.
\end{thm}  
 
Since a topological threefold is simply connected if and only if it is homeomorphic to the $3$-sphere,
this  implies  that the link of $x$  is  homeomorphic to the  $3$-sphere if and only if $x$ is smooth. 
\index{Mumford's theorem!topologically smooth=smooth for surfaces}
In view of the now proven Poincar\'e conjecture~\cite{perel1,perel2,perel3}  stated in  1904 by H. Poincar\'e~\cite{PConj}, this even holds
if one replaces "topological" with "differentiable".\index{Poincar\'e conjecture}

I shall give an outline of Mumford's proof which is based on three  properties for surfaces:
 \begin{enumerate}[\bf 1.]
\item For a normal surface singularity $(W,x)$ there is a unique resolution $\pi: (\widetilde W, E)\to   (W,x)$ 
with the property that $\widetilde W$ is minimal in the sense that $E$ does not contain   a  smooth rational component  of self-intersection $-1$.
Such a  resolution is called the minimal resolution of $(W,x)$.
\item If $E\subset W'$ is a divisor in a smooth surface $W'$, then after blowing up $W'$ one may assume that all components of $W$ are smooth
and two components are either disjoint or meet transversally.
\item For any resolution of  $(W,x)$ the  exceptional divisor  is a connected set of smooth curves whose intersection matrix is negative definite.
\end{enumerate}

The  components of  $E$ can be represented by their "dual  graph" $\Gamma_E$ whose vertices are the components of $E$, and an edge connects two vertices if
and only if the components intersect.
 The idea is now to consider the link in a resolution $W'$ obtained from the minimal resolution by further blowing up so that \textbf{2}  holds.
The new exceptional divisor,  which for simplicity is still denoted  $E$,  admits a tubular neighborhood $\mathsf T(E)\subset W'$ whose boundary maps differentiably to the link   of the surface singularity. So the link can be identified with  $\partial \mathsf T(E)$. The advantage is that $\mathsf T(E)$ has $E$ as a deformation retract since it is a circle bundle over $E$.
If $\phi: \mathsf T(E)\to E$ is the retraction, there is an induced surjective homomorphism $\phi_*:\pi_1(\partial \mathsf T(E)) \to \pi_1(E)$.
So, if $ \partial \mathsf T(E)$ is simply connected, all components of $E$ must be rational curves and there are no loops in the dual graph   $\Gamma_E$.
 Assuming that $\Gamma_E$ is   a (non-empty) simple tree  (every $E_i$ intersects at most two other  $E_j$), property \textbf{3}  can be shown to  imply that $\pi_1(E)$ is 
 a non-trivial cyclic group and so this is excluded. If $\Gamma_E$ is not a simple tree, the argument is more complicated and 
 Mumford uses a group theoretical property
 as well as an analysis of the blowing-up process just used (leading  to  the exceptional divisor  $E$).

\medskip
One cannot weaken the hypothesis to $b_1( \partial \mathsf T(E))=0$, as    shown by  the following example:
\begin{exmple}Take $x$  to be the origin of the hypersurface $W$ with equation  $x_3^r=x_1^p+x_2^q $ where $p, q$, and $r$  are pairwise relatively prime. 
Note that the projection $(x_1,x_2, x_3)\mapsto (x_1,x_2)/\sqrt{ | x_1|^2+|x_2|^2}$ induces a map from $S^5\setminus \sett {(0,0, x_3)}{|x_3|=1}$ to $S^3$
and since the line $x=y=0$ meets $W$ only in the origin, the link $\lnk W$  can be projected to the $3$-sphere $ S^3=\set{| x |^2 + | y |^2 = 1}$ in $\bC^2$. 
This exhibits the link  as  an $r$-fold cyclic covering of  $S^3$  branched along the  torus knot $x_1^p+x_2^q=0$, By 
results of H.\ Seifert \cite[p. 222]{seifert},  $H_1(\lnk W) = 0$.  The fundamental group $\pi_1(\lnk W)$ must be non-trivial by Mumford's result, since the origin is singular.
So it  is a non-trivial   perfect group, that is,  its abelianization $H_1(\lnk W)$ is trivial. 

The case $(p,q,r)=(2,3,5)$ is special, since this gives a quotient singularity $x$ obtained by letting the dihedral icosahedral group $\widetilde \bI$ act on $\bC^2$.
The action restricts to $S^3\subset \bC^2$  whose  quotient under the action gives the link of $x$. 
 Recalling  that an $n$-dimensional  topological manifold    is a \textbf{\emph{homology $n$-sphere}} if it has the same homology as  $S^n$, the just constructed  link is called the \emph{\textbf{Poincar\'e homology $3$-sphere}}. See  \cite[p. 65]{milnorbook} for more details.\index{homology sphere}\index{sphere!homology ---} 

\end{exmple}

\section{\ihs\ in dimensions $\ge 3$}

By Theorem~\ref{thm:milnor}.4, a \index{sphere!topological ---}
 link $\lnk w$  is simply connected and has dimension $2m-1\ge 5$ and so, if $\lnk w$ is a homology sphere, it is homeomorphic to a sphere  by the generalized Poincar\'e conjecture,  which  for  these dimensions is a classic result due to  S.\ Smale and J.\ Stallings.   

Assuming that $m\ge 2$, J. Milnor has found a criterion to determine whether $\lnk w$ is a topological sphere using the   monodromy-operator:

\begin{thm}[\protect{\cite[Thm 8.5]{milnorbook}}] Assume $m\ge 3$. Then $\lnk w$ is a topological sphere  if and only if   $\det( \id -h_*)=\pm 1$.
\label{thm:milnor2}
\end{thm}
\begin{proof} By the preceding remarks, it suffices to show that $H_j(\lnk w)=0$ for $j\not=0,2m-1$.
By Remark~\ref{rmk:homlink}.2  and Proposition~\ref{prop:homlink} this follows as soon as $h_*-\id$ is invertible which is the case if and only $\det(h_*-\id)=\pm 1$.
\end{proof}
  
 \subsection*{The Brieskorn--Pham polynomials}
 There is an important class of examples for which one can compute the characteristic polynomial of $h_*$
 quite easily, namely the diagonal polynomial singularities,\index{singularity!Brieskorn-Pham ---}
 \index{Brieskorn-Pham singularity}
 also called \textbf{\emph{Brieskorn--Pham singularities}}:
 
 \begin{equation}
  z_1^{a_1}+\cdots+ z_{m+1}^{a_{m+1}}=0, \quad a_1,\dots, a_{m+1}\ge 2. \label{eqn:BPh}
 \end{equation} 

The result in this case, due to E.\ Brieskorn~\cite{brieskorn}  and P. Pham~\cite{pham} (see also \cite[\S 9]{milnorbook})  is as follows:

 \begin{thm} \label{thm:BPh} For the  singularity~\eqref{eqn:BPh}  the characteristic polynomial of the monodromy operator $h_*$ 
 has $\mu=(a_1-1)(a_2-1)\cdots (a_{m+1}-1)$ characteristic  roots of the form $\omega_1\omega_2 \cdots\omega_{m+1}$, where $\omega_j$
 is any $a_j$-th root of unity other than $1$.
  \end{thm}

 \begin{exmple} 
 \label{exm:trefoil}
 The generalized trefoil knot $f=0$, where \index{knot!generalized trefoil --}
 $f=x_1^2+\cdots+x_m^2+x_{m+1}^3=0$, $m\ge 3$. Here the relevant roots    are 
 $ (-1)^m \exp(2\pi \ii /3)$ and $(-1)^m\exp(4\pi \ii/3)$.
The characteristic polynomial thus is $t^2-t+1$ for $m$ odd, and $t^2+t+1$ for $m$ even. One deduces from Theorem~\ref{thm:milnor2} that 
 for all odd $m$ the  
 link is a topological $(2m-1)$-sphere.
 For $m=3$ the link is also diffeomorphic to $S^5$, but for $m =5$ one gets an exotic sphere.
  \end{exmple}

 \subsection*{Classifying cDV-singularities in dimension $3$}
 In this survey germs of isolated hypersurface singularities in  a fixed complex space  are  called isomorphic
 if  there is a local biholomorphic coordinate change under which the singularities correspond.
 In general it is quite difficult to  obtain such a classification. For three-dimensional  cDV-singularities
 there are some partial results. In particular, by \cite[Prop. 1.3]{MinDisc}, any weighted homogenous 
 \ihs\ of  $cA_n$-type is isomorphic to  one of three classes  of 
 singularities of  invertible polynomial type, whose corresponding matrix is given by\index{c@$cA_n$-type singularity}
 \index{singularity!of $cA_n$-type}
\begin{align}
 \begin{pmatrix}
 2& 0 &0 &0 \\
 0& 2& 0& 0\\
 0& 0& *& 1\\
 0&0&0&   n
 \end{pmatrix}   \qquad  \begin{pmatrix}
  2& 0 &0 &0 \\
 0& 2& 0& 0\\
 0& 0& *& 1\\
 0&0&1&  n
 \end{pmatrix}     \qquad   \begin{pmatrix}
  2& 0 &0 &0 \\
 0& 2& 0& 0\\
 0& 0& *&  0\\
 0&0& 0&   n
 \end{pmatrix}  \label{eqn:cA3}
 \\
  \text{chain type} \hspace{5em}  \text{loop  type} \hspace{3em} \text{Fermat  type} \nonumber .
\end{align} 
These are investigated in the preprint~\cite{APZ} which resulted from the seminar for which the  notes are elaborated in the  present paper.

 \subsection*{Links of \ihs\  of dimension $3$}
 \index{link!of 3-dimensional singularity}
 
 In this case the link is a  simply  connected compact $5$-dimensional oriented manifold. Moreover,  it
 is the boundary of the Milnor fiber which is  $2$-connected.
 Such $5$-dimensional manifolds have been classified 
  by S.\ Smale:
 
\begin{thm}[\protect{ \cite[Thm. 2.1]{smale}}]  \index{manifolds of dimension 5 (Smale's theorem)}
\label{thm:smale}
Let $M$ be a simply connected oriented compact $5$-manifold which is the boundary of a $2$-connected manifold.
Then $H_2(M)=F\oplus (T\oplus T)$, where $F$ is free and $ T$  is torsion. Moreover $M$ is  homeomorphic to 
\begin{enumerate}
\item $ S^5$ if $H_2(M)=0 $;
\item the connected sum $M_F$ of $b_2(M)$ copies of $S^3\times S^2$ if $H_2(M)\not=0$ and is free;
\item the connected sum $M_F \# M_T$ if $H_2(M)$ has torsion. $M_T=M_{d_1}\#\cdots \# M_{d_k}$ is uniquely 
determined by  the elementary divisors 
$d_1|d_2|\cdots|d_k$ of the torsion group $T$. \footnote{To avoid misunderstanding the notation, $H_2(M_T)= T\oplus T$.  For example,   for all $q$, one has 
$H_2(M_q)=  \bZ/ q\bZ \oplus   \bZ/ q\bZ $.}
\end{enumerate}
This result holds in particular for links of $3$-dimensional \ihs s.
\end{thm}
This theorem shows for example  that
the generalized trefoil knot for $m=3$ as in Example~\ref{exm:trefoil} is indeed diffeomorphic to $S^5$.

\begin{exmple} \label{ex:gendpeven} The generalized double point links for $w=z_1^2+z_2^2+z_3^2+z_4^{2k}=0$ have Milnor number $\mu=2k-1$ and so the characteristic roots for $h_*$ are
$ -\rho, \dots ,-\rho^{2k-1}$, where $\rho$ is a primitive $2k$-th root of unity. Since $-\rho^k=1$, $h_*- \id$ has $1$-dimensional kernel and so 
$b_2(\lnk w)= b_3(\lnk w)=1$.        To determine the possible torsion in $H_2$,  one needs an integral representation $\mathsf T$
for  the monodromy operator $h_*$ . This can be done as explained in \cite[p. 94--95]{dimca} resulting in a $(2k-1)\times (2k-1)$ matrix of the shape
\[ 
\mathsf T= \begin{pmatrix}
 0  &0 & \dots & 0& 1\\
-1&0  &\dots & 0& 1\\
0 & -1 & \dots & 0& 1\\
\vdots &\ddots &\ddots&  &\vdots\\
0& \dots &0 &-1 & 1
 \end{pmatrix}  .
 \]
Then $\mathsf T-I $ can be reduced by integral elementary row-operations into  the diagonal matrix $ {diag}(1,1,\dots,1,0)$ which again shows that $b_2(\lnk w)=1$,
but even more: there  is no torsion in $H_1(\lnk w)$. Applying Smale's result, one deduces that $\lnk w$ is diffeomorphic to $S^2\times S^3$, which is independent
of $k$. Replacing $2k$ with $2k+1$, a similar argument shows that  $\lnk w$  is diffeomorphic to the $5$-sphere.
In other words: \textbf{\emph{the link does not determine the singularity}}.         
\end{exmple}

 \chapter[On compound du Val singularities]{On compound du Val singularities}
      \label{lect:cDV}

 {\small \noindent 
 In this chapter $(X,x)$ is a germ of a  complex-algebraic   variety $X$ with an isolated  singularity at 
$x$, but not necessarily an \ihs.
}

\section*{Introduction}

The ambiance for this chapter  has changed to complex algebraic geometry. The following topics will be briefly treated:
\begin{itemize}
\item the canonical divisor of a singular variety;
\item   discrepancies of a resolution;
\item small resolutions of $3$-dimensional \ihs\, 
and how to construct these for cDV singularities according to Brieskorn, Pinkham et al.,
\item   the purely algebraic concept of the local class group in relation to the link and   to     small resolutions.
\end{itemize}

\section{The canonical divisor}
\label{sec:onsings}

 A  \textbf{\emph{Weil  divisor}} on   $X$ is a finite formal sum $\sum n_iD_i$, $n_i\in\bZ$, where the 
 $D_i$ are codimension $1$ subvarieties of $X$. A   meromorphic function $f$ defines the divisor $(f)= 
(f)_0-(f)_\infty$, \index{divisor!Weil ---}\index{divisor!Cartier ---}
where $(f)_0$, respectively 
 $(f)_\infty$  is the divisor of zeroes, respectively poles of $f$.  Such divisors are the \textbf{\emph{principal 
divisors}}.
 
 The set of Weil divisors form an abelian group.  The principal divisors form a subgroup therein.
 A  \textbf{\emph{Cartier divisor}} is a global section of the quotient sheaf $\cM^\times_X/\cO^\times_X$, where 
$\cM_X$  is the sheaf  of   meromorphic functions on $X$.   Alternatively, a Cartier divisor is given by a collection  $\set{U_j,f_j}$ of 
non-zero   meromorphic  functions $f_j$ on $U_j$, where $\set{U_j}$ is an   open  cover of $X$ such that  the functions  $f_j$ 
 and $f_k$  in $U_j\cap U_k$  coincide up to multiplication
 with a non-zero  holomorphic function. The Cartier divisors on $X$ form a  multiplicative group in an obvious way. 
 
 On a smooth variety $X$ there is no difference between Cartier and Weil divisors.
 Since a codimension $1$ subvariety on a singular variety $X$ need not be the zero-locus of a function
 (think of a line on  a cone), a Weil divisor need not be a Cartier  divisor. 
 However, the second description
 of a Cartier divisor shows that the divisors   $\set{f_i=0}$ on $U_i$
 glue to give a Weil divisor on $X$. So a Cartier divisor determines a Weil divisor.
\par
Another central concept is the so-called \textbf{\emph{canonical sheaf}} $\omega_X$ of  a variety $X$ having 
normal  singularities, and its associated Weil divisor $K_X$, the \textbf{\emph{canonical divisor }}of $X$. To define these, recall that\index{canonical!sheaf},\index{canonical!divisor} 
on an  $n$-dimensional smooth variety $U$ the canonical sheaf $\omega_U$ is associated to the canonical bundle $K_U=\Lambda^n T^*_U $.\label{page:CanSheaf}
 The sections are the  regular, or -- in the analytic  setting --   holomorphic  $n$-forms. 
 If one allows poles, one speaks of rational, respectively   meromorphic  $n$-forms.
 With  $X^0\subset X$ the open subvariety  of smooth points in $X$,  one  defines the canonical sheaf $\omega_X$    as the 
sheaf associated to presheaf given by 
\[
U\mapsto \set{   \text{rational $n$ forms $s$ on } U,  \text{  regular on } U\cap X^0} .
\] 

It is instructive  and useful to know how the canonical divisor of a smooth variety $X$ behaves under the simplest bimeromorphic map, the  blow up  
$\sigma: Y=\blow XZ \to X$  in a smooth subvariety $Z$, as   defined in Example~\ref{exm:BlSm}.
In terms of the  exceptional divisor $E$ 
the canonical divisors  of $X$ and $Y$ are related by the formula \label{page:CanDiv}
\begin{equation}
\label{eqn:CanBndlBlow}
K_Y =\sigma^* K_X+  (n-m-1)E.
\end{equation} 
This can be shown by a local calculation. See e.g. \cite[p. 608]{pagg}.

\section{Discrepancies}
\label{sec:Discreps} 
Let me now proceed to the behavior of the canonical divisors on singular varieties under desingularization.
Here one makes use of the following basic concepts:

\begin{dfn} \label{dfn:can}\index{singularity!canonical ---}\index{singularity!terminal ---}\index{singularity!index of ---}
\index{canonical!singularity}\index{terminal!singularity}
A germ $(X,x)$  of a normal algebraic variety is a \emph{\textbf{canonical}} (resp. \emph{\textbf{terminal}}) 
singularity  if the following two conditions hold simultaneously:
\begin{enumerate} 
\item for some integer $r\geq 1$ the Weil divisor $rK_X$ is Cartier; the smallest such $r$ is called the \emph{\textbf{ 
index}} of $X$;
\item  for any resolution $\sigma: Y \to X$ with    exceptional divisor  $\sum_i E_i$ (which may be zero) one has
\[
rK_Y = \sigma^*(rK_X)+\sum a_i E_i
\]
with all $a_i\geq 0$ (resp. $>0$); the $a_i$ are called the \emph{\textbf{discrepancies}}.\index{discrepancy}
\\
$\displaystyle  r^{-1} \min_i \set  {a_i} $ is the  \textbf{\emph{minimal discrepancy}} for $\sigma$.
\end{enumerate}
If only (1) holds, the singularity is called \textbf{\emph{$\bQ$-Gorenstein}} and then 
$rK_Y = \sigma^*(rK_X)+\sum  a_i  E_i$ where some of the $a_i$ are possibly negative. 
If $r=1$ one has a \textbf{\emph{Gorenstein singularity}}. 
Notice that $\sigma^*(rK_X)\cdot C=0$
for any curve $C$  in the exceptional set. Hence $(rK_Y -\sum a_i E_i)\cdot C=0$ for such curves $C$.
A singularity which satisfies this property is called \textbf{\emph{numerically Gorenstein}}.\index{singularity!Gorenstein ---}
\end{dfn}

A resolution $\sigma:Y\to X$ is called \emph{\textbf{crepant}} \index{resolution!crepant ---}if all its discrepancies vanish, i.e.,  
$\sigma^*K_X=K_Y$,  as  in  Example~\ref{ex:hyps} below  when  $k=2$, or  for a
  small resolution (see Definition~\ref{dfn:Sings}.\textbf{4.}

\begin{exmple} \label{ex:surfsing} For an A-D-E surface singularity $(X,x)$ there exists  a resolution $\sigma: Y\to X$ with $K_Y=\sigma^*K_X$.
For a proof see e.g. \cite{durfee2}. So these singularities are canonical. 
The converse lies deeper. See for example \cite[(4.9) (3)]{reidSing}. Note that a smooth point can also be called a canonical singularity.
If one blows up once, the minimal discrepance becomes $1$, and this is upper bound for discrepancies in the surface case.
\end{exmple}

The following example shows that discrepancies can have any sign.
\begin{exmple} \label{ex:hyps}  Consider a hypersurface $X=\set{f(x,y,z)=0}$ in $\bC^{3}$ with an ordinary $k$-fold 
point
at the origin. Let $U$ be one affine chart of the blow up of $\bC^3$ with coordinates $(u,v,w)$  where the blow 
up $\sigma: U\to \bC^3$ is given by $w=z, x=uw, y= vw$. Then 
\[
\sigma^* f =f(uw,vw,w)= w^k \cdot g(u,v,w),
\]
where $g=0$ is the equation of $Y$ in $U$. Here $w=0$ is the equation of the exceptional divisor in $U$.
The canonical differential on $X$ is given by 
\[
\omega_X = \text{Res} _X \frac {dx\wedge dy \wedge dz}{f}= \frac{dy\wedge  dz} { f_x }  .
\]
Note that $g_u = w^{-k+1} \cdot ( f_x)$.
Now near  a point  $P\in U\cap E$ be a point where $\partial g /\partial u\not=0$, in coordinates $v,w$, write
\begin{align*}
\sigma^*(\omega_X)& = w  \cdot \frac{dv\wedge dw}  { f_x}   = w  \cdot \frac{dv\wedge dw}  { w^{ k-1} \cdot ( g_u)}\\
&= w^{2-k} \cdot  \frac{dv\wedge dw}  {   g_ u }=w^{2-k} \text{Res} _Y \frac {du\wedge dv \wedge dw}{g} \\
& = w^{2-k} \cdot  \omega_Y.
\end{align*} 
So on $Y$ the canonical differential  of $X$ has divisor $(2-k )E$. In terms of divisors, $K_Y= \sigma^* K_X+ (2-k 
)E$, i.e., the  discrepancy equals $2-k$.  It  is $1$ for a smooth point, $0$ for an ordinary double point and $<0$ 
if the  multiplicity  is $\ge 3$.
 \end{exmple}
 
 The next example shows that one can also have fractional discrepancies.
 \begin{exmple}
 Take the quotient $X=\bC^2/\mathbf{\mu}_3$, where $\mathbf{\mu}_3$ is the cyclic group of cube roots of unity 
acting 
linearly
 on $\bC^2$ by sending $(x,y)\in\bC^2$ to $(\rho x,\rho y)$, $\rho\in \mathbf{\mu}_3$.
 By considering the invariant quadrics, one easily sees that $X$ is the affine cone in $\bC^4$ over the twisted cubic 
curve.
 Since $\rho(dx\wedge dy)= \rho^2 (dx\wedge dy$ the $3$-canonical form $(dx\wedge dy)^3$ is invariant.
 Up to a unit this form gives  a generator of $3K_X$ which one sees as follows.
  If $\pi: \bC^2\to X\subset \bC^4$
 is the quotient map, using $u_0,u_1,u_2,u_3$ as coordinates, with $\pi^* u_0=x^3,\pi^* u_1=x^2y, \pi^* 
u_2=xy^2,\pi^* 
u_3= y^3$,
 one finds that 
 \[
 s= \frac {(du_0\wedge du_1)^{\otimes 3}}{u_0^4},\quad \pi^*s= \text{unit}\cdot (dx\wedge dy)^3.
 \]
 This shows that this singularity has index $3$.
 Next, blowing up $\bC^4$ at the origin gives a resolution $\sigma:Y\to X$ of $X$.
Consider the $(z,t)$-chart in $Y$  with $\sigma(z,t)= (z,zt,zt^2,zt^3)=(u_0,u_1,u_2,u_3)\in X$. Then
\[
  \sigma^*s= \frac{(dz\wedge z \cdot dt)^{\otimes 3} }{z^4 }=  \frac{(dz\wedge dt)^{\otimes 3}}{  z},
\]
and so  $3K_Y= 3\sigma^* K_X- E$ as divisors, where $E$ is the exceptional curve. Hence the discrepancy equals 
$- 1/3$ in this case.
 \end{exmple}

 \begin{rmk}
 \label{rmk:OnMinDiscr}
 Surface singularities have a unique minimal 
resolution  and so  it makes sense to define the \textbf{\emph{minimal discrepancy}}  for the singularity as
the minimal discrepancy of such a resolution. In higher dimension in \index{discrepancy!minimal ---}
general  no minimal resolution exists. Moreover, resolutions exist where the exceptional locus is not divisorial, the 
so-called  small resolutions to be discussed  in \S~\ref{sec:smallres}. 
These have to be discarded if one wants to make    sense of  the minimal discrepancy.
 
 Note that for any  resolution of singularities  $Y\to X$,  again  blowing up $Y$ 
 in a  smooth subvariety  of  codimension $c$  
   contained in  an exceptional divisor $E\subset Y$ and with discrepancy $k>0$ creates a new exceptional component $F$ in $Z$ with 
discrepancy    $k+ n-1-c\ge k$ (since $kE$ contributes $kF$ to the new canonical divisor) 
and so the minimal discrepancy does not change. 

 \end{rmk}

 Using this remark, one can compare different resolutions using 
suitable   blowings up and then show  that the  minimal discrepancy  is the same for all  resolutions. See \cite[Ch. 17]{flips}. 
This  then  by definition is the minimal discrepancy of the singularity. This also applies to smooth points $(X,x)$. The preceding discussion shows that
their minimal discrepancy equals $\dim X -1$
 
A  terminal singularity (which is not a smooth point)  turns out to have  minimal discrepancy in the interval $(0,1]$.   See  Section~\ref{sec:Mclean}.

\section{Small resolutions of cDV-singularities}

\label{sec:smallres}

\subsection{More on cDV's}
 
 First recall the  definition.
\begin{dfn}
A $3$-dimensional hypersurface singularity $(X,x)$, $X=\set{f=0}$   is a  \emph{\textbf{compound du Val 
singularity}}  (cDV for short)   if $f$ is analytically equivalent to\index{singularity!compound du Val (cDV) ---}
$
 g(x,y,z)+ t h(x,y,z,t)\in \bC[x,y,z,t]$,  
where $g=0$ is the equation of a du Val (surface) singularity and $h$ is an arbitrary polynomial.
 In other words, a cDV point is a threefold singularity such that  some hyperplane section is a du Val surface 
singularity.
\end{dfn}

In dimension $3$  M.\ Reid  characterized index $1$ cDV's:

\begin{thm}[\protect{\cite[Thm 1.1]{reid1}}]
Isolated terminal threefold singularities of index one are exactly the isolated cDV singularities.
\end{thm}

\begin{rmk} 
 1. A cDV-singularity need not be isolated, for instance $xy= z^2 t $ has as its singularity locus the line 
$x=y=z=0$.
\\
2. 
The \emph{general} hyperplane section of a Du Val singularity  of a cDV given by $g(x,y,z)+th(x,y,z,t)$ may 
have 
a different type of singularity 
than the singularity given by  $g=0$. For example, taking $xy-z^2 t =0$, the hyperplane $z=t$ gives a singularity $xy=z^3$ which is 
equivalent to the $A_2$-singularity
$x^2+y^2+z^3=0$, while setting $t=0$ gives $xy=0$,  an $A_1$-singularity.
\end{rmk}

\begin{rmk}
 It is known (cf.\ \cite[p. 363]{reid1} for a proof in dimension $3$)  that any  canonical  singularity of index one is 
rational 
(cf. Definition~\ref{dfn:Sings}.4). 
So in particular, \textbf{\emph{a cDV-singularity is rational}}.
\end{rmk}

\subsection{Atiyah's example of a small resolution(\cite{Atiyah})} \label{ssec:atiyah}

\index{resolution!small ---}
 In  Example.~\ref{ex:embres}, the resolution of the  threefold   $X\subset  \bC^4$ with equation $x_1x_4=x_2x_3=0$  
 has been performed by blowing up the origin in $\bC^3$, using that $X=\cone Q$, the cone over the quadratic 
 hypersurface $Q\subset \bP^3$.  In inhomogeneous coordinates  the point 
$(\lambda,\lambda')\in \bP^1\times\bP^1$ can be identified with $ p_{\lambda,\lambda'}=(\lambda\lambda': \lambda:\lambda':1)\in Q\subset\bP^3$. It was shown that
$
 \blow{\bo}{ \cone Q} =\sett{(\lambda,\lambda',x)\in \bP^1\times\bP^1\times \bC^4}{x\in  \bo  p_{\lambda,\lambda'}}\to \cone Q 
 $
is a resolution of singularities of $ \cone  Q$ with exceptional 
 divisor $E= Q\times \bo \subset \bP^3\times\bC^4$. Now as in Example~\ref{ex:hyps} one shows that in this case  $K_Y= p^* K_X+E $ and so the singularity is terminal  of index $1$ and has discrepancy $1$.
 The quadric has  two systems of lines
\begin{align*}
 a_\ell  \,  :   \,  x_1=\ell x_2,\quad x_3=\ell x_4,\\
  a_{\ell' } \, :  \,  x_1=\ell' x_3,\quad x_2=\ell' x_4.
\end{align*} 
 The blow up $\blow \bo Q$ admits a projection into $\bP^1\times \bC^4$, where $\bP^1$ is either one of the first two factors above.Their  images are 
\begin{align*}
Y:=&\sett{(\lambda,x)}{x\in  \bo a_\ell} \subset \bP^1\times\bC^4,\\
Y':=&\sett{(\lambda',x)}{x\in  \bo a_{\ell'}}\subset \bP^1\times\bC^4. 
\end{align*} 
Projecting $Y,Y'$ into $\bC^4$ of course gives $ \cone Q$. The projections
to either one of the $\bP^1$
exhibit  $Y$ and $Y'$ as the total space of a plane bundle
with fiber over $\lambda$, respectively $\lambda'$, given by the plane $\bo a_\ell$, respectively $\bo a_{\ell'}$.
In particular $Y$ and $Y'$ are smooth. The fiber over $x\in \bC^4$ of the induced projection
$\pi: Y\to  \cone  Q$ is the pair $(   \bo x, x )$ if $x\not=\bo$ and  $\pi^{-1} \bo=\bP^1\times \bo$. In other words, $\pi$ is a small resolution, and similarly for the projection $\pi': Y'  \to \cone  Q$.

 The three resolutions fit into the commutative diagram
  \begin{equation}
  \label{eqn:atiyah}
  \begin{split}
  \xymatrix@R=15pt@C=15pt{  
& & Q=  \bP^1\times\bP^1  \ar@<-1ex>[ddll]    \ar@<1ex>[ddrr]   \ar@{_{(}->}[d] & &  \\
         &  &      \blow {\bo}{\cone Q}\ar[dd]      \ar[dl]   \ar[dr]   &  &\\
    \bP^1 \ar@{^{(}->}[r] & Y  \ar[dr]_{\pi }  &    &Y' \ar[dl]^{\pi'}  &  \bP^1 \ar@{_{(}->}[l] \\
        &    &              \cone Q.                     &   & 
 }
  \end{split}
\end{equation}
The transition from $Y $ to $Y '$ is a birational map known as a \textbf{\emph{flop}}.
 
 \subsection{Constructions of small resolutions for cDV-singularities}
 \label{ssec:ConstrSmallRes} 
There is a general procedure to construct small resolutions for isolated cDV-singu\-larities:  one starts  from a  smooth
threefold   $X$  fibered as  a family $X_t$ of surfaces over a $t$-parameter disc, 
where $X_t$ is smooth for $t\not=0$  and $X_0$ is an  isolated   ADE-surface singularity. Now this surface singularity
can be resolved.  Replacing $t$ by  a power equips $X$ with a cDV-singularity at $\bo$, but resolving the fiber $X_0$ 
does not in general resolve the threefold singularity. However,
E. Brieskorn \cite{bries} has shown that  this does occur provided one chooses the    power of $t$ suitably,
and then one of course obtains a small resolution of the threefold singularity:

\begin{thm}[\protect{\cite[Satz 2]{bries}}] \label{thm:OnCox} Let $\Delta \subset \bC$ the unit disc with coordinate 
$t$, 
and let $ X\subset  \bC^3 \times \Delta  $ be a smooth $3$-fold 
with equation $f(x,y,z,t)=0$ such that the projection 
$ X\to \Delta$  is surjective and  smooth over $\Delta\setminus \set{0}$. Assume that $X_0$, the fiber over $0$, has 
an isolated $ADE$-type surface singularity.
Then the  singular threefold $f(x,y,z,t^m)=0$  has a cDV-singularity at $\bo$. It admits a small resolution if and only if $m$ is  multiple of the so-called Coxeter  number 
of the surface singularity, given below.
\begin{center}
  \begin{tabular}{@{} |c|c|  @{}}
    \toprule
    Type  & Coxeter number  \\ 
    \midrule
    $A_n$  & $n+1$   \\ 
    $D_n$  &  $2n-2$   \\ 
    $E_6$  & $12$ \\ 
 $E_7$  &  $18$    \\
  $E_8$  & $30$     \\
    \bottomrule
  \end{tabular}
\end{center}
\end{thm}

\begin{exmple} $x^2+y^2+z^{n+1}+ t=0$ is a smooth variety passing through  $\bo =(0,0,0,0)$ but  $x^2+y^2+z^{n+1}+ t^{m(n+1)}=0$
has a cDV-singularity at $\bo$. It  admits a small resolution for all natural numbers $m $. Note that for $n=m=1$ one recaptures
Atiyh's example above. See  also Example~\ref{exm:AtiyahBis} for a more detailed explanation in Brieskorn's set-up.
\end{exmple}

Brieskorn's  construction   uses  a so-called \textbf{\emph{semi-universal unfolding}} of a given ADE-surface singularity. Roughly 
speaking, this is a family from which all \index{singularity!semi-universal unfolding of  ---}
deformations of the singularity can be obtained  by pulling back.
The construction of the semi-universal unfolding is quite simple. Instead of the jacobian ring $\jac f$ of $f=0$, one 
uses 
a monomial basis for the $\bC$-algebra 
$R_f:= \jac f/ (f)$,
the \textbf{\emph{Tjurina algebra}}. It turns out that
  each monomial provides a deformation parameter. 
  For the  ADE-singularities, $f\in \jac f$ and so one can work with the Jacobian ring itself.

\begin{exmple}   For  $A_n$ singularities   $f=x^2+y^2+ z^{n+1}=0 $, $n\ge 1$ the ring $R_f$ 
has as a monomial basis   $\set{1,z,z^2,\dots, z^{n-1}}$.
The semi-universal unfolding is the relative hypersurface $\cX\subset  \bC^3\times \bC^n$ over $\bC^n$ given by 
the 
equation
\[
f_{\bt }(x,y,z)= x^2+y^2+ z^{n+1} +g(z,\bt),\quad g(z,\bt)=\sum_{j=0}^{n-1} t_j z^j ,\quad \bt=(t_0,\dots,t_{n-1}).
\]
Hence the parameter  space is an  $n$-dimensional complex vector space with coordinates $t_0,\dots,t_{n-1}$.
The universal unfolding admits a finite cover defined by the
  factorization  into linear factors of the augmented deformation polynomial:  
 \[
 F(z,\bt):= z^{n}+\sum_{j=0}^{n-1} t_j z^j  = \prod_j (z+a_j  ),\quad \bt=(t_0,\dots,t_{n-1}) .
 \]
Indeed, setting $t_j= \sigma_{j+1}(a_1,\dots,a_n)$, $t=0,\dots,n-1$, where $\sigma_k$ is the $k$-th elementary 
function 
in the $a_j$ defines a ramified cover
\[
\bt: \widetilde B=\bC^n \to B=\bC^n,\quad \ba=(a_1,\dots,a_n)\mapsto \bt(\ba)=(\sigma_1 (\ba),\dots,\sigma_n(\ba)).
\] 
The branch locus $\Delta\subset B$ is the locus where at least two roots coincide and is called the
\textbf{\emph{discriminant locus}}.
Pulling back the universal unfolding $\cX\to B$ to $\widetilde B$  
gives  $  \cX \times_B \widetilde B \to \widetilde B$ described by the  equation 
\[
h(x,y,z,\ba)= f(x,y,z) - F( t_j(\ba))=0  , \quad \ba=(a_1,\dots,a_n).
\]
This family has   singular fibers over   the discriminant locus.
\par 
To pass to threefolds, one gives
 a holomorphic map $\phi: (\Delta,0) \to (B,\mbold 0)$  and  lifts  $\cX$  to $\widetilde B$.
 Concretely, one writes  $\ba(t)= (a_1(t),\dots,a_n(t))$ as a holomorphic map with  $\ba(\mbold 0)=0$,
 and then substitutes  in $h(x,y,z,\ba)=0$.  The new  family gives a 
  threefold $X$ fibered over the unit disc, say $\pi: X\to \Delta$, as  summarized in the commutative diagram
 \[
 \xymatrix{ & \cX\ar[d]\\
 X\ar[d]_\pi \ar[ur]& \widetilde B\ar[d]_{\bt}\\
  \Delta \ar[r]_\phi \ar[ru]_{\ba} & B.
 }
 \]
 In the present situation, one assumes that
 $\phi(\Delta)$  meets the discriminant locus only in $\mbold 0$, intersecting it  transversally. 
 Due to branching, $X $ has an isolated quotient singularity located in  the fiber  $\pi^{-1} (0)$.  This is also a 
singularity 
of  this  fiber. 
  E.~Brieskorn exhibits a   resolution    of $X$  resolving at the same time 
  the singularity of the fiber.  Hence the  exceptional set is contained in  the fiber  over $0$. In other words, this gives 
a small resolution 
  of  $X$.
 \end{exmple}

 \begin{exmple} \label{exm:AtiyahBis}
 The semi-universal unfolding of $x^2+y^2+z^2=0$ is given by  $x^2+y^2+z^2+t=0$ which gives a smooth 
threefold. The augmented deformation polynomial is $z^2+t= (z+a_1)(z+a_2)$ which gives a $(2:1)$-cover branched in the 
locus  $t=0$. Indeed, this is exactly the locus 
 in the $t$-parameter where the fiber is singular. Branching in it gives $x^2+y^2+z^2+t^2=0$, a threefold with a 
singular  point in $(0,0,0,0)$. There are two   resolutions corresponding to the $2$-roots of the polynomial $z^2+t^2=0$. 
These are exactly the  two  small resolutions  described by the diagram~\eqref{eqn:atiyah}.
 \end{exmple}
  
\begin{rmk}
Root systems come up in  the procedure outlined above, since for any $ADE$-singularity the  cover $\widetilde B $ 
can be interpreted as the complex root-space
 of the corresponding root system and   $\widetilde B \to B$   as
 the quotient   under the action of the Weyl group.  Observe for instance that for $A_n$-type
 double points the covering group is  the symmetric group $\germ S_n$ acting as a permutation group on the roots 
of  the extended deformation polynomial $F$ which is indeed  isomorphic  to the Weyl group of the root system $A_n$. 
 Subgroups of the Weyl group give intermediate resolutions of the $A_n$-surface singularity and one can show that 
the total space remains smooth only if the result is again a cDV-threefold singularity of $A$-type.
\end{rmk}

 As  shown in \cite{bries},  any  cDV-singularity   admitting a  small resolution can be gotten 
 from a similar procedure as in the case of an $A_n$-type  cDV.

\medskip  
Using  a general method due to H.\ Pinkham~\cite{pink}, 
S. Katz~\cite{katz} found  a   systematic way to find other  cDV 
singularities of $A_n$-type and of $D_n$-type admitting a small resolution. 
 The statement is easiest to give for the first type:

\begin{thm}[\protect{\cite[Thm. 1.1]{katz}}]    A  cDV-singularity of $A_n$-type given by 
$x^2+y^2+g (t,z)=0 $ admits a small resolution with a chain of $  n$ smooth rational curves
intersecting transversally if and only if $g(t,z)=0$ is a singularity with $n+1$ distinct branches at the origin.
 \label{thm:katz}
 \end{thm}
 
 For the $D_n$-type singularity Katz shows that the  (on  $(n-1)$ parameters depending)  semi-universal unfolding of  the   
 $D_n$  surface-singularity $x^2+y^2z-z^{n-1}=0$   leads to the family of threefold singularities
 \begin{equation}
 \label{eqn:katz1}
 x^2+y^2z  -  [ z^{n-1} +  \sum_{j=0}^{n-2} \phi_j(t)  z^j ]=0,
 \end{equation}   
 which depends on   the $n-1$ analytic  functions $\phi_j(t)$, $j=0,\dots, n-2$, each vanishing at  $0$.
 The associated family of curves  
 \begin{equation*}
 F(z,t)= 0, \quad  F(z,t):=z^n+  z\cdot \left(  \sum_{j=0}^{n-2} \phi_j(t)  z^j \right) +t^2 ,
  \end{equation*} 
 is then used to describe some (but not all) cases where the corresponding cDV-singularity has a small deformation:
 
 \begin{thm}[\protect{\cite[Thm. 1.2]{katz}}]  
 For any choice of germs of analytic functions $\phi_0,\dots,\phi_{n-2}$  
 vanishing at $0$  the cDV-singularity given by \eqref{eqn:katz1}   
  admits a small resolution in case the associated curve $F(z,t)=0$ has $n$ 
smooth branches  each tangent to $z=0$ with multiplicity $2$. 
The resulting exceptional set  consists of $n$ smooth rational curves 
whose  graph is of type $D_n$.
\end{thm}

\section{Local class groups, links   and  small resolutions}

 A  useful algebraic (or analytic) invariant of    an isolated singularity $(X,x)$ is its  local class group:  
 \begin{dfn}
 \begin{enumerate}[\qquad\bf 1.] \index{local class group}
\item The \textbf{\emph{local class group}} \footnote{The stalk at $x$ of the structure sheaf $\cO_X$
is usually denoted $\cO_{X,x} $. 
} $\Cl_x(X):=  \Cl(\cO_{X,x})$ at a point $x\in X$ is  the quotient group of the Weil  divisors modulo the 
Cartier  divisors  of $(X,x)$.   Its  rank  is denoted  by $\rho(x)$. \label{page:lcg}
 \item
 $(X,x)$ is \textbf{\emph{locally  factorial}}, respectively \textbf{\emph{locally $\bQ$-factorial}}, if  the group $\Cl 
_x(X)$ \index{singularity!locally ($\bQ$)-factorial ---}
is zero, respectively  torsion, or, equivalently, if $\rho(x)=0$.
   \end{enumerate}  \label{dfn:clsgrp}
\end{dfn}

The following general result of H.\ Flenner  (\cite[Satz 61]{flenner})   relates the local class group to the link:

\begin{prop} \label{prop:flenner} For an isolated  rational singularity $(X,x)$, the local class group $\Cl_x(X)$ is 
isomorphic 
to $H^2(\lnk {X,x})$.\footnote{Here it is not necessary that the singularity is a hypersurface singularity.}
\end{prop}
 
This can be used in conjunction with the
 following  criterion~\cite[Thm. 5.7]{grassi}  by  A. Grassi et. al.  which treats the case $\Cl_x(X)\otimes\bQ=0$:

\begin{thm} Let $(X,x)$  be  a   rational \ihs\  such that its  link $\lnk  {X,x}$   has finite fundamental group.
Then $\lnk  {X,x}$   is a rational homology sphere  if and only if $(X,x)$  is locally  $\bQ$-factorial. 
\end{thm}

By J.\ Milnor's result \ref{thm:milnor}.\textbf{4}, for $\dim X=m\ge 3$ the link of an \ihs\  is simply connected. 
Using  Proposition~\ref{prop:homlink}, one deduces:

\begin{corr} Suppose  $(X,x) $ is an isolated rational \ihs\  of dimension $m\ge 3$.
Then the following conditions are equivalent:

\begin{enumerate}[\rm (i)]
\item $(X,x)$ is locally  $\bQ$-factorial.
\item The  link  of $(X,x)$ is homeomorphic to the $(2m-1)$-sphere.
\item $\det(h_*-\id)=\pm 1$, where $h_*$  is  the monodromy operator.
\end{enumerate}
This equivalence holds in particular for isolated cDV-singularities.
\end{corr}

For $3$-dimensional isolated singularities there is a relation with small resolutions (see Definition~\ref{dfn:Sings}.4):   

\begin{thm}[{\cite[Coroll.4.10]{grassi}}]
Let $(X,x)$ be a germ of an isolated   terminal threefold-singularity. If $(X,x)$ is locally analytically $\bQ$-factorial, 
then  $X$ does not admit a small resolution. Conversely, if $X$ is not locally analytically $\bQ$-factorial, then there 
exists  a small   \textbf{partial}
 resolution $Y \to X$ such that $Y$ has at worst $\bQ$-factorial singularities.
\end{thm}

I next discuss topological implications of the existence of a  small resolution culminating in Theorem~\ref{thm:onSmallRes} below.
First some easy observations:
\begin{lemma} Let  $ (Y,E)\to (X,x)$ be a resolution of an \ihs\ in   dimension $\ge 3$ and let $T_E$ be a tubular neighbourhood of $E$. Then
\begin{equation}
H_*(\lnk{X,x})\simeq  H_*(\partial T_E) \simeq H_*(T_E\setminus E). \label{eqn:h1}
\end{equation} 
\end{lemma}

\begin{proof}
Viewing  $X$ as a hypersurface of   the   ball  $B^{2m+2}=\sett{\bz\in\bC^{m+1}}{|\bz\|^2\le \epsilon}$,  one  may 
identify  its boundary $\partial X$ with $S^{2m+1}\cap X=\lnk {X,x}$.  Since $X \setminus \set{x}= Y \setminus E$ and   
$\partial  X= \partial Y$, the link can be considered  as the boundary of  $ Y$.
The tubular  neighborhood  $T_E$ of $E$ is a disc bundle over $E$ and its boundary,  the sphere bundle $\partial T_E$ 
can also be viewed as a  submanifold  of $Y$ and $\partial T_E$  is a homeomorphic copy of $\partial Y$.
On the other hand, $\partial T_E$ is a  deformation  retract of $T_E\setminus E$. 
See for example the discussion in  \cite[\S1]{durfeehain}. In homology this induces the stated isomorphisms.
\end{proof}

If  $m=3$,  one deduces:

\begin{prop} If $(X,x)$ is a   $3$-dimensional \ihs\ admitting  a small resolution $(Y,E)$, $E$ a curve, then 
$H_2(\lnk {X,x})$ is free, of rank  equal to the number of irreducible components   of $E$. \label{prop:onSmallRes}
\end{prop} 
\begin{proof}
Observe that  Lefschetz duality for the  manifold $T_E$ and the compact subset $E$ states that 
\[
H_k(T_E,T_E\setminus E)\mapright{\sim} H^{2m-k}(E),\quad m=\dim T_E=\dim X,\quad k\in\bZ.
\]
In our case $m=3$ and   $\dim_\bC E=1$ so that $H_k(T_E,T_E\setminus E)=0$ for $k=1,2$ and the long exact 
sequence  for the pair $(T_E,T_E\setminus E)$ shows that 
\begin{equation}
 H_2(T_E\setminus E)\simeq H_2(T_E) \simeq H_2(E). \label{eqn:h2}
\end{equation}  
The last isomorphism holds since $E$ is a deformation retract of $T_E$.
If $E$ has  $\ell$ irreducible components, then $H_2(E)\simeq \bZ^\ell$ and so Equations~\eqref{eqn:h1}, 
\eqref{eqn:h2} 
complete  the proof.
\end{proof}

These topological properties are related to algebraic properties of the local class group
via  H.\ Flenner’s result, Proposition~\ref{prop:flenner},  stating that  $H^2(\lnk {X,x})= \Cl_x(X)$. 
Hence, since $H_3(\lnk {X,x})\simeq  H^2(\lnk {X,x})$ is without torsion (cf. Proposition~\ref{prop:homlink}) and has the same rank as 
$H_2(\lnk {X,x})$, one deduces:\index{singularity!rational ---}

\begin{corr}
If $(X,x)$ is an isolated  $3$-dimensional rational singularity with a small resolution whose exceptional set consists 
of $\ell$ irreducible components,
then $  \Cl_x(X)\simeq \bZ^\ell$. 

In particular,  $(X,x)$ is locally factorial if and only if  $(X,x)$ is locally $\bQ$-factorial if and only if $\ell=0$, i.e.,  
$(X,x)$ does not admits a small resolution.  \label{corr:flenner}
\end{corr}

Combining Proposition~\ref{prop:onSmallRes}, Corollary~\ref{corr:flenner}, Theorem~\ref{thm:smale} and 
Proposition~\ref{prop:homlink}, 
one deduces:

\begin{thm} 
If $(X,x)$ is a   rational $3$-dimensional  \ihs\  admitting  a small resolution   whose exceptional set 
consists  of $\ell\ge 1$ irreducible components.
Then
\begin{enumerate} [\rm(i)]
 \item $  H^2 (\lnk{X,x})\simeq H_3(\lnk{X,x})$ is free of rank $\ell$;
 
 \item $  \Cl_x(X)\simeq \bZ^\ell$;

\item  $\lnk{X,x}$ is diffeomorphic to a connected sum of $\ell$ copies of $S^2\times S^3$
 
 \item $1$ has multiplicity $\ell$ as a root of  the characteristic polynomial of the monodromy $h_*$.

\end{enumerate}
\label{thm:onSmallRes}
\end{thm}

\section{Small resolutions and symplectic cohomology}
\label{sec:SmallResSH}

\index{resolution!small --- and symplectic cohomology} 
The definition of symplectic cohomology and its symplectic invariance is postponed to
   \chaptername~\ref{lect:hamreeb&sympcoh}. 
 In particular the Milnor fiber of an isolated  cDV singularity $(X,x)$ having a natural symplectic structure,
  carries the symplectic cohomology  
   $\sh{\bullet}{\mf {X,x},\bC}$ as a symplectic invariant. Here it is important to note that
contrary to ordinary cohomology,   there  might be non-zero  groups in infinitely many negative degrees.
  
Surprisingly, conjecturally   there is a strong   relation between the occurrence of symplectic cohomology in 
these negative degrees 
and the occurrence of small resolutions  as stated as  \cite[Conjecture 1.4]{EvansLekili}:

\begin{conj}
Let $(X,x)$ be an isolated cDV singularity. Then $(X,x)$ admits a small resolution whose exceptional set has $\ell$ irreducible components if and only if 
$\sh{\bullet}{\mf f,\bC}$ has rank $\ell$ in every negative degree.   
\label{conj}
\end{conj}

By Theorem~\ref{thm:onSmallRes} in dimension three
 this conjecture is equivalent to:

\begin{conj} Suppose  $(X,x)$ be an $3$-dimensional isolated cDV singularity.
Then $(X,x)$ admits a small resolution if and only if     
\[
\rank( \sh {-k} {\mf{X,x}}  ) = b_2(\lnk  {X,x} )= b_3( \lnk {X,x}) =\rho(x), \text{  for  all } k>0 .
\]
 In particular, if $\sh {-k}{\mf {X,x}}=0$ for some  $k>0$, the conjecture implies that   $(X,x)$ admits  no small resolution.
\end{conj}

In \cite{EvansLekili}, Conjecture \ref{conj}  has been  verified for the following cDV singularities:
\begin{enumerate} [(a)]
\item $x^2+y^2+z^{n+1}+t^{k(n+1)}=0 ,\quad k,n\ge 1$, \hfill $cA_n$
\item $x^2+y^2+zt (z^{n-1}+t^{k(n-1)})=0 ,\quad k,n\ge 1$,  \hfill $cA_n$
\item $x^2+y^3+z^3 +t^{6k} =0$ \hfill $cD_4$
\item  $x^2+y^3+z^4+t^{12k}=0$ \hfill $cE_6$
\item  $x^2+y^3+z^5+t^{30k}=0$ \hfill $cE_8$
\end{enumerate}
Observe that apart from case (b), the existence of small resolutions follows from \break E.~Brieskorn's result~\ref{thm:OnCox}.
For case  (b), note that the hyperplane $z=at$ gives indeed an $A_n$-singularity and that the curve $zt (z^{n-1}+t^{k(n-1)})=0$
has $n+1$ distinct branches so that there exists a small resolution by Theorem \ref{thm:katz}.

As an outcome of the seminar on which the present notes are based, the conjecture also has been proved for all 
cDV singularities of $A$-type, i.e. those enumerated in \eqref{eqn:cA3}. See~\cite{APZ}.
 
Among the new results, I want to mention the following two which concern contact structures on $S^5$ and on connected sums of $S^2\times S^3$:\index{contact structure!on $S^5$} \index{contact structure!on $\#_\ell S^2\times S^3$}
 
\begin{thm}[\protect{\cite[Theorem E]{APZ}}]
Two invertible $cA_n$ singularities  in standard form \eqref{eqn:cA3} \index{c@$cA_n$-type singularity}\index{singularity!of $cA_n$-type}
have contactomorphic links if and only they are  deformation equivalent.
In particular, their Milnor numbers are the same. In case one of them admits a small resolution, then so does the other and both
have the same number of exceptional curves.
\end{thm}
 
By \cite{kata} the link  of a Fermat-type polynomial $x_1^2+x_2^2+x_3^m+x_4^n$ is diffeomorphic\index{singularity!of Fermat type} 
to $\#_\ell S^2\times S^3$ if  $\ell:= \gcd(m,n)-1\ge 1$ and diffeomorphic to $S^5$ if $\ell=0$. In case $\ell\ge 1$ this confirms
Theorem~\ref{thm:katz} together with Theorem~\ref{thm:onSmallRes} since then  this singularity 
admits a small resolution with $\ell$ exceptional curves. Note that the Milnor number
of such a singularity equals $(n-1)(m-1)$.
 
\begin{thm}[\protect{\cite[Theorem F]{APZ}}]Two Fermat type singularities (of the above type) 
define the same contact structure on $\#_\ell S^2\times S^3$  
if and only if both admit a small resolution and both have the same Milnor number. If $\ell=0$ the resulting contact structures
on $S^5$ are the same if and only if  the Milnor number is the same. In particular, this gives infinitely many contact structures on $S^5$.
\end{thm}

 \chapter[Basics of  symplectic and contact geometry]{Basics of symplectic and contact geometry}
  	\label{lect:symp&contact}

\section*{Introduction}

 In this chapter    some central  notions  in symplectic and contact geometry are discussed:
 \begin{itemize}
\item Liouville fields,
\item contact manifolds, their symplectic completions and Liouville domains,
\item Reeb vector fields and the linearized return map,
\item   symplectic fillings of isolated singularities.
\end{itemize}

\section{More on symplectic geometry}
\label{sec:sg}

\subsection{Basic notions}

Recall from Section~\ref{sec:introsymgeom} of \chaptername ~\ref{lect:overview} that a \index{symplectic!manifold}
\textbf{\emph{symplectic manifold}} $N$ is an  even-dimensional smooth manifold equipped
with   a closed non-degenerate real $2$-form $\omega$, the symplectic form.
The non-degeneracy of $\omega$ means that the natural map\footnote{As usual, $\iota_X$ denotes contraction against the vectorfield $X$.}\label{page:contract}\label{page:tbndl}
\begin{equation}
\label{eqn:iso}
  \phi_\omega  : T_N  \to T_N^* ,  \quad  X\mapsto i_X \omega 
 \end{equation} 
is an isomorphism. 
This observation implies that any smooth function $H: N \to \bR$ defines a so-called \textbf{\emph{Hamiltonian vector field}} $X_H$ on $N$
determined by \index{Hamiltonian!vector field}
\[
\iota_{X_H}\omega = -dH   \iff \omega(X_H, - )= -dH(-)  .
\]
Using that  $\omega$  is closed, this allows to define a Lie-algebra structure on smooth functions on $N$, 
given by the \textbf{\emph{Poisson bracket}}:
\[
\set{F,G}:= \omega(X_F,X_G)= dF(X_H).
\]
See e.g. \cite[Exercise 3.5]{SYmpTop} for a proof of the Jacobi identity.
\par

\begin{exmple} \label{exm:HamField} Identify $\bC^n$ with complex coordinates $z_j=x_j+\ii y_j$ with $\bR^{2n}$ 
with real  coordinates $(x_1,\dots,x_n,y_1,\dots,y_n)$. The    symplectic form   given by 
$ d( \sum x_j dy_j - y_jdx_j) = 2 \sum dx_j\wedge dy_j$ associates to the function $H= \| z \|^2=\sum x_j^2 +y_j^2$ the Hamiltonian 
field $\displaystyle X_H=   \sum_{j=1}^n   y_j \dd {} {x_j} - x_j\dd{} {y_j}$.
If one identifies tangent vectors  $ \displaystyle  \sum (p_j \dd {} {x_j} +  q_j\dd{} {y_j})$ 
on $\bR^{2n}$ at a point $\bp=(p_1,\dots,p_n,q_1,\dots,q_n)$ 
with  the corresponding points of $\bC^n$,
this can also be written as $X_H(\bp) =   J (\bp)$, where $J (p_1,\dots,p_n,q_1,\dots,q_n)= (-q_1,\dots,-q_n, p_1,\dots,p_n)$
is coming from the usual complex structure on $\bC^n$ identified as above with $\bR^{2n}$.
If one uses instead any function of $\| z\|^2$, say $ h(\|z^2\|$, one sees that 
\[
  X_{h }  (\bp)  =   h'( \|\bp\|^2) \cdot J(\bp).
\]
\end{exmple}

Note that $dH(X_H)= i_{X_H}\omega(X_H)= \omega(X_H,X_H)=0$ and so the vector field $X_H$ is tangent to the level sets $\set{H=\text{constant}}$.  
The vector field  $X_H$  generates the \textbf{\emph{Hamiltonian flow}}, a $1$-parameter group $\varphi^t_H$ of diffeomorphisms of $N$ determined by\index{Hamiltonian!flow}
\[
\frac d {dt} \varphi^t_H = X_H\comp \varphi^t_H,\quad \varphi^0_H=\id,\quad t\in(-\epsilon,\epsilon).
\]
On compact $N$ this flow is complete, that is, it exists for all "time" $t$.
 Moreover, one has:

\begin{lemma} \begin{enumerate}
\item The diffeomorphisms $\varphi^t_H$ are symplectomorphisms;
\item For every symplectomorphism $\psi$ of $(N,\omega)$, the Hamiltonian vector field of $H\comp \psi$ is the pull back $\psi^* X_H$ of the Hamiltonian vector field for $H$;
\item  The Lie bracket preserves Hamiltonian vector fields: $[X_F, X_G]=X_{\set{F,G}}$.
\end{enumerate}

\end{lemma}

\subsection{Liouville fields} \label{sec:liouville}

Assume that $(N,\omega) $ is a  symplectic manifold equipped with  a
\textbf{\emph{Liouville field}}, i.e. a  vector field $Y$ on  $N$ which  preserves  $\omega$
  in the sense that $\cL_Y \omega=\omega$, where $\cL_Y$ is the Lie derivative.\index{Liouville!field}
  It then follows that\label{page:Lie} 
  \[
  \omega= \cL_Y \omega \overset{Cartan's}{\underset{formula}{=}} d(i_Y \omega)+ i_Y(d\omega)= d(i_Y \omega),  
  \]
 since $\omega$ is closed. Hence  $\omega$ is exact. This shows that  the existence of a Liouville field is a strong property.

\begin{exmple} \label{exm:affinehyps}  \textbf{1.} Consider an \textbf{\emph{affine hypersurface}}  $V\subset \bC^{n+1}$. The metric 
form of  the standard metric $\rho(z)=\|\bz\|^2$ on $ \bC^{n+1}$ reads
\begin{align*}
\omega\quad & = \sum -\half  \ii dz_j \wedge \overline{dz_j}  
 =  \sum dx_j\wedge dy_j   = \half d\big(\underbrace{\sum x_j dy_j- y_jdx_j}_\lambda\big) ,
\end{align*}
 which is a real valued exact symplectic form. Recall (cf. \chaptername ~\ref{lect:overview}, Example~\ref{ex:BasicExs}.2) that it is a K\"ahler form and that
the restriction  to the non-singular part $V_{\rm ns}$ of $V$ is also a  K\"ahler form.  
\par
Assume that for some $c>0$ one has  $V_{< c}=V\cap\rho^{-1}[0,c)\subset V_{\rm ns}$ and that 
  the boundary $V_c$ is a submanifold. Then the  $1$-form
$\lambda  $ restricts to $V_c$  equipping it with     a contact form. The vector field 
\[
 Y_\lambda=\half  \left(\sum_j x_j \dd {} {x_j}+ y_j \dd {} {y_j}\right)
\]
 (defined on an open neighborhood of  $V_c$)  is  a Liouville field since $d(\iota_{Y_\lambda} \omega)=\omega$. It is indeed a
 radial vector field transversal to  $V_c=\sum x_j^2+y_j^2=c$ (since $\nabla(V_c)= 4Y_\lambda$).
 \\
 \textbf{2.} \index{cotangent bundle}
 The \textbf{\emph{total space $N=T^ *U  $ of the cotangent bundle  of a smooth manifold $U$}} is a symplectic manifold (cf. \chaptername ~\ref{lect:overview}, Example~\ref{ex:BasicExs}.1).
Here $\omega=\omega_{\rm can}=d\lambda_{\rm can}$ is exact. The $(2n-1)$-form 
$\lambda_{\rm can}\wedge (d\alpha)^{n-1}$ in local coordinates $x_1,y_1\dots,x_n,y_n$ can be given as a multiple of
\[
dx_1\wedge\cdots\wedge dx_n\wedge \sum_j (-1)^j(y_j  dy_1\wedge\cdots\wedge \widehat{dy_j} \wedge\cdots \wedge dy_n,
\]
which   restricts non-degenerately to the subvariety $\sum y_j^2=r^2$ and hence is a contact form on this subvariety.
\par 
This can be done more intrinsically by picking  a Riemannian metric $g$ on $ U$  inducing an  associated   norm $\|-\|_u$  on
 each cotangent space $T^*_uU$. Defining
\[
T^*_{\le r}U =\sett{V \in T^*_uU}{\forall\, u\in U, \|V\|\le r},\quad S^*_{r }U =\partial T^*_{\le r}U,
\]
the sphere bundle $S^*_{r }U$ is diffeomorphic to  the local model above   given by 
the equation $\sum y_j^2=r^2$. Note that  $r$ is  a function in the $U$-variable  $u$ alone.

Observe    that  the isomorphism $\varphi_\omega: TN \to T^*N$  defined in \eqref{eqn:iso}
associates to the form $\lambda_{\rm can}$
  a vector field $Y_\lambda$. This vector field preserves $\omega_{\rm can} $ since
\[
\cL_{Y_\lambda}(\omega_{\rm can})= d\comp  i_{Y_\lambda} \omega_{\rm can}+i_{Y_\lambda} \comp d(\omega_{\rm can})=d \lambda_{\rm can}=\omega_{\rm can} 
\]
and hence is a Liouville field.
Note that  in local coordinates $Y_\lambda=\sum (-1)^j  y_j \displaystyle \frac d {dy_j}$ and so 
$Y_\lambda$ is a vector field   transversal to the sphere bundle $\sum y_j^2=r^2$.

\end{exmple}

\section{More on contact geometry} \label{sec:contact}

\subsection{Gray stability}

A central and useful result  in contact geometry reads as follows:\index{Gray  stability}

\begin{thm}\label{thm:gray}[Gray's stability theorem]  Let $M$ be a smooth compact manifold admitting a smooth family
$\xi_t$, $t\in[0,1]$ of contact structures. Then there is an isotopy of $M$, giving a smooth family 
of diffeomorphisms $F_t : M \to  M$ such that $(F_t)_* \xi_0=\xi_t$ for all $t\in [0,1]$.
\end{thm}
 For a proof we refer to \cite[Section 2.2]{geiges}. This result states that   the contact structure (or its contact form)  
 can be smoothly  varied without changing the contactomorphism class of the contact manifold. This turns out to be crucial
in order to  define  meaningful contact invariants.  
As an example, relevant for these notes, see e.g. \cite[Prop. 2.5]{kk} on contact forms
 on the link of an \ihs\ defined by weighted homogeneous hypersurfaces. It is instructive to go through the elementary proof of this result.
 
A vector field $X$ on $M$ is called a \textbf{\emph{contact field}} \index{contact!field} if  
$\cL_X  \alpha = g\cdot \alpha$ for some function $g$ on $M$.
These fields are characterized as follows:
\begin{crit} \label{crit:ContField} A  vector  field $X$ on $M$ is a contact field if and only if for some function $H:M\to\bR$ one has
\begin{eqnarray}
\iota_X \alpha &=&H \label{eqn:CF1} \\ 
 \iota_X (d\alpha) &=   & dH + (\iota_Y dH)\alpha, \quad Y=X_\alpha.\label{eqn:CF2}
\end{eqnarray}
\end{crit} 
\begin{proof} If the above relations holds, take $g=\iota_Y(dH)$. Then it follows directly that $\cL_X \alpha= g\alpha$ and
so  $X$ is a contact field.
Conversely, if $\cL_X  \alpha = g\cdot \alpha$, take $H=\iota_X\alpha$.  Then
\begin{eqnarray*}
\iota_X(d\alpha) &=\cL_X\alpha -d (\iota_X\alpha)\\
&= g\cdot \alpha - d H.
\end{eqnarray*} 
So it suffices to show that $g=\iota_YH$. To see this, note that $d\alpha(X,Y)=0$ since $Y=X_\alpha$ is the Reeb vector field and hence, evaluating the above equation on $Y$ gives 
 $0=   g\cdot \alpha(Y)  - d H (Y)=   g-\iota_Y (dH)$.  \end{proof}
\begin{corr} For the constant function $a$ on $(N,\alpha)$ the field $a X_\alpha$  is the corresponding contact field.
\end{corr}

 \subsection{Reeb vector fields and their flow} \label{sec:reeb}
 
 Let $(M,\xi)$ be a contact structure with contact form $\alpha$. There is a unique vector field $R_\alpha$, the \textbf{\emph{Reeb vector field}} characterized by\index{Reeb!vector field} \label{page:Reeb}
 \begin{enumerate}
\item $R_\alpha$ contracts to $0$ against  $d\alpha$, i.e., $\iota_{R_\alpha} (d\alpha)= d\alpha(R_\alpha,-)=0$;
\item $\alpha(R_\alpha)=1$.
\end{enumerate}
Since there is a unique direction in which $d\alpha$ contracts to $0$, this explains (1) while (2) is a normalization. 
The first item  implies that $R_\alpha$ is everywhere transversal to the field $\xi$ of hyperplanes defining the contact structure. 
The flow $\phi^t$ induced  by this vector field  preserves $\alpha$ (and hence the contact structure)  since 
\[
\cL_Y \alpha= d i_Y \alpha + i_Y d\alpha= d (\alpha(Y))+ 0=0,  \quad Y=R_\alpha .
\]

\begin{rmk}
Note that the Reeb vector field for $f\cdot \alpha$ might be very different from the one for $\alpha$. So
the contact form   admits admits   Reeb vector fields for each of  the contact forms.
\end{rmk}

\begin{exmples} \label{exm:Reeb} \textbf{1.} 
Recall (cf. Example~\ref{exmpl:BasSymp}.(1))  that the unit  sphere $S^{2n-1}\subset \bC^n$ admits  the   contact form 
$\alpha=    \sum_{j} x_j dy_j-y_jdx_j  $. The 
  contact hyperplane   at $\bp\in S^{2n-1}$  is the subset of the tangent vectors $X$ at  $\bp$
orthogonal to $\bp$ and to $J\bp$ where $J$ is the standard almost complex structure on $T_\bp\bR^{2n}=\bR^{2n}$ given by
$J(\cdots,x_j,y_j,\cdots)=(\cdots,-y_j,x_j,\cdots)$.
The Liouville field  
\[
Y=  \sum_{i } x_i \dd{}{x_i} + y_i \dd {}{y_i}   
\] 
  at a point $\bp\in \bC^n$ gives  the radial vector $  \vv{0\bp}$ and at $\bp\in S^{2n-1}$  this vector is  
orthogonal to $\xi_\bp$ and is outward pointing. Identifying  tangent vectors with the corresponding vectors in $\bR^{2n}$, 
one has $Y_\bp=   \bp$. 
\par
The   field   $R_\alpha=      \sum- y_j \dd{}{x_j} +x_j \dd{}{y_j}$ has   value  $    J( \bp)$
 at $\bp   \in  S^{2n-1}$   which   is tangent to $S^{2n-1}$ but does not belong to the
 contact field (since $J^2(\bp)=-\bp$ is not a tangent vector).  On $S^{2n-1}$ the identity $\sum_{j=1}^{2n} x_j^2+y_j^2=0$
 implies that $ \iota_{R_\alpha} \sum dx_j\wedge dy_j)=-  \sum x_jdx_j+ y_jdy_j=0$ and since
 $ \iota_{R_\alpha} \alpha=  \sum x_j^2+y_j^2=1$, $R_\alpha$ is the Reeb field.
 Its  flow $F_t:S^{2n-1}\to S^{2n-1}$     is given in complex coordinates $z_j=x_j+\ii y_j$, $j=1,\dots, n$  
  by 
 \begin{equation}
 \label{eqn:FlowSphere}F_t(z_1,\dots,z_n)= e^{ \ii t}\cdot (z_1,\dots,z_n).
 \end{equation} 
 Since   $\dot F_t(\bp)=  \ii F_t(\bp)= J  F_t(\bp)$,  the tangent vector at $\bp$, 
 coincides with the value of $R_\alpha$ at $F_t(\bp)$.
\\
 \textbf{2.} 
 On $\bR^{2n+1}$ with coordinates $(x_1,\dots,x_n$, $y_1,\dots,y_n,t)$ the standard contact structure
is the one with contact form $\alpha:=dt- \sum y_j dx_j$. Note that $\alpha\wedge (d\alpha)^n$ is the volume form on $\bR^{2n+1} $ and so is indeed non-degenerate.
The kernel of $\alpha$ is the field of hyperplanes spanned at $(t,\bx,\by)$  by the vectors
$d/dy_1,\dots,d/dy_n$,  $d/dx_1+y_1\cdot d/dt ,\dots,d/dx_n+y_n\cdot d/dt$. In other words, this is the field  of hyperplanes
\[
(t' ,\bx' ,\by' )\mapsto \set{t-(\sum_j y'_j) x=0}.
\]
Note that $d\alpha= \sum dx_j\wedge dy_j$  does not contain $dt$ and so $d/dt$ is the Reeb vector field. It is everywhere transversal to the field of hyperplanes.

\end{exmples}

\medskip  

Since $R_\alpha$ is everywhere transversal to the field $\xi$ of hyperplanes, one has a direct sum splitting
$
T_pM = (R_\alpha)_p \oplus \xi_p
$
and this splitting is preserved by the flow  $\phi^t$ of the Reeb vector field.

\subsection{Symplectization of a contact manifold}
\label{ssec:Symps}

A contact manifold $M$ gives rise to a symplectic manifold, the cylinder \label{page:symplectize} 
\[
\cyl {M_\alpha} =( M\times (-\infty,\infty),  \omega),\quad \omega= e^t ( d \alpha- \alpha\wedge dt) = d(e^t \alpha),
\] 
which is called the \textbf{\emph{symplectization  of   $(M,\xi)$}}. The \index{symplectization} 
Liouville field on it  is the vector field $d/dt$.  One   also uses the radial coordinate $r= e^t$ instead of $t$ 
and this gives a symplectomorphism  with $(M\times (0,\infty), d(r\alpha))$. The Liouville field becomes $r^{-1} d/dr $ 
and the positive end corresponds to $r>1$.

\begin{rmk} \label{rmk:OnSymectiz}
The symplectic structure on $\cyl {M_\alpha}$ is constructed from a given contact form $\alpha$.  As explained in
Section~\ref{sec:introsymgeom}, the contact structure allows an entire family of contact forms $f\cdot \alpha$ where
$f$ is a positive differentiable function on $M$.  However, all of the symplectizations are symplectomorphic, an explicit  
symplectomorphism being iven by 
\[
\phi: \cyl {M_{f\alpha}}  \to \cyl {M_{ \alpha}} ,\quad (x,t) \mapsto (x, t \log f(x)),
\]
since $\phi^* (e^t \alpha) = f e^t \alpha = e^t (f \alpha)$.
\end{rmk}

The following relation between contact fields  on $(M,\alpha)$ and Hamiltonian fields on $ (M\times (0,\infty), d(r\alpha))$
is very useful for what follows:

\begin{lemma} Let $H:M\to \bR$ be a function determining the contact field $X_H$ as in Criterion~\ref{crit:ContField}.
Then Hamiltonian field of the function $\widetilde H= r\cdot H$ on $ (M\times (0,\infty), d(r\alpha))$ 
 is given by $X_{\widetilde H}(x,r)= X_H(x)+ d_YH $,
where $Y=X_\alpha$ is the Reeb field and $d_YH\in \bR$ is identified with the tangent vector $ d_Y H \cdot d/d r$.

\end{lemma}
\begin{proof} One calculates 
\begin{eqnarray*}
d(r\alpha)(X_{\widetilde H})& =& (dr\wedge \alpha + r d\alpha)(X_H+ d_Y H )\\
				   &=& dr\wedge \iota_{X_H}\alpha + r \iota_{X_H}(d\alpha)+\\
				   & &\hspace{5em} (d_Y H)\cdot \alpha + r\cdot 0\\
				   & \overset{ \eqref{eqn:CF1}, \eqref{eqn:CF2}}{= }& Hdr + r dH -(\iota_Y dH)\cdot \alpha + (d_YH)\cdot \alpha\\
				   &= &d(rH) = d\widetilde H. \qedhere
\end{eqnarray*} 
\end{proof}
In a similar way one shows: 
\begin{add} \label{add:RandHam} The Hamiltonian field of the function $h(r)$ on $ (M\times \bR_+, d(r\alpha))$ induces for all $r\in \bR_+$ 
the field $h'(r) R_\alpha$ on $M\times r$. In particular, the coordinate function $r$ induces the Reeb field on  $M\times r$. 
\end{add}

 \subsection{Liouville fields and contact strucures} 
Symplectic manifolds equipped with Liouville vector field induce a contact structure
on  any smooth hypersurface transverse to the field: 

\begin{prop}[\protect{\cite[Prop. 3.57]{SYmpTop}}]\label{prop:contacttype} Let $(N,\omega)$ be a symplectic manifold containing  a  compact  hypersurface $S$ 
(i.e.  a submanifold   of $N$ of codimension $1$).
Then there exists a Liouville field $Y$ in a neighborhood of $S$ which is transverse to $S$ if and only if there exists a contact form $\alpha$ on $S$ such that
 $d\alpha= \omega|_S$. If $\omega$ is given, in fact $\alpha=i_Y \omega$ defines a contact form on every hypersurface transverse to $Y$.
\end{prop}

Such $S$ is called a \textbf{\emph{hypersurface of contact type}}. The  Liouville flow $\psi_t$, associated to $Y$   maps $S$ to the \textbf{\emph{positive (negative) side of $S$}} for positive (negative) time $t$. \index{hypersurface of contact type}
The contact structure on $S$ depends on $Y$. Suppose $Y'$ is another Liouville field. Then $d(\iota_{Y'-Y}\omega)=0$.
In case  $b_1(N)=0$, this implies $\iota_{Y'-Y}\omega=d H$ for some Hamiltonian function on $N$.
In other words, $Y'-Y=X_H$, the Hamiltonian vector field associated to $H$ on $N$. The converse is also clear.
Consequently, the collection of Liouville fields is convex. But then by Gray's Stability Theorem~\ref{thm:gray} one concludes:
\begin{lemma}
The  contact structures on a hypersurface of contact type  resulting from the various Liouville fields are all contactomorphic.
\end{lemma}

Suppose that there exists a symplectic manifold $(N,\omega)$ containing   a smooth hypersurface $S$ of contact type such that the negative side is contained
in a compact  manifold  $W \subset N$ with boundary $\partial W=S$. The resulting contact manifold  $(M, \alpha)$, $\omega|_S=d\alpha$  is said to be 
\textbf{\emph{symplectically fillable}} and  $W$  is  a symplectic filling.\index{symplectic!filling}\index{filling!symplectic ---}
 
\begin{exmples} \label{exmple:liouville}  
 \textbf{1. Cotangent bundles.}   If $N= T^*U$, $U$ a smooth $n$-dimensional manifold, the Liouville vector field points outwards of the associated ball-bundle $W=T^*_{\le r}U$.
The complement of $W$  is contactomorphic to  the symplectic cylinder cylinder on $S=\partial W$ and so $S$ is symplectically fillable with $W$.
\\
\textbf{2. Milnor fibers.}   
The fibers of the Milnor fibration of an \ihs\ given by a hypersurface $\set{f(\bz)=0}$ in $\bC^{n+1}$ (with singularity at the origin)
all have a symplectic structure induced by the K\"ahler structure on $\bC^{n+1}$ (see Example~\ref{ex:BasicExs}(3)).

According to Example~\ref{exmpl:BasSymp}(3) 
 the link of the singularity which is the common boundary  of these fibers,  admits a contact structure with contact form 
$\lambda =\half (\sum_{j=1}^{n+1}  x_j dy_j -y_jdx_j)|_{\lnk f}$ and an outwards pointing 
  Liouville field. The (closed) Milnor fiber  then can be viewed as symplectic filling of the link
  $(\lnk f, -\lambda)$.  The link is   the boundary $\lnk f\times \set{0}$ of a "positive cylindrical end"
 $\lnk f\times [0,\infty)$  glued to the Milnor fiber (in $t$-coordinates). 
 \end{exmples}

  The above examples are Liouville domains:
 \begin{figure}[h]
  \begin{center}
  
    \begin{tikzpicture}
      \begin{scope}[scale=0.85]
        \cylind{-1}{-4} \cylind{-3}{-4} \cylind{1}{-4} \cylind{ 3}{-4}
        \begin{scope} 
         \draw[green]  (-3,-2) circle [x radius=0.5, y radius=0.25];
          \draw[green] (-1,-2) circle [x radius=0.5, y radius=0.25];
          \draw[green] (1,-2) circle [x radius=0.5, y radius=0.25];
          \draw[green] (3,-2) circle [x radius=0.5, y radius=0.25];
         
          \shade[color=green!20!white]  (-3.5,-2) to[out=90,in=180] (0,0) to[out=0,in=90] (3.5,-2);
          \draw (-2.5,-2) to[out=90,in=180] (-2,-1.5) to[out=0,in=90] (-1.5,-2);
          \draw (-0.5,-2) to[out=90,in=180] (0,-1.5) to[out=0,in=90] (0.5,-2);
          \draw (1.5,-2) to[out=90,in=180] (2,-1.5) to[out=0,in=90] (2.5,-2);
          \draw[<-] (-3.4,-3)--(-0.6,-3); \draw[->] (0.6,-3)--(3.4,-3);
          \node at (0,-1) {\(W\)};   \node at (0 ,-3 ) {\(\cylend  {S_\alpha}\)};
           \node at (4.7,-2.2) {\(\partial W={S_\alpha} \)};  
          \draw[<-] (3.5,-2.2 )--( 4   ,-2.2  );  
         \end{scope};
      \end{scope};
    \end{tikzpicture}
     \caption{A Liouville domain $W$ with its symplectic completion $\widehat W$.}
\label{fig:liouville}

    \end{center}
\end{figure}
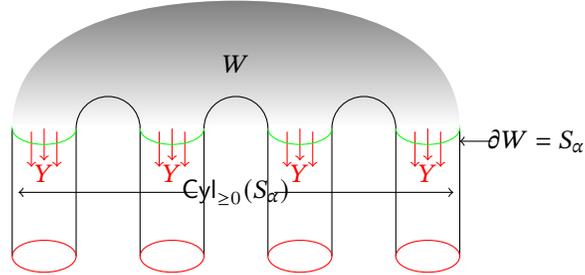

 \begin{dfn} A \textbf{\emph{Liouville domain}} \index{Liouville!domain}is a compact  symplectic manifold $(W,\omega)$ with boundary $S=\partial W$ and Liouville field $Y$ defined in a neighborhood of $S$ and which points
outward of $S$.  Then  $S=\partial W$
  is called a \textbf{\emph{symplectically convex}} boundary of $W$.\index{symplectically convex boundary}
 \end{dfn}
 
 The flow of the Liouville field gives a  suitable  neighborhood $U$ of $S$ in $W$  a
 cylinder-like structure, say  
 \[
 G: S\times [-\delta,0] \mapright{\sim} U\subset   W.
 \]
 Hence  $( S\times [-\delta,0] , G^*(e^t \cdot d\alpha))$ then becomes a compact subset of $\cyl{W_{\alpha}}$.
 So using $G$ the domain $ W$ can be glued along $S\times \set{0}$ to the positive cylindrical end 
 $ \cylend {S_\alpha} = S_\alpha   \times [0,\infty) $ (in $t$-coordinates) which by definition   gives 
  $\widehat{W}$, the \textbf{\emph{symplectic completion}} of $W$
  illustrated in  Figure~\ref{fig:liouville}.\label{page:SympComp} \index{symplectic!completion}

\begin{rmk}
\label{rmk:BndLiouvDom}
By Remark~\ref{rmk:OnSymectiz} two contact forms on $S$ giving the same contact structure on the cylindrical  ends
give symplectomorphic cylindrical ends. For each  $\epsilon>0$ the symplectomorphism restricted to $t\ge \epsilon $ extends to
a neighborhood of $S$ in  the symplectic filling
by replacing $f(x)$ for $0 < t< \epsilon$ by  a positive function $g(x,t)$ with $\lim_{t\to 0}g(x,t) =1$  and 
 $\lim_{t\to \epsilon}g(x,t) =f(x)$. In particular $\omega$ does not change on $W$ under the symplectomorphism. 
 
 So the contact form $\alpha $  on $S\times \set{0}$ is induced by the Liouvillle field, but the contact form 
  on  $S\times \set {t} $, $t\ge \epsilon$  can be supposed to be    equal  to  $f\cdot \alpha$, $f $ any positive function on $S$ and
  so is not necessarily induced by  a Liouville field.
 
  \end{rmk}

 On a Liouville domain $(W,\omega) $ the  Liouville vector field $Y$  preserves $\omega$ and 
 points outwards from  its     boundary $S=\partial W$ while the Reeb field for $\alpha=\iota_Y \omega|_S$ is tangent to
 $S$ but is not contained in the contact field $\xi$.  So one gets a direct sum splitting
\begin{equation}
\label{eqn:OnLiouvilleDoms}
T_pW =  (R_\alpha)_p \oplus \xi_p \oplus Y_p,\quad p\in S=\partial W.
\end{equation}
Observe that  $\omega$ restricts non-degenerately to  the span of $(R_\alpha)_p$ and $Y_p$,  since $\omega(Y,R_\alpha)= \alpha(R_\alpha)=1$.
A  periodic flow  of $R_\alpha$  induces a flow of the contact field  $\xi$ which preserves its  symplectic structure. 
So a trivialization  of $TW$ along a closed 
orbit    of the flow induced by $R_\alpha$  preserves $\xi$   induces a curve in the symplectic group $\psi:[0,T]\to \spl{2n-2 }$ starting at $I_{2n -2}$. 
If $\gamma$ is a periodic orbit of period $T$, then  $\psi(T)$ is called the  \textbf{\emph{linearized return map}}.  
\index{linearized return map}  
 
\begin{exmple}[The standard sphere $S^{2n-1}$] \label{exm:reeb} \index{unit sphere} This is a continuation of the calculations of  Example~\ref{exm:Reeb}.\textbf{1.} 
Observe that the contact field $\xi$ of $S^{2n-1}$ at the point $e_1$ is given by the $2n-2$ tangent vectors
$e_3,\dots, e_{2n}$. Rephrased in terms of the complex basis $\set{e_1,e_3,\dots,e_{2n-1}}$,  this subspace
  can be written $\xi_{e_1}=\bC e_3+ \bC e_5+\cdots+\bC e_{2n-1}$. 
  The tangent map (or linearization) of the Reeb flow   $F_t: \bx \mapsto e^{\ii t} \bx$ is the linear map $\psi(t):\bC^n\to \bC^n$
  given by multiplication by $e^{\ii t}$ and its restriction to the subspace $\xi_{e^{\ii t}  e_1}\subset T_{e^{\ii t}} S^{2n-1}$  is likewise multiplication by $e^{\ii t}$.
  This is a symplectic matrix as it should, and  the linearized return map  is  its value at $2\pi$ which is the identity.   
\end{exmple}

 \section{ Strongly Milnor fillable links}
 \label{sec:strongfill}
 
 These notes are mainly devoted to the symplectic and algebraic geometry of isolated hypersurface singularities and for 
 those the link is the boundary of the Milnor fiber. For isolated singularities of varieties that  cannot be embedded  
 as hypersurfaces  in $\bC^N$  there is an alternative filling. To explain  this, 
 let $X\subset \bC^N$ be an algebraic subvariety and assume that $(X,x)$ is  an isolated normal singularity.
If $S(x,\delta)$   is  the euclidean sphere    in $\bC^N$ with center $x$ and small enough radius $\delta$, the intersection
$ L_x= X\cap S(x,\delta)$   is diffeomorphic to the link of $x$ in $X$.  
 If  $\omega= d\lambda $ is the standard K\"ahler form on $\bC^N$ 
(see Eqn.~\eqref{eqn:StFormCn}), then $\lambda$ gives    $L_x$   a contact structure.

The idea is to regard a tubular neighborhood   of the exceptional set
in a resolution of singularities of $x$ as a substitute for   the Milnor filling.  By
Hironaka's results recalled in Section~\ref{sec:classcentral},   there is a "good"  embedded
 resolution 
 \[
 \sigma: (\widetilde X,E)\to (X,x)
 \]
 for which  $E=\cup_{i\in I}E_i$ is a hypersurface whose components $E_i$ form a normal crossing divisor.
 Since $\sigma$ is a resolution of the singularity at $x$, the inverse image under $\sigma$ of the link  embeds diffeomorphically in $\widetilde X\setminus E$ and 
  the link (with its contact structure) admits a special type of filling in $\widetilde X$, called a \textbf{\emph{strong Milnor filling}} in
 the following sense:\index{Milnor!strong --- filling}\index{strong Milnor filling}\index{filling!strong Milnor ---}
 
\begin{thm} \label{thm:OnSympFilling} Let $B(x,\delta)$   be 
the euclidean ball  in $\bC^N$ with center $x$ and radius $\delta$ and set 
\[
W= \sigma^{-1} (X\cap B(x,\delta)).
\]
Then for small enough $\delta$, its boundary 
$\partial W$,  is diffeomorphic to the link of $x$ in $X$ and $W$ is a symplectic filling of the link.
\end{thm}
\begin{proof}[Outline of the Proof]
Clearly, for $\delta$ small enough, $\partial W$ is contactomorphic to the    link $L_x=\lnk {X,x}$
with  contact structure  induced from the standard contact structure on the $2N-1$-sphere.

A single blow up     $p:\widetilde \bC^N\to \bC^N$ in  the point  $x$ has a K\"ahler metric of the form $\eta:=p^*\omega +\epsilon \cdot \tau$,
where $\omega= d\lambda $ is the standard K\"ahler form on $\bC^N$ (see Eqn.~\eqref{eqn:StFormCn}) and 
$\tau$ is a closed $(1,1)$-form which is strictly positive along the fibers of $E=p^{-1}x \to x$ and zero outside a compact neighborhood of $x$.
For a proof see e.g \cite[Prop. 3.24]{voisin}. Assuming for simplicity that  one blow up resolves $X$, then $W$ is a submanifold
of  $ \widetilde \bC^N$ and $\eta$ restricts to a K\"ahler form on $W$. The form $\eta$ restricts to $p^*\omega$ near the boundary, i.e., near the link 
of the singularity,   provided $\epsilon$  is small enough. But since $\omega=d \lambda$ for some 
$1$-form  $\lambda$, on the link one has $\eta=d\alpha$ where $\alpha=p^*\lambda$.
This  is the contact form defining the contact structure on the link.
The general case is  slightly more  complicated. Details are  left to the reader. 
\end{proof}

\begin{rmk}
There is an alternative procedure  due to McLean as explained in the proof of  \cite[Lemma 5.25]{mclean}. 
This approach  is better suited to make a comparison between minimal discrepancy and symplectic phenomena
near the contact boundary. The alternative construction gives a contact form on the link  which is isotopic to the  classical contact structure on the link
considered  above.
So, by Gray's stability theorem, Theorem~\ref{thm:gray}, the two contact structures are  contactomorphic.  
\end{rmk}

\chapter[Hamiltonian and Reeb dynamics, symplectic cohomology]{Hamiltonian and Reeb dynamics, symplectic cohomology}
      \label{lect:hamreeb&sympcoh}

\section*{Introduction}
The main goal of this chapter   is  to give a basic idea of  symplectic cohomology. This requires  
  to introduce  (in Sections~\ref{ssec:maslovind}--\ref{sec:CZindex}) 
   the Conley--Zehnder index  of a periodic orbit of a  Hamiltonian  flow.
After the definition of symplectic cohomology in Section~\ref{sec:scoh} (assuming some  deep results in global analysis),  I shall 
\begin{itemize}
 \item  extract  contact invariants from symplectic cohomology for a certain type of  boundary of a Liouville domain, 
namely a so-called dynamically convex boundary.
 \item give an overview of McLean's results which relate the algebraic notion of minimal discrepancy and the symplectic notion of highest 
 minimal index of  periodic orbits of    the Reeb flow(=Hamilton flow restricted to  the link). Applied to  the link of a cDV singularity,
 these   results  will be shown to determine  whether the singularity  is canonical or terminal
  (Theorem ~\ref{thm:mclean}).
 Another important result for 3-dimensional singularities is the characterization  of 
 smoothness in terms of contact invariants (cf. Corollary~\ref{cor:S5IffSmooth}).
\end{itemize} 
 
\section{The Maslov index} \label{ssec:maslovind} 
This  index is an 
an integer associated to   a  \textbf{\emph{loop}}   in the symplectic group  based  at the identity. 
It is a (based) homotopy invariant. 
  To explain the definition, let $(V,\omega)$ be an even dimensional real vector space $V$ equipped with a non-degenerate skew form $\omega$.
By definition  the symplectic group \label{page:SympGr}  is given by 
\[
\spl V := \sett{ T \in \gl{V}{}}{ \omega(Tx,Ty)= \omega(x,y) \text{ for all } x,y\in V}.
\]
 A symplectic basis $\set{e_1,\dots,e_n,f_1,\dots,f_n}$ for $V$ is one for which 
 $\omega(e_i,e_j)=0$, $\omega(f_i,f_j)=0$ and $\omega(e_i,f_j)=-\omega(f_j,e_i)= \delta_{ij}$, 
 $i,j =1,\dots, n$. In other words,  this is a  basis in which $\omega$ is represented by the matrix
\[
J_n= \begin{pmatrix}
0_{n } &\id_{n } \\
-\id_{n } & 0_{n }
\end{pmatrix} .
\]
In such a basis the  symplectic transformations are given by the symplectic matrices 
$
\sett{M\in \bR_{2n\times 2n}}{ \trp M J_n M= J_n}$. 
These have  polar decomposition given by 
\[
 M=PQ,\quad P=(MM^T)^{ \half},\,   Q=P^{-\half}  M \in \ogr{2n}\cap \spl{2n}.
 \]
Hence   $Q$  can be written as    $Q =\begin{pmatrix} X& -Y\\ Y&X
\end{pmatrix}$  which leads to  a homomorphism 
\[
\rho: \spl {2n}  \to S^1\subset \bC,\quad \rho(M) =  \det (X+\ii Y).
\]
So  a path $\psi:I=[0,1]\to \spl{2n}$ projects to a path $\rho\comp \psi$ in the circle. Closed paths based  at $\psi(0)=I_{2n}$
then give closed path on $S^1$. Such a loop has    a winding number which defines  the \textbf{\emph{Maslov index}} of the loop. \index{Maslov!index}
It does not depend on the chosen symplectic basis and it is invariant under homotopies preserving the base point.
It also is additive under concatenation of loops which shows that $\pi_1(\spl V,I_{2n})\simeq \bZ$.
\par
For  non-closed paths  whose end point is a matrix with no eigenvalues equal to $1$ the above definition can be modified.
The idea is to extend the path in a careful way so that  $\rho$ gives a closed path on the circle.  The choice of path makes use of the the
 \textbf{\emph{Maslov cycle}}, the hypersurface\index{Maslov!cycle}   $\Sigma\subset \spl{2n} $ consisting of matrices with eigenvalue $1$, i.e., 
\[
\Sigma= \sett{M\in \spl{2n}   }{\det (M-I)=0}.
\]
To explain the procedure,  first  note that $\Sigma$   divides $ \spl{2n}$ in two connected components determined by the sign of $\det (M-I)$. One chooses
 matrices, say $K^\pm$    in each of these components with the property that its square  under $\rho$
 projects to $1\in S^1$. For instance, one
may take $K^+= \begin{pmatrix}
-I_n & 0_n\\ 0_n & -I_n
\end{pmatrix}$,   respectively $K^-=Q $, $Q=\begin{pmatrix}
D &0 \\0& D^{-1}
\end{pmatrix} $ with $D= \text{diag}(a, -1.\dots, -1)$, $a>0 , a\not=1 $.   
By assumption  the endpoint  $\psi(1)$ of the path belongs to a connected component of $ \spl{2n}\setminus \Sigma$
and so this point can be connected  within this  component   to the appropriate  
 point $K^\pm$ yielding a path   $\widetilde \psi: [0,  2]\to   \spl{2n}$.  So  now  $t \mapsto \rho^2(\widetilde \psi(t))$ is
 loop on the circle. It winding  number  is the searched for Maslov  index for non-closed paths,\label{page:MaslovInd}
\begin{equation}
\label{eqn:mu}
\mu(\psi)=  \deg (t \mapsto \rho^2(\widetilde \psi(t)) ) \in  \bZ.
\end{equation} 
Note that this excludes the case where  $\psi$ itself is a loop at $I_{2n}$. 
  The definition of $\mu$ can be extended to all paths in such a way that its value   does not change  under 
homotopies leaving endpoints fixed. This property together with a few more characterizes the $\mu$-invariant, the
obvious one being  the   additivity  under concatenation. See  \cite{robbsal} 
where it is also shown  (\cite[Remarks 4.10 and 5.3--5.5]{robbsal}) that
for loops at $I_{2n}$  one gets  twice the
original Maslov index and that it  coincides with the above defined  $\mu$-index for non-closed paths.

In particular  there is  freedom  to move the path in such a way tha the Maslov index becomes susceptible to calculation,
for example by deforming it to become non-degenerate in the
following sense.
 
 \begin{dfn} A path $\psi: [0,T]\to \spl{2n}$ starting at $I_{2n}$   is said to be   \textbf{\emph{non-degenerate}} if 
 \begin{itemize}[nosep]
\item $\psi$ meets $\Sigma$ transversally;
\item at an intersection point the so-called crossing quadratic form  (see below)  is non-degenerate.
\end{itemize}
The  \textbf{\emph{crossing quadratic form}} at  an intersection point   $ \psi(t_0)\in \Sigma$ 
is the   quadratic  form $Q(\psi,t_0)$ on $(V,\omega)$ which  defined by 
\[
Q(\psi,t_0)(v)= \omega(v, d\psi/ dt|_{t_0}(v)).
\]
\end{dfn}

If the crossing form is non-degenerate it can be diagonalized over $\bR$  such that diagonal entries are non-zero. 
Denote its signature ($\#$ of positive enties  - $\#$ of negative entries) by $\text{sgn} (p)$. 
It turns out that the  Maslov index of any  non-degenerate path 
$\psi:[0,1]\to \spl{2n}$, closed or non-closed is expressible in terms of these signatures:
\begin{equation}
\label{eqn:muindex}
\mu(\psi)= \half \text{sgn }    \psi(0)+ \half  \text{sgn }  \psi(1) +  \sum_{0<t_* < 1} \text{sgn } \psi(t_*).
\end{equation} 
This is a half-integer and  an integer for paths with end points not on the Maslov cycle or for loops. 
See    e.g. e  \cite[p. 45--47]{SYmpTop} or in \cite{robbsal} for a proof.
The following example shows that for non-degenerate paths the $\mu$-index can be calculated in a straightforward fashion.
See also Example~\ref{exm:reebBis}.

\begin{exmple} \label{exm:Maslov}
Consider the path $\psi(t)=  \begin{pmatrix}
\cos (2\pi a\cdot t) & \sin (2\pi   a\cdot t)\\  -\sin  (2\pi  a\cdot t )&\cos    (2\pi   a\cdot  t)
\end{pmatrix}$, $t\in[0,1]$, $a\in \bQ^+$. This path intersects $\Sigma$ when $at\in \bZ$.
If $a$ is not an integer this is the case for $t_*=(\lfloor a\rfloor - k) /a$, $k=0,\dots,\lfloor a\rfloor $.
Since $\dot\psi(t_*)=  J=\begin{pmatrix}
0& 1\\-1&0
\end{pmatrix} $, one finds $ \omega(v, Jv)=v\cdot v$, a form of index $2$. Hence $\mu(\psi)= 1+2\lfloor a\rfloor$.
The path $\psi$ is a loop in case $a$ is an integer, say $a=m$, and then  $\mu(\psi)=  1+2(m-1)+1=2m$. 
So this is twice the original Maslov index.
\end{exmple}

\section{The Conley--Zehnder index} \label{sec:CZindex} 
One applies the preceding construction  first of all in the setting of a symplectic manifold  $N$ equipped with  a smooth function $H$.
The  flow of the Hamiltonian vector field $X_H$   induces a  path in the tangent bundle of $N$ along 
any integral  curve $\bx(t)$.  The tangent bundle admits a fiber-wise symplectic structure and since
  the flow preserves the symplectic structure, it induces  a path  of symplectic frames $\bF(t)$.
  Assuming  that $\bx(t)$ is a closed path, say $\bx(0)=\bx(1)$,  one   makes  an important 
\begin{center}
\textbf{Assumption:}  There is an orientation preserving  trivialization
$\sigma: TN|_{\bx([0,1])} \mapright{\sim}  [0,1]\times \bR^{2n} ,  2n=\dim N $.
\end{center}
Hence the standard basis for $ \bR^{2n}$ gives the standard frame $\bE(t)$ along $\bx(t)$ and one may  assume that $\bF(0)=\bE(0)$.
Then $\bF(t)= \psi(t)\bE(t)$, $\psi(t)$ a path of symplectic matrices starting at $I_{2n}$,
  which is called the  \textbf{\emph{path in $\spl {2n}$ induced by the trivialization $\sigma$}}.
\par The above assumption holds if there is a $2$-disc spanning the closed path $\bx$ so that $\bx$ is contractible.
This will automatically be the case if the manifold $N$ is simply connected, for example if $N$ is a Milnor fiber of some isolated hypersurface singularity.

\begin{dfn} \label{dfn:Maslov}
Let  $\bx:[0,1] \to N$ be a smooth closed integral curve of the Hamiltonian flow associated to $H$ and let $\psi: [0,1]\to \spl {2n}$ be the  
path of symplectic matrices induced by the trivialization $\sigma$ (which has been  assumed to exist).
The  \textbf{\emph{Conley--Zehnder index}} $\cz {H,\bx }$ of $\bx$ is
equal to the index $\mu (\psi)$ (see \eqref{eqn:mu}).\index{Conley--Zehnder index}\label{page:CZIndex}
\par 
 This index  does not depend on the chosen trivialization $\sigma$ and it is a homotopy invariant (for paths leaving begin and endpoints fixed).
If the  path $\psi$ is non-degenerate,   it is given by \eqref{eqn:muindex}.
 In case $\psi$ is a loop, this integer  is twice   the Maslov index. 
\end{dfn}

\begin{exmples}  \label{exm:CZindex}   \textbf{1.} By Theorem ~\ref{thm:milnor},  the tangent bundle of the \textbf{Milnor fiber}
 is trivializable. 
Hence in the   Conley--Zehnder  indices for $1$-periodic orbits  of the Hamilton flow are well defined.  \\
 \textbf{2.  Liouville domains.}
The Reeb flow on the boundary $S$ of a Liouville domains $W$ and   on the slices $M_r=\set{r=\text{constant}}$  
of its cylindrical end in the completion $\widehat W$ can be compared  with the Hamiltonian flow on $W$
for  special Hamiltonians   which on the cylindrical end are  of the form $h(r)$ and so are constant along $M_r$.
By Corollary~\ref{add:RandHam} the Hamiltonian  flow coincides with the Reeb flow (but the "speed" may be different).    Recalling  the splitting \eqref{eqn:OnLiouvilleDoms}:
\begin{equation*}
T_pW =  (R_\alpha)_p \oplus \xi_p \oplus Y_p,\quad p\in S,
\end{equation*}
note that the  vector subspace $ \xi_p$ is preserved by the Reeb flow $R_\alpha$. It is a symplectic subspace since $\omega$ 
restricts non-degenerately to it. Indeed,  
it does so on its  (symplectic) orthogonal complement, i.e.,  the span of $Y$ and $R_\alpha$, as observed just after 
 \eqref{eqn:OnLiouvilleDoms}.  Observe however   that, like the Reeb vector field,    
 also  the  splitting depends on the chosen contact form $\alpha$.

 Under the assumption that $TW|_S$ can be trivialized,  this implies that  one can define a  Conley-Zehnder index for 
 a closed path $\bx$ of the Reeb flow as the Maslov index of the associated path in $\spl{2n }$ obtained by   following  the
 induced flow in the tangent bundle $TW$ along  $\bx$. 
This path starts at the identity and ends at a matrix describing the    linearized return map defined in \S~\ref{sec:reeb}
in the situation of contact fields.
   \index{Reeb!Conley-Zehnder index of periodic   ---  flow}
   \end{exmples}
 
One can also consider periodic orbits $\gamma$ of the Reeb flow on an $2n-1$-dimensional
 \textbf{contact manifold} $(M,\xi) $. 
The contact field  $\xi$  is preserved by the Reeb flow and if
the tangent bundle of $M$ can be trivialized along $\gamma$, the linearized flow restricted to
$\xi$  induces a path in 
$\spl {2n-2}$ and so has  a Conley--Zehnder index.

\begin{exmple}[The standard sphere $S^{2n-1}$ revisited]  \label{exm:reebBis} 
 Example~\ref{exm:reeb}  exhibits  a Reeb orbit on $S^{2n-1}$ for which  the lifted path   ends on  the Maslov cycle.
 The   Maslov index of the orbit starting at $(1,0,\dots,0)$ is equal to $ n-1$ since the path $t\mapsto \det (e^{\ii    t}\cdot I_{n-1} )  $, 
 $t\in [0, 2\pi$],  has winding number $(n-1)$. 
 The Conley--Zehnder index, which is the Maslov  index for paths then equals  $2(n-1)$, as observed previously. Note that the path is a non-degenerate loop.
  
  Consider  the  perturbed contact structure with contact form  
  $   \sum_{j} a_j(x_j dy_j-y_jdx_j)  $  on $S^{2n-1}$, where the $a_j\in \bQ^+$ are 
   linearly independent over $\bQ$. Of course the contact field is not the same as the standard contact field (this
   is only the case if all the $a_i$ are the same), but it   is homotopic  to
   the standard contact structure. Its Reeb field in complex coordinates
  at $\ (p_1,\dots,p_n)$ is given  by by $(( \ii t/a_1)\cdot  p_1,\dots,   (\ii t/a_n)\cdot p_n)$ with flow $F_t( e^{\ii t/a_1}z_1 ,\dots, e^{\ii t/a_n}z_n)$.
  Since the periods  $2\pi a_j$ are independent over $\bQ$,  the only
 way to get a  periodic Reeb orbit occurs when all but one coordinate equals  zero  which gives  $n$ of these
through each basis vector   $e_j$ and with period  $2\pi a_j$.
   
  The linearized flow is  represented in the standard basis by the diagonal matrix $\psi(t)=(e^{\ii t/a_1},\dots, e^{\ii t/a_n})$. Let me consider the flow starting at $e_1=(1,0, \dots,0)$.  As in  Example~\ref{exm:reeb} one finds that the restriction to the contact field along this
  orbit is  described by the  path of a complex diagonal  matrix   $(d_2(t),\dots, \dots,d_n(t))$,  
 $d_k(t) = e^ { \ii  t  a_1 / a_k } $, $k=2,\dots,n$,  where the time has been rescaled to be in the interval  $  [0,1]$.
 The   index of this path  can be calculated using   \eqref{eqn:muindex}.
 The crossings occur when for some $k\ge 2$ one has $  t_*\in  2\pi (a_k/a_1)  \bZ$ and
 also  $0\le t_*   \le 2\pi  $. Since $a_1/a_k$ is not
 an integer, this is the case if $t_*     =    (\lfloor a_1/a_k \rfloor - j)   (a_k/a_1)     \cdot 2\pi     $, $j=0,\dots,
  \lfloor a_1/a_k \rfloor $.
 As in Example~\ref{exm:Maslov}, one finds that at each of these crossings (except for $j=0$) the index equals $2$ while
 for $j=0$ the contribution equals $(n-1)$.
 Hence, the  Conley--Zehnder index equals $n-1+2\sum_{k=2}^{n-1}     \lfloor a_1/a_k \rfloor$. 
 Similarly, for the other closed orbits at $e_k$, one finds $\gamma_k=n-1+2\sum_{k\not=j}      \lfloor a_j/a_k \rfloor$, $j=3,\dots, n$.
Suppose $a_1<a_2\cdots<a_n$, one has $ \lfloor a_j/a_n \rfloor=0$ for $j=1,\dots,n-1$ and so $\gamma_n= (n-1)$
 which is the minimal index for any such flow. Higher minimal indices are  only possible if there is a non-trivial relation with $\bQ$-coefficients  among the $a_j$
 and then $2(n-1)$ is the highest possible minimal index, realized by the standard flow.
\end{exmple}

 \section{Symplectic Cohomology of a Liouville domain $(W,\omega)$}
 \label{sec:scoh}

\begin{small}
In this section it is assumed that $W$ is a parallellizable Liouville domain, i.e $TW$ is trivializable. 
\end{small}
\medskip 
 
\subsection{Interlude: Morse (co)homology} \label{ssec:MorseHom} 

 Floer (co)homology  which underlies the concept of \index{Morse!(co)homology}
symplectic cohomology is an extension of Morse homology. Since the   definition  of Floer homology  is
quite involved, the construction of Morse homology helps to understand Floer homology.\index{Floer!homology}
As a  reference for details I advise  the illuminating  lecture notes~\cite{hutch}.
\par
The  setting of symplectic geometry is now changed:  one starts with  a   
differentiable manifold $M$ equipped with a  differentiable function $f: M \to \bR$
with isolated critical points, i.e.,  points where $df=0$. The  Hessian $H_p(f)$ at a critical point $p$ is the 
quadratic form given in local coordinates $ x_1,\dots,x_m $ by the 
matrix of the second order partials $\partial^2 f /  \partial x_i \partial x_j $ at  $p$.
This is independent of the choice of coordinates. If  $H_p(f)$ is non-degenerate, it is a diagonalizable matrix with, 
say $h_+(p)$ positive and $h_-(p)$ negative eigenvalues. In other words, locally at $p$,
 coordinates $y_1,\dots,y_m$ can be found such that the function
$f$ can be written as 
\[
f(y_1,\dots, y_m)= f(p)+ y_1^2+\cdots y_i^2 -(y_{i+1}^2+\dots+y^2_m),\quad i=h_+(p),   m-i= h_-(p).
\]
If all the Hessians are non-degenerate, $f$ is called a \textbf{\emph{Morse function}}, and $h_-(p)$ is the \textbf{\emph{Morse index}} at $p$.
The   $i$-th Morse chain group    is  given by \label{page:Mcomplex}
\[
C_i^{\rm Morse} M = \bigoplus_ p \bZ \cdot p, \quad p \text{ critical point with } h_-(p) =  i.
\]
In order to define the Morse \emph{complex} relating the Morse chain groups,  one first  chooses a metric $g$ 
making it possible to define the negative gradient  vector field $-\nabla (f)$
and its flow $\Psi_s: M\to M$, i.e. $\Psi_0=\id$ and $d \Psi_s/ds  = -\nabla (f)$.  Next, to  every critical point $p$ of a Morse function $f$ one associates two submanifolds,
$N^\pm _p = \sett{ q \in M}{ \lim_{s\to \pm \infty} \Psi_s (q)= p}$, the ascending and descending submanifolds at $p$. 
One can show that $N^\pm_p$ is an embedded disc of dimension $h_\pm(p)$. If for all critical points of the Morse function the ascending and descending submanifolds at $p$
are transversal, one calls $(f,g)$ a \textbf{\emph{Morse--Smale datum}}. \index{Morse--Smale datum}
Given a Morse function  $f$, the pair $(f,g)$ is Morse--Smale for generic metrics $g$.
Assuming this, for a pair $(p,q)$ of critical points,  one sets 
\[
\cA(p,q)= N^-_p\cap N^+_q , 
\]
which, assuming that $h^-_p > h^-_q$, is a manifold of dimension $h^-_p - h^-_q$. 
This manifold contains all the flow lines from $p$ to $q$, that is,
paths $\gamma:\bR\to M$ with $\lim_{s\to -\infty}-\nabla (f)(\gamma(s))=p$ and $\lim_{s\to  +\infty}-\nabla (f)(\gamma(s))=q$.
 Hence there is an induced $\bR$-action on $\cA(p,q)$ and one obtains
\[
\cM(p,q)= \cA(p,q)/\bR,
\]
a manifold of dimension $h^-_p - h^-_q-1$. If $h^-_p - h^-_q =1$ this is a $0$-dimensional manifold, so a number of points.  For all $p\not=q$,
one can orient the manifold $\cM(p,q)$ as explained in ~\cite{hutch}. In particular, in this way one can count the number of signed points of $\cM(p,q)$ which is denoted $\# \cM(p,q)$.
One then sets
\[
d:  C_i^{\rm Morse} M \mapright{\quad} C_{i-1} ^{\rm Morse} M), \quad d (p) = \sum_q  \# \cM(p, q) \cdot q,\quad h^-_q=i-1.
\]
It is easy to show that $d\comp d=0$ and the  homology of this complex by definition is the Morse homology $H_*^{\rm Morse}(M)$.
It is independent of all choices. 
Moreover, the map assigning to a critical point $p$ of Morse index $h^-_p= i$ its descending manifold $N^-_p$ considered as a singular $i$-simplex can be shown to
induce  an  isomorphism  \label{page:Morse}
\begin{equation}
\label{eqn:MorseIsSing}
H_i^{\rm Morse}(M) \mapright{\sim} H_i (M,\bZ).
\end{equation} 
Dualizing the above complex yields Morse cohomology, which therefore  is isomorphic to singular cohomology.

\subsection{Definition of symplectic cohomology for a Liouville domain}
 
While the Bott complex  for  $M$ is built on the set of critical points of a Morse function, Floer homology on a symplectic 
manifold $(N,\omega)$
is built on the set $\cP(H)$ of periodic orbits of the Hamiltonian flow of a function $H:M\to \bR$, 
where instead of the gradient of a metric, one uses the form $\omega$
to define the Hamiltonian vector field $X$  defined by $\omega(X , -)= - dH$. 
Since the critical values of $H$  can be considered as \emph{constant periodic orbits},
 the procedure that will be outlined below 
indeed exhibits Floer cohomology as an extension of Morse  cohomology.
More precisely, if $H$ and all of its first  and second derivatives are small enough, one
can show \cite[\S 1.2]{survey} that $H$ has no periodic obits at all and so in that case Floer cohomology coincides 
with Morse cohomology.  

\par 
I shall exclusively deal with symplectic cohomology for   Liouville domains, and so I shift  to the usual notation  $W$ instead of $N$. The symplectic form is then exact  near the cylindrical end of  $ W$, say
 $\omega=d (r \alpha)$ where $\alpha$  gives $\partial  W$ the structure of a compact contact manifold. 
To define symplectic (co)homology on $W$ one  makes use of  the completion $\widehat W$ of $W$. 
Moreover, in the definition special  Hamiltonian  functions  on $\widehat W$ are used:
\begin{dfn} An \textbf{\emph{admissible Hamiltonian}} on $\widehat W\ $ is a smooth function 
$H :\widehat W  \to \bR  $  with
\begin{enumerate}[(i)]
\item $H$ is a general Morse    function  on  $W$ which is small in the $C^2$ norm on the complement
in $W$ of the negative cylindrical end;
\item  on the positive cylindrical end one has $H=a  \cdot r +b $, $a\in \bR$ positive, $b\in \bR$.
\end{enumerate} \index{Hamiltonian!admissible ---}
\end{dfn}
In addition, one assumes that $\omega$ comes from a K\"ahler metric $g$ making $(H,g)$  
 Morse--Smale.  As  just was observed, requirement (i)  implies
  that the Hamilton flow of  such a  Morse function only has  critical points 
on the complement in $W$ of the negative cylindrical end  so this takes care of the cohomology of $W$.
The symplectic information comes from    the Hamiltonian flow
  near the cylindrical end of  $ W$  which coincides with the Reeb flow for $\alpha$. 
 
  \par
  
  The \textbf{\emph{Floer cochain group}}
   in degree $k$ is the free group   on all periodic  orbits of index $n-k$:\label{page:floer}
\begin{equation}
\label{eqn:cfk} \fcc  k H=\bigoplus_{\bx} \bZ\cdot \bx,\quad \bx\in\cP(H),\, \cz{H,\bx}=n-k,\quad 2n=\dim W .
\end{equation} 
The  definition of   the   boundary operator of the Floer complex  resembles that of the one of the Morse complex, but it is more involved.
 For details consult e.g.\ \cite{seidel2}; in outline this goes as follows. 
 \\
\textbf{1}. One starts with an 
almost complex structure  $J$ on $\widehat W$ compatible with $\omega$
on $W$ and with $  d(r \alpha)$ on the cylindrical end.
Each $J$ defines a compatible metric on the tangent spaces, that is, a metric  $g_J$  for which 
 $g_J(X,Y)= \omega(X,Y)$, $X,Y\in T_x\widehat W$. $x\in \widehat W$. This metric is used  to form  the  covariant derivative $\nabla_J$.
\\
\textbf{2}. The constructions that follow require     finite dimensional moduli spaces  of   "tamed" smooth maps 
  from the infinite cylinder $\bR\times S^1$
  to $W$ and which converge to  the given periodic orbit  $\bx$, respectively to the periodic orbit  $\by$ when one goes to   either end of the cylinder.
This only turns out to be possible if $J$  is  general enough. In particular   one cannot expect such a  $J$ to be integrable.
 More precisely, let $\cU(\bx,\by)$ be the collection of smooth maps $u: \bR\times S^1\to \widehat W$ with the following properties:
 \begin{enumerate} [(i)]
\item $u(s,t)$ converges to $\bx(t)$, respectively to $\by(t)$ if $s\to -\infty$, respectively  $s\to \infty$;
\item $\displaystyle \del s  u + J\comp \del t u=\nabla_J H $.
\end{enumerate}
 The set $\cU(\bx,\by)$ admits an $\bR$ action induced by  the action which sends $\lambda\in \bR$ to $u(s,t)\mapsto u(s,t+\lambda)$.

 \begin{thm} Assume that $ \cz{H,\bx}> \cz{H,\by}$ and $J$ is sufficiently general.  
 Then the quotient $\cM(\bx,\by)= \cU(\bx,\by)/\bR$ is a   (non-empty)
  finite dimensional topological manifold which can be compactified to
 an oriented  smooth manifold with corners. Furthermore, $\dim \cM(\bx,\by)= \cz{H,\bx}- \cz{H,\by}-1$. In particular, if $\cz{H,\bx}- \cz{H,\by}-1=0$, $\cM(\bx,\by)$
 is a finite set of points. In that case, the orientation gives each of the points a sign. 
 \end{thm}
 
 \noindent \textbf{3}. This clearly suggests to define the operator $\partial: \fcc k H \to \fcc {k+1}  H$ by setting
 \[
\partial  \by   := \sum_{\bm\in \cM(\bx,\by)}  \text{sign} (\bm)\cdot   \bx, \quad \cz{H,\bx}=n-k+1 ,\, \cz{H,\by}= n-k .
 \]
 One can show that $\partial \comp\partial =0$ so that $\fcc * H$ becomes a cochain complex 
 whose cohomology groups, the \textbf{\emph{Floer cohomology groups}} \index{Floer!cohomology}are denoted\label{page:FloerCoh}
 \[
 \fcoh k {\widehat W,H} =\frac{ \ker (\partial: \fcc k H \to \fcc {k+1} H) } {\im  ( \partial: \fcc {k-1} H \to \fcc {k } H)}.
  \] 
 \textbf{4}. 
 Finally, to arrive at symplectic cohomology, one orders  the set $\cH$ of   admissible Hamiltonians as follows.   Recall that every $H\in \cH$
 looks on the  cylindrical end  like $ar+b$ for some $a>0$. 
 Choose $a$ to be a non-period  and denote such a Hamiltonian by $H^a$. The  order on $\cH$ is induced by the real number $a$. The group
 \label{page:SH<}
 \[
  \sh k {\widehat W}^{<a}= \fcoh k {\widehat W,H^a}.
  \]
    takes care of periodic orbits having periods $<a$ and one can show that it is essentially independent of the choice of $H $ as long as $a$ is fixed.
  \par
  If   $a \le b$  there is a chain map from $\fc k  {H^a}$ to $ \fc k {H^b}$ defined as follows.
  Pick $\bx$, $\by$  $1$-periodic with respect to $H^a$, respectively $H^b$. Choose a non-decreasing family of
admissible Hamiltonians $\set{H_s}$, $s\in \bR $,  which joins  $H^a$ to $H^b$ and a family $J_s$ of almost complex structures. Let $\cB(\bx,\by)$ be the space
  of smooth maps $u: \bR\times S^1\to \widehat W$ with the following properties:
 \begin{enumerate} [(i)]
\item $u(s,t)$ converges to $\bx(t)$, respectively to $\by(t)$ if $s\to -\infty$, respectively to $\infty$;
\item $\displaystyle \del s  u + J_s \comp \del t u=\nabla H_s$.
\end{enumerate}
This set has an $\bR$-action (as before) with finite quotient $\cB(\bx,\by)/\bR=\cN(\bx,\by)$ 
which allows to  define a   chain map which on $\bx\in \cP(H^a)$ of index $n-k$  assigns a chain on periodic orbits for $H^b$ of the same index: 
 \[
\phi_{ab}  (\bx )  := \sum_{\by \in \cP(H^b) }   \# \cN(\bx, \by)\cdot   \by.  
 \]
In cohomology it induces a homomorphism 
  $H(\phi_{ab}):  \fcoh k {\widehat W,H^a} \to \fcoh k {\widehat W,H^b}$, the \textbf{\emph{transfer map}}.
Passing to the  direct   limit then defines symplectic cohomology:\index{symplectic!cohomology}\label{page:SH}
  \[
  \sh k W:= \varinjlim_{a}  \sh k {\widehat W}^{<a} .
  \]
 It can be shown that    $ \sh k W$ is independent of  all choices and hence it
is a symplectic invariant of $W$. If $b_1(W)=0$ is turns out to be also an invariant of 
  $\widehat W$.    Note however that a priori non-zero groups $ \sh k W$ may occur for all integral 
values of $k$ because this is true for  the Conley--Zehnder index.
Finally, as for  ordinary cohomology, one can define a graded product structure on
the direct sum of these groups, resulting in a $\bZ$-graded ring. 

Summarizing the above discussion, one has:

\begin{thm}{\protect{\cite[\S 2.5]{seidel}, \cite{viterbo}}} 
\begin{enumerate}
\item The $\bZ$-graded  cohomology, $ \sh * W$  is a symplectic invariant of $W$  and, if $b_1(W)=0$, also  of $\widehat W$.
\item $\sh  * W$ has an associative graded  product giving it a graded ring structure. This  ring structure is   a symplectic invariant.
\end{enumerate}\label{thm:fund}
\end{thm}

\begin{rmk}
\label{rmk:Alpha'sChoice}
Recall that if $\omega$ is the symplectic form on $W$ and $Y$ is a Liouvlille field, the
contact form on the boundary $S=\partial W$ is given by $\alpha=\iota_Y\omega$. 
By Remark~\ref{rmk:BndLiouvDom}   two  contact forms on the boundary of $W$
are related by an isotopy and hence are contactomorphic. 
So, if needed, one may choose a particular contact form
without altering the symplectic cohomology of $W$.
This will be relevant for  the discussion in  Section~\ref{ssec:OnContactInv} . 
\end{rmk}

\begin{exmples} 
\textbf{1.}   
I present here a simplified version of the computation in   \cite[\S 3.2]{survey}  which shows that  
complex unit balls $B_n$  in $\bC^n$ for all $n\ge 1$ have  zero symplectic cohomology.   The argument exhibits  the
essential role of   the transfer maps.
\par
One considers $\bC^n$ as the symplectic completion of $B_n$. For simplicity, consider the symplectization of
the  unit sphere $S^{2n-1}\subset \bC^n$ as the complement of the ball $  \|z\|^2 \le 1-\delta$, $0<\delta<1$, 
 with coordinate $r= \|z\|^2$.
Then the Hamiltonian vector field of the function $H_{a,b}= 2a r+ b$ is $2 \pi a X$, $X$ the Reeb vector field for 
the contact structure on the sphere
of radius $r$. As in Example~\ref{exm:CZindex}.{\bf 2}, the periodic orbits of the linearized flow are given by the diagonal complex 
matrix  $\text{diag} (e^{2 \pi \ii a  t},\dots, e^{2 \pi \ii  a t})$ for various values of $a$. Assuming that $k  <a<  k+1  $ 
so that $\lfloor a\rfloor =k$,  Example~\ref{exm:Maslov}  gives  Conley--Zehnder index of this orbit $\gamma_a$ as $\cz {\gamma_a}=
 n(2k+1)$. So for each choice of $a\in(k,k+1)$ this gives one generator for   
 $\fcc  {-2kn}   {H^a}$ due to the convention of Eqn.~\eqref{eqn:cfk}. For this choice of $a$ the groups $\fcc  {j}   {H^a}$, $j\not=-2kn$ 
 all vanish so that $\fcoh   {-2kn} {H^a}$ is $1$-dimensional.
 If $b-a>1$ the  transfer maps $\fcoh {-2kn}   {H^a} \to \fcoh {-2kn}   {H^b} $   obviously do not preserve the 
cohomological degrees since the latter depend on the choice of the parameter $a$. Hence in the limit the resulting
  symplectic cohomology groups all vanish.
 \\
\textbf{2.} The arguments in \cite[\S 6b]{seidel2} show that for this example  the symplectic cohomology
is determined by the contact structure of the boundary, namely the symplectic cohomology of a
any  Liouville domain of (real) dimension $\ge 4$   vanishes if its boundary is contactomorphic to the 
standard contact sphere. 
 \\
\textbf{3}. In loc.\ cit. \S 3.a, one finds a few more  examples, including that of an open  Riemann surface of genus $g\ge 1$ with one boundary component. 
This is shown to have symplectic cohomology for infinitely many degrees: it has $2g $  generators in degrees $0$, one generator in degrees $-1,2$ and  in degrees $4kg -2k-1$, $k=1,2,\dots$.
\end{exmples}

\subsection{Symplectic cohomology as a contact invariant} \label{ssec:OnContactInv}

It is not clear that symplectic cohomology for $W$ leads to a contact invariant of the boundary $S$, but as explained below,
a  certain convexity condition does lead to contact invariants.

Observe first that  the periodic orbits in the interior of $W$ give rise to a  subcomplex $\dfcc  - H *$  of  
   the so-called \textbf{\emph{negative symplectic chains}}   which reduces to the Morse chain complex for $H$ on $W$. 
Indeed,  these orbits are constant, and so  precisely give the critical points  $\bx$ of the function $H$ and 
by ~\S~\ref{ssec:MorseHom} the subcomplex $\dfcc  - H *$    is  indeed the Morse complex for $H$
where   the Morse index  of  $H$ at $\bx$ equals $n- \cz \bx$,  $2n=\dim W$.  See also    \cite[Lemma 2.1]{frau} and  formula (8) in loc.\ cit.

And so, comparing with the indexing \eqref{eqn:cfk} for the Floer complex, taking the limit,   
and using the (dual of the) isomorphism \eqref{eqn:MorseIsSing},
one obtains  the usual $k$-th singular cohomology group of $W$:\label{page:shPM} 
\begin{equation}
\label{eqn:NegSh}
  \varinjlim_{H \in \cH} H^k (\dfcc *H -) = H^k(W).
\end{equation} 
So there is an induced homomorphism  $  \sh  k W \to H^k (W)$ which  measures de difference between Morse cohomology 
and symplectic cohomology.   The quotient complex 
\[
\dfcc  + H *  = \dfcc {}  H */ \dfcc - H* 
\]
is called the
complex of the \textbf{\emph{positive symplectic chains}}.   \index{symplectic!positive --- chains}
\index{positive symplectic chains}
 
After  taking  the limit for admissable Hamiltonians  the  resulting cohomology, 
the \textbf{\emph{positive symplectic cohomology}} \index{symplectic!positive --- cohomology}
 \index{positive symplectic!cohomology}
 \[
 \dsh {k}  W + = \varinjlim_{H \in \cH} H^k(\dfcc  * H +)    , 
 \]
is independent of the choice of admissible Hamiltonians, but a priori 
depends on the contact form for the contact structure on $S$. For the full symplectic cohomology this is not the
case as remarked before (Remark~\ref{rmk:Alpha'sChoice}).
K. Cieliebak and A. Oancea  show in \cite[\S 9]{cieloan} that this is 
also true for positive symplectic cohomology. Surprisingly, if   the contact structure on $\partial W$ 
satisfies the following convexity condition,
  it does not depend on the filling.
  
\begin{dfn}   \label{dfn:IP} If the Conley--Zehder index for all period orbits $\bx$ of the Reeb flow for \emph{\textbf{some}}
contact form $\alpha$  for  the contact structure $\xi$ on $S=\partial W$ satisfies  the inequality
\begin{equation}
\label{eqn:indpos}\cz \bx   + n-3 >0, \quad 2n=\dim_\bR W,
\end{equation} 
  the Reeb flow   is called \textbf{\emph{dynamically convex}} and the contact manifold $(S,\xi)$,  
  is called \textbf{\emph{index positive with respect to $\alpha$}}.\index{dynamically convex}\index{index positive}
\end{dfn}

 The criterion then reads:
 \begin{crit}[\protect{\cite[Thm. 9.17]{cieloan}}]     If  the contact structure on $S $ is index positive, then
 $\dsh {* }  W + $ is independent of the symplectic filling $W$, that is, it is a contact invariant of $S$.  \label{crit:Cieloan}
 \end{crit}

From this it follows that
the full symplectic cohomology groups in negative degrees give  contact invariants for $S$ since 
$ \dsh k W + \mapright{\simeq}  \sh k  W $ for $k<0$. This follows directly from the long exact sequence in cohomology  
for the pair $(\dfcc  {}H *,\dfcc {-} H *)$ which reads 
\begin{equation}
\label{eqn:onSympCoh}
\cdots \to H^ {*-1} (W)\to  \dsh {* }  W  + \to    \sh * W \to H ^*  (W)     \to \cdots ,
\end{equation} 
using  that $W$ has no cohomology in degrees $<0$.

  By results obtained by McLean on the contribution of the   Reeb orbits on the link of a cDV-singularity 
  discussed in the next section
  one can draw   more detailed conclusions in that situation. See   Proposition~\ref{prop:SHforcDV} in the Section~\ref{sec:Mclean}.
    
 \section{McLean's results}
 \label{sec:Mclean}

 \subsection*{1} The first of McLean's result is a topological characterization of numerically Gorenstein singularities
 which are not necessariily  \ihs. 
 Recall from   \S~\ref{sec:onsings} that $(X,x)$ is numerically Gorenstein
 if and only if for some (and hence all) resolutions $ Y\to X$ with smooth normal crossing
 exceptional divisors $E_i$  and  for some set of rational numbers $a_i$ the $\bQ$-divisor 
 $K_Y -\sum_i a_iE_i$ is numerically trivial. 

As explained in  Section~\ref{sec:strongfill},  the link  $\lnk{X,x}$ has a suitable symplectic filling which is a Liouville domain
and so the Reeb flow exists near the link, and its   periodic orbits can be investigated.
Defining  their Conley--Zehnder indices require these orbits to be contractible which in general is not the case.
However, the condition that $(X,x)$ be  numerically Gorenstein allows to weaken this condition in view of the
following result of McLean: 
  
 \begin{lemma}[\protect{ \cite[Lemma 3.3]{mclean}}]   An isolated normal singularity $(X,x)$ is numerically Gorenstein 
 if and only if $c_1(TX|_{\lnk { X,x}})  \in H^2(\lnk { X,x})$ is torsion. In particular this is true for
 canonical singularities such as cDV-singularities.
  \end{lemma}
  
 Using this, one  can modifiy  the usual Conley--Zehder index
 as explained in \cite[\S 4.1]{mclean}.
Suppose that  for instance   $m\cdot c_1(TX|_{\lnk { X,x}}) =0$ which is the case
for index $m$ singularities. This has the effect that 
all the  Conley--Zehder indices    belong to $\displaystyle \frac 1m \bZ$.  
In particular, for  an index $1$ singularity,  one gets integer invariants, as is the case    for isolated cDV-singularities.

  \subsection*{2}  McLean's  second result implies that the contact structure on the link determines
 whether the singularity is canonical or terminal, as will be explained  
 after having given the relevant  definitions. 
 
 First, recall the following \textbf{algebra-geometric notion} which generalizes  the one discussed  in Remark~\ref{rmk:OnMinDiscr}:      
 
  \begin{dfn} The \textbf{\emph{minimal discrepancy $\text{md}(X,x)$\index{discrepancy!minimal ---}\label{page:md} 
  of $(X,x)$}}  equals the infimum of $\min(a_i)$  taken over all non-trivial 
   \emph{divisorial}  resolutions $ Y\to X$ of $(X,x)$ with center $x$\footnote{In particular,  $\sigma$ is not the identity.}  for which  $K_Y -\sum_i a_iE_i$ is numerically trivial.\footnote{Recall  also (Remark~\ref{rmk:OnMinDiscr})  that the minimal discrepancy is attained for any one resolution of the singularity.}
   \end{dfn}
   
The surface case has been dealt with in Example~\ref{ex:surfsing}:   the only canonical singularities are the A-D-E singularities (with minimal discrepancy equal to $0$) and
the terminal singularities are the smooth points (with minimal discrepancy equal to $1$). All other surface singularities have (possibly fractional) minimal discrepancies $<0$.
 
As shown in Remark~\ref{rmk:OnMinDiscr},    smooth points have  minimal discrepancy equal to $\dim X-1$. 
A  well-known  conjecture states that, as for surfaces, the converse holds:
\begin{conj}[\protect{Shokurov's conjecture \cite[Conjecture 2]{shokufano}}] \index{Shokurov's conjecture}
A normal isolated Gorenstein singularity is smooth  if  and only if 
$\text{md}(X,x)=\dim (X) -1$. 
\end{conj}
  In addition to  surfaces this conjecture  has also been shown  in dimension $3$ (see  Markushevich  (\cite{MinDisc}) for index $1$,
and Y. Kawamata (appendix to \cite{shoku}) for higher index). 

 For an isolated threefold singularity,  by \cite[Ch. 17, Prop. 1.8]{flips} \footnote{Koll\'ar considers also blow-ups in centers that strictly contain the singular locus, while in the present situation only points are blown up. This explains the different upper bound.},  two possibilities occur:
  \begin{enumerate}
\item either  all  $a_j\ge -1$, and  then  $\text{md}(X,x)\in [-1,2]$ and the infimum is a minimum;
\item alternatively, $\text{md}(X,x)=-\infty$.
\end{enumerate}
 
For canonical  singularities  $\text{md}(X,x) \ge 0$ and for terminal singularities of index $1$ one has $\text{md}(X,x) \ge 1$.
As a consequence of  the  validity of the smoothness criterion in terms of the minimal discrepancy  in dimension $3$,  one has:

  \begin{prop} 
 \label{prop:onMindiscr} 
 Let $(X,x)$ be  a  $3$-dimensional  normal    (non-smooth)  terminal  singularity 
 of index $1$, then $   \text{md}(X,x)=1 $.
 \end{prop}
 \medskip
  \par 
 
Next, I shall explain the central \textbf{symplectic notions}. Consider the 
symplectic filling $W$ of the link $(S, \xi)=(\lnk{X,x},\xi)$ of the   isolated normal  singularity
$(X,x)$ with the contact structure as described in Section~\ref{sec:strongfill}. 
The invariant of $(S,\xi)$ considered by  McLean 
is constructed from   the collection of  all $1$-periodic orbits  $\bx:[0,1] \to S$ of  the Reeb flow. 
Recall that the  Reeb flow is constructed
from the contact form $\alpha$ and different contact forms defining the same contact structure may lead to 
different Reeb flows and so the 
  Conley--Zehnder index $\cz \bx$ of  an  orbit  of the Reeb flow, as given in Definition~\ref{dfn:Maslov}, also depends on the
   chosen contact form $\alpha$.  
  So a  contact invariant of the link of $(X,x)$ has to take all  of these into account, which
  leads to the following invariant introduced by McLean, the highest minimal index:

 \begin{dfn} Set $ \text{i}(\bx)  := \cz \bx + (n-3) $. Then  the \textbf{\emph{minimal index}} of the Reeb flow for $\alpha$  is defined as\index{index!(highest) minimal --}
\[
 \text{mi}(\lnk {X,x} , \alpha)         := \inf_ \bx  \text{i} (\bx).
 \]
Then the \textbf{\emph{highest minimal index}} is defined as \label{page:mi}
\begin{align*} 
 \text{hmi}(\lnk {X,x}, \xi)   &:= \sup_\alpha  \text{mi} (\lnk {X,x} ,\alpha ),  
\end{align*} 
where the  supremum is taken over all   contact forms $\alpha$ with  $\ker \alpha=\xi$  (but one needs to
preserve the orientation of $TX|_\xi$).    

\end{dfn}
 \begin{rmk} \label{rmk:CZ}
(a)   If one chooses the form $\alpha$ such that  $ \text{mi}(\alpha )= \text{hmi}(\lnk {X,x}, \xi) $, 
  the definition of the minimal index and the highest minimal index  implies that   for all orbits
   $\bx$ of the associated Reeb flow the Conley--Zehder index satisfies  the inequality
\[
\cz \bx   + n-3  \ge    \text{hmi}(\lnk{X,x}, \xi) . 
 \]
Consequently,   recalling \eqref{eqn:indpos},  $\text{hmi} (\lnk{X,x}, \xi) >0$ implies that  the  contact manifold is index positive with respect to $\alpha$ in the sense of Definition~\ref{dfn:IP}.
\\
(b)   Orbits for which  linearized return map $D$ has eigenvalue $1$ are sometimes excluded, but
in deformation arguments these come up and in order to make the highest minimal index well behaved,
  one subtracts  the  correction term $ \delta =\half \dim \ker (D-\id)$ in  the definition of $ \text{i}(\bx)$. 
Since  $2\delta$ counts the multiplicity of the eigenvalue $1$ of  the linear symplectic automorphism
$\tilde D|_\xi$ on a vector space of dimension $2n-2$ which is   a symplectic subpace, it has indeed
 even dimension so that
$0\le  \delta\le n-1$ with equality $ \delta=n-1$ if and only if $  D|_\xi=\id$.
 \end{rmk}
   \medskip

  \noindent McLean's  main result \cite[Thm. 1.1]{mclean} reads as follows:\index{McLean's theorem}
 
 \begin{thm} Let $(X,x)$ be a normal   isolated numerically $\bQ$-Gorenstein
 singularity  with $H^1(\lnk{X,x},\bQ) = 0$. Then:
 \begin{itemize}
\item  if $\text{md}(X,x) \ge 0 $, then $\text{hmi}(\lnk {X,x}, \xi)= 2\text{md}(X,x)$;
\item  if $\text{md}(X,x) < 0 $, then $\text{hmi}(\lnk{X,x}, \xi)<0$.
\end{itemize}
In particular, the contact structure on the link determines
 whether the singularity is canonical ($\text{md}(X,x) \ge 0 $) or terminal ($\text{md}(X,x) > 0 $). \label{thm:mclean}
\end{thm}
 
 In view of  Remark~\ref{rmk:CZ}(a)  this implies:
 \begin{corr}  \label{cor:terminalcar} Any    a normal   isolated numerically $\bQ$-Gorenstein
 singularity  $(X,x)$ with $H^1(\lnk {X,x},\bQ) = 0$ (e.g.  a  cDV-singularity) 
has an   index positive link  if and only if it is terminal.
 \end{corr}

   \begin{exmple}[\textbf{The standard sphere $S^{2n-1}$ with  contact structure $\xi$}]  \label{exm:reebBis2}\index{unit sphere}
 Example~\ref{exm:reebBis} makes it plausible that  $\text{hmi}( \xi)= 2(n-1)$.
 Since it is the link of a  smooth point in an $n$-dimensional complex algebraic variety $X$
 for which  $\text{md}(X,\mathbf 0)= n-1$, Theorem~\ref{thm:mclean} shows that indeed $\text{hmi}( \xi)=2(n-1)$.
 This  seems quite difficult   to show directly, since the Reeb orbits for $f\alpha$  
 are difficult to control (here $\alpha$ is  the standard contact form and $f:S^{2n-1}\to \bR$ is an everywhere positive smooth function).
 \end{exmple}

  Example~\ref{exm:reebBis2} together with Theorem~\ref{thm:mclean}
  suggest   a conjectural Mumford-type result (cf. Theorem~\ref{MumfThm}). This is    true  dimension $3$:

 \begin{corr} \label{cor:S5IffSmooth} If Shokurov's conjecture holds, then a normal Gorenstein singularity $(X,x)$ is smooth 
 if and only if its link is contactomorphic to the standard sphere $S^{2n-1}\subset \bC^{n}$.
 So  this is in particular true for $n=2$ and $n=3$.
 \end{corr}
 
 Coming back to surfaces, this gives a symplectic proof of Mumford's result stated here as Theorem~\ref{thm:Mumford}:
 
 \begin{corr}[\protect{\cite[p. 508]{mclean}}]A normal surface germ  $(X,x)$ is smooth if and only if its link is simply-connected, i.e., homeomorphic to $S^3$.
 \end{corr}
 \begin{proof} Suppose that the link is homeomorphic to $S^3$. Then (by the now proven Poincar\'e conjecture) it is also diffeomorphic to $S^3$. Using a desingularization of $(X,x)$
 Theorem~\ref{thm:OnSympFilling} shows that  the link is strongly Milnor fillable and so by the classification of contact structures on $S^3$ as discussed in Example~\ref{exmpl:BasSymp}.(1)  the link $S^3$ can be  assumed to have   its standard contact structure.
Then, since Shokurov's conjecture holds here,  Corollary~\ref{cor:S5IffSmooth} implies that $(X,x)$ is a smooth germ. 
 \end{proof}

\subsection*{3}  Finally I shall explain some further consequences of Theorem~\ref{thm:mclean} for the 
\textbf{contact geometry of the link}  of an isolated 3-dimensional  terminal  singularity of index $1$. 
 The link of such singularities is index positive, but since 
 by Proposition~\ref{prop:onMindiscr}   such a  singularity  has 
   minimal discrepancy    $1$,   McLean's theorem together with the validity of Shokurov's conjecture in this case
   gives a stronger inequality for the Conley--Zehnder indices:
  
   \begin{corr} \label{cor:cDV} For a  $3$-dimensional isolated terminal   singularity of index $1$ (so in particular  for  a cDV-singularity) 
   one has   $\text{hmi}(\lnk{X,x}, \xi)=2$ and so $\cz \bx \ge 2$ for all periodic orbits $\bx$ of the Reeb flow  for the contact form which realizes  the highest minimal index. 
   \end{corr}

 This result has consequences for the symplectic cohomology of Milnor fibers of a cDV-singularity: 

  \begin{prop} \label{prop:SHforcDV} For a   cDV-singularity  $(X,x)$ with   Milnor number $\mu $, one has
    \begin{itemize}
    \item $\dsh k {\mf{X,x}} +$ is a contact invariant; 
    \item $\dsh k {\mf{X,x}} +=\sh k{\mf{X,x} }$ for $k<0$ and $\dsh k  {\mf{X,x}} + =0$ for $k\ge 2$;
    \item $
    \sh k {\mf{X,x}}  =      0                         \text{ for }  k= 2, k\ge  4$; 
 \item   $  \rank (\sh 3{\mf{X,x} }) =\mu$.
\end{itemize}
  \end{prop}
  
  \begin{proof}  Since  the link of a cDV-singularity is index positive,  Criterion~\ref{crit:Cieloan} 
  states that $\dsh k {\mf{X,x}} +$ is a contact invariant for the contact structure on the link. 
  By Remark~\ref{rmk:Alpha'sChoice}
 one may then  assume
  that the contact form on the link is the one realizing $\text{hmi}(\lnk{X,x})$.  
    Note that
  any  Reeb orbit  with $\cz\bx=3- \ell$  on the link contributes   only to positive
 symplectic  cohomology in degree     $\ell$. Corollary~\ref{cor:cDV} states that $\cz \bx\ge 2$ for all periodic 
 Reeb orbits for the chosen contact form  for   the link  and so
    $\dsh k  {\mf{X,x}} + =0$ for $k\ge 2$. Now invoke the long exact sequence \eqref{eqn:onSympCoh} to deduce that 
    $\sh k {\mf{X,x}}  =      0  $   for $k= 2, k\ge  4$, $\sh 3 {\mf{X,x}}\simeq H^3({\mf{X,x}} ) $, where we
  use that   the Milnor fiber of a three-dimensional  \ihs\  having the homotopy type of a $3$-sphere, 
  only has cohomology in ranks $0$ and $3$.
  \end{proof}

 \chapter[Matrix factorizations and Hochschild cohomology]{Matrix factorizations and Hochschild cohomology}
   \label{lect:mf}

\section*{Introduction}

Since direct calculation of symplectic cohomology is often not possible, in recent years a roundabout way has been proposed which
gives a route to calculate symplectic cohomology for the Milnor fiber of several classes of  invertible matrix singularities. 
In outline this goes as follows. 

\begin{enumerate}[(a)]
\item One first applies the classical theory of matrix factorizations. One can assign
many matrix factorizations to  a given \ihs, but there is one
that plays a predominant role, its Koszul matrix factorization. All matrix factorizations come with  their cohomology. It should be considered as a first rough invariant. of the \ihs.
\item The next step is to reinterpret this invariant via  Hochschild cohomology. This is quite intricate and uses   the entire category of matrix factorizations of a given \ihs.  
Indeed, the supplementary structure  of dg-category 
allows to  apply the machinery of  categorical Hochschild cohomology.
\item For invertible matrix singularities  one  refines Koszul cohomology in such a way that  the extra symmetry of these singularities is reflected therein.
This step uses equivariant matrix factorizations. For many classes of invertible matrix singularities the resulting refined Hochschild cohomology has been calculated.
\item The categorical framework allows comparison with other categories associated to invertible matrix singularities which are rich enough so 
that Hochschild cohomology makes sense for them. These categories are relevant because their Hochschild cohomology is precisely the symplectic cohomology
of the Milnor fiber of that singularity.\footnote{Rather, of the  "dual" singularity associated to the transpose matrix.}
The  associated homotopy categories are conjecturally equivalent to 
the homotopy categories  of equivariant matrix factorizations (for any  given  invertible matrix singularity, or its "dual").
Since these  conjectures have been shown to hold  under easily verifiable conditions,   
the refined Hochschild cohomology in these cases -- for which  there is an explicit formula --  is thus the same as the symplectic cohomology.
 \end{enumerate}

In this, admittedly long  chapter  only steps (a)--(c) will be dealt with. Sections~\ref{sec:basic}  and \ref{sec:Koszmf} treat step   (a).
This basically reports on D.~Eisenbud's  original treatment ~\cite{eis} of matrix factorizations. The categorical reformulation of step (b) is due to    T. Dyckerhoff in   \cite{Dyck}.
The  main ideas  from loc.\ cit. will be explained  in Section~\ref{sec:MFasStabs}  after the preliminary Section~\ref{sec:dgcats}. 

The chapter ends with a brief summary the calculation of Hochschild cohomology   in the equivariant setting,  with the necessary details in the
case of invertible matrix singularities. This will be further detailed in \chaptername~\ref{lect:Hochschild}.
  
 Step (d) will be briefly explained  in a later chapter; see in particular Section~\ref{sec:SympIsHH}. 
   
 \medskip
 
  \begin{small}
Up to now \ihs s have been defined by polynomials $w\in \bC[x_1,\dots,x_m]$ vanishing in  the origin, but
only the germ of $w$ at the origin is of importance. The maximal ideal $(x_1,\dots,x_m)$ corresponding to
the origin thus plays a special role. In the algebraic   setting of a polynomial ring 
$R=k[x_1,\dots,x_m]$ over any algebraically closed field $k$ the 
maximal ideal $\m=(x_1,\dots,x_m)$  is called the irrelevant ideal. 
 $S=R/w$  is  the ring of functions on  the \ihs\ given by $\set{w=0}$. 
\par
Much of what follows can be rephrased in terms of 
local rings $(R,\m)$, where $w\in \m$. For instance, the $\m$-adic completion of  $R=k[x_1,\dots,x_m]$  is  the ring 
$\widehat R:= k[\![ x_1,\dots,x_m]\!]$ of formal power series in $x_1,\dots,x_m$. 
It is a regular local ring in which the irrelevant ideal now becomes the maximal ideal.
If $k=\bC$, and using the usual (classical) topology), one can work in the local ring $\cO_{\mathbf 0}$, of germs at $\mathbf 0$ 
of holomorphic functions on $\bC^m$ in which the irrelevant ideal is    the maximal ideal and $w=0$ then defines  an \ihs.
The ring $\cO_{\mathbf 0}$ is not only local, it is   the  subring of $\bC[\![ x_1,\dots,x_m]\!]$ consisting of
convergent power series.
\end{small}

\section{Basics  on matrix factorizations}
\label{sec:basic}

\subsection{Matrix factorizations are related to  singularities}  \label{ssec:mfsings}

Consider  an   \ihs\  defined by a polynomial  $w\in R= k[x_1,\dots,x_m]$   with an  isolated singularity at the origin.  
 Then $0\to R\mapright{\cdot w} R \to 
S:=R/(f)$ is a minimal free  $R$-resolution of  $S $ as an $R$-module. Setting $d_1=w\cdot \id_R$ and
$d_0=\id_R$, one has $d_1\comp d_0= d_0\comp d_1= w\cdot \id$.  This is an example of   a matrix factorization of $w$:

\begin{dfn} Let $(R,\m)$ be as above  and  let $w\in \m$. Then
\begin{enumerate}[(i)]\index{matrix factorization}
\item A \textbf{\emph{matrix factorization}}   of $w$ is a free $\bZ/2$-graded $R$-module $X$ of finite rank equipped with an odd degree $R$-linear map $d:X\to X$ such that $d\comp d=w\cdot \id$.
In other words, $  X^0\mapright{d_0} X^1$ and $  X^1\mapright{d_1} X^0$ with $d_1\comp d_0 = d_1\comp d_0=w\cdot \id$ (as above).

\item A morphism $\phi: (X ,d)\to (Y ,d')$ of matrix factorizations is a $\bZ/2$-graded map $\phi$
such that $\phi \comp d'=d\comp \phi $.
\end{enumerate}
\end{dfn}

Choosing  a basis of the free $R$ modules $X^0$ and  $X^1$ makes clear where the terminology originates from:
a matrix factorization of $w\in R$  is  represented by a matrix of rank $2k$, $k=\rank X^0=\rank X^1$, 
\[
(A,B):=\begin{pmatrix}
0 & A \\ B&0
\end{pmatrix} \quad A,B \in R^{k\times k},\, \, \text{ such that }AB=BA= w\,  I_k.
\]

\begin{exmple} \label{ex:BasicMatFact}
\begin{enumerate}[(i)]
\item The \textbf{\emph{zero matrix factorization}}: $X^0= X_1=R$, $d_0=d_1=0$, is the matrix factorization of $0\in R$.
Note that there is no matrix factorization of $1$ since $1\not\in \m$.
 
\item Consider the double point $xy=0$ in $\bC^2$. Setting $R=\bC[ \![x,y]\!  ]$, $S=R/(xy)$, the double point  admits a matrix-factorization
\[
(x,y)=\left( \xymatrix{
0 \ar[r] & R \ar@/^/[r]^{x}  &  R\ar@/^/@{..>}[l]^{y} \ar[r]_{p\quad } & R /(xy) 
}\right).
\]
The operation of "adding a double point" to an \ihs\ at $\mathbf 0$    given by a polynomial $w\in \bC[x_1,\dots,x_{m}]$
consists of replacing $w$ by $w+xy\in  R=\bC[ x_1,\dots,x_m,x,y] $. If $(A,B)$ gives a matrix factorization for $w$, then
\[
\left(\begin{pmatrix}
  x I_k & A \\B &-y I_k
\end{pmatrix}, \begin{pmatrix} y I_k & A\\ B& -x I_k
 \end{pmatrix}\right) = \begin{pmatrix}
 0 &    0   &    x I_k & A        \\
 0 &     0   &  B     &-y I_k       \\
y I_k & A    &       0 &0              \\ 
  B& -x I_k &0 &0  
 \end{pmatrix}  \in R^{2k,2k}
 \]
gives a matrix factorization of $w+xy $.
\end{enumerate}

\end{exmple}

\subsection{A geometric incarnation of matrix factorizations} \label{ssec:matfGeom}

Ultimately one wants to apply this to \ihs's, geometric objects, and indeed, a  geometric flavor  can be given to the above construction, if instead of free $R$-modules of finite rank, 
one considers vector bundles  over  affine  $n$-space $\bA^n$, or, more precisely,
locally free $\cO_{\bA^n}$-modules. So one replaces $R=\bC[x_1,\dots,  x_n]$ with the sheaf of regular functions on
the space $\spec R=\bA^n$. 

A  matrix factorization for  a degree $d$ polynomial $w$ is a pair $(\cE^\bullet, d)$ consisting of a $\bZ/2$-graded vector bundle
$\cE^\bullet= \cE^0\oplus \cE^1$ and two  vector bundle morphisms\footnote{Degree $d$ polynomials are sections of the line bundle $\cO_{\bA^n}(d)$, 
and as   usual, for any vector bundle $\cE$, one sets 
$\cE(d):= \cE\otimes_{\bA^n}(d)$.}  $d_0:\cE^0\to \cE^1(d) $ and $ d_1:\cE^1\to \cE^0  $  such that $d_0\comp d_1 =w\cdot \id$
and $d_1\comp d_0= w\cdot \id \otimes 1$.
Clearly, the preceding incarnation is the special case where the vector bundles are direct sums of line bundles (these are all isomorphic to $\cO_{\bA^n}(d)$)
since morphisms between such vector bundles are given by matrices of polynomials.  

More generally, replacing in he above definition $\bA^n$ with any variety (or a scheme)  $X$ and $w$ 
by a section $w$ of  a line bundle $\cL$ on  $X$, one obtains the definition of a matrix factorization of $w$.

\subsection{Relation with maximal Cohen--Macaulay modules} \label{ssec:MCmods}
 
First recall two basic concepts for finitely generated modules $M$ over a regular local ring $(R,\m)$.
A sequence $(x_1,\dots,x_{k })$, $x_i\in R$, $i=1,\dots,k$,
is called   $M$-regular if $x_i$ is a nonzero divisor in $M/(x_1,\dots,x_{i-1})M  $ for all $i=1,\dots,k$.
The projective dimension (abbreviated below as "pd") and the depth of $M$ are then defined as follows:
\begin{align*}
 \text{pd} _R(M) & =  \text{ length of a minimal $R$-free resolution of  } M,\\
 \text{depth}(M)  &=  \text{ maximal length of an $M$-regular sequence}.
\end{align*} 
These are related as follows:

\begin{thm*}[Auslander--Buchsbaum  \protect{\cite[Thm. 19.9]{CommAlg}}] If $R$ is a regular local ring, then $\text{\rm pd}_R(M)=\dim(R)-  \text{\rm depth}(M)$.
\end{thm*}

Suppose that the $R$-module $M$  is an $S=R/wR$-module for some $w\in \m$ such
that  $\text{depth}(M)=\dim (S)=\dim(R)-1$, then the above theorem shows that $\text{\rm pd}_R(M)=1$.
The $S$-module $M$ is then called a \textbf{\emph{maximal Cohen--Macaulay $S$-module}}.\index{Cohen--Macaulay module}
By assumption there  is a minimal free $R$-resolution $ 0\to X^0 \mapright{d_0}  X^1 \to M$ 
and since $w\cdot M=0$ one can construct   $d_1: X^1\to X^0$  such  that $d_0\comp d_1=w\cdot \id$:
\[
\xymatrix{
0 \ar[r] & X^0  \ar@/^/[r]^{d_0}  &  X^1 \ar@/^/@{..>}[l]^{d_1} \ar[r]_{p} & M   }, \quad d_0\comp d_1=w\cdot \id.
\]
  Then from the injectivity of $d_0 $ one derives that also $d_1\comp d_0 = w\cdot \id$. It follows that $X^0$ and $X^1$ have the
same rank and so $d^0$ and $d^1$ are represented by square matrices of the same size.
So this shows:

\begin{lemma}  \label{lem:BasicMatFact}   Let $(R,\m)$ be a regular local ring, $w\in \m$ and    $S=R/(w)$,  
A maximal Cohen--Macaulay   module  $M$  over $S$ gives a canonical matrix factorization of $w$ over $R$, 
$ \xymatrix{
 X^0  \ar@/^/[r]^{d_0}  &  X^1 \ar@/^/@{..>}[l]^{d_1}   }$ with $\coker d_0= M$. \end{lemma}
  
\begin{rmk} \label{rmk:eisenbudtrick}
Note that $M$, considered as an $S$-module receives a $2$-periodic free resolution
\begin{equation*} 
\cdots\mapright{d_{k+1}}  \overline{X^k}\mapright{d_k} \overline{X^{k-1}}\mapright{d_{k+1}}\cdots \mapright{d_{ 1}} \overline{X^{0}} \to M\to 0,
\end{equation*}
where $ \overline{X^k} = X^0\otimes _RS$ for $k$ even,  $ \overline{X^k} = X^1\otimes _RS$ if $k$ is odd,
and the  $d_k$ are induced by $d_0$ and $d_1$. Note that this is a complex since $d_0\comp d_1= d_1\comp d_0=w\cdot \id$ is zero in $S$.   
  \end{rmk}

\section{Koszul matrix factorizations}
\label{sec:Koszmf}

The basic constructions leading to the  relevant matrix factorizations 
can be performed  over an arbitrary  commutative ring $R$ with a unit $1$. The point of departure is a
free  $R$-module  of  rank $k$ with given basis $\be_1,\dots  ,\be_k$ and an   $R$-module homomorphism
$\phi:N\to R$. This defines 
  a  \textbf{\emph{Koszul sequence }} \index{Koszul!sequence} after  
identifying  $\phi$ with the corresponding row-vector $  {\bff}=(f_1,\dots,f_k)$:\label{page:Koszul}
\[
N^\bullet(\bff):= \{0\to\Lambda^ k N \mapright{\delta_\bff}\Lambda^ {k-1} N\to\cdots\to\Lambda^ 2 N \mapright{\delta_\bff} N\},
\]
where $\delta_\bff$ is the $R$-linear map given on $\Lambda^j N$ by
\begin{equation}
\label{eqn:Koszres}
\delta_\bff (e_{i_1}\wedge\cdots\wedge \be_{i_j})= \sum_{m=1}^j f_{i_j} (-1)^{m-1} \be_{i_1}\wedge\cdots\widehat{e_{i_m}} \cdots \wedge \be_{i_j}.
\end{equation} 
The   derivations   actually do  not depend on the choice of a basis for $N$
and  $ \delta_\bff \comp \delta_\bff =0$. The indexing is chosen such  that the complex starts at
degree $-k$ and ends at degree $0$ with   $N$. Then the cohomology groups $H^m(N^\bullet(\bff))$ vanish for $m>0$ and $m< -k$.
If $\bff=f_1,\dots,f_k)$ is an $R$-regular sequence, only $H^0$ survives and equals the quotient ring $R/(f_1,\dots,f_k)$. In other words:

\begin{lemma}\label{lem:KoszFree}  Suppose that   $ \bff =(f_1,\dots,f_k)$ is an $R$-regular
sequence.
Then the  Koszul sequence $N^\bullet(\bff)$ is a resolution of  the $R$-module $ R/(f_1,\dots,f_k)$, that is,
$H^0(N^\bullet(\bff))= R/(f_1,\dots,f_k)$ and $H^i(N^\bullet(\bff))=0$ for $i\not=0$. 
\end{lemma}

There is also a dual version,  using a vector  in $N$ which one identifies with a column-vector  $  \trp \bg \in \oplus^k R $,  
\[
N^\bullet (\trp \bg):= \{ \Lambda^ k N \mapleft{d_\bg}\Lambda^ {k-1} N \leftarrow \cdots \leftarrow\wedge ^2 N  \mapleft{d_\bg} N  
\leftarrow 0\},
\]
where one reverses the arrows and defines 
\[
 d_\bg (y)= \bg\wedge y ,  \quad y\in \wedge^j N ,  \quad j=0,\ldots, k-1.
\]
 Considering  this as a cohomological complex starting at degree $-k$, there is only cohomology in degrees $-k,\dots, 0$.
If  $(g_1,\dots,g_n)$    is an $R$-regular sequence, then  only $H^0(N^\bullet(\bg))= R/(g_1,\dots,g_k)$ survives. Here one uses the isomorphisms 
$\Lambda^k N\simeq R$, and $\Lambda^ {k-1} N\simeq N$.

Observe also that 
\begin{align}
\label{eqn:homotopy}
d_\bg \comp \delta_\bff+ \delta_\bff \comp d_\bg=( \bff \cdot \trp \bg ) \cdot \id_R  
\end{align}  
which defines a homotopy operator.

\begin{exmple}  \label{ex:KoszIHS}
\begin{enumerate}[\quad(1)] 
\item Let $R=\bC[x_1,\dots,x_m]$   and let 
$N = R^m$.
 Assume  that $f_1,\dots,f_m$ is an $R$-regular sequence. Then 
$H^k(N^\bullet(\bff))=0$ if $k\not=0$ and $H^0(N^\bullet(\bff))= R/(f_1,\dots,f_m)$.
In particular, if $f=0$ has an \ihs\ at $\mathbf 0$, then the partial derivatives of $f$ form an $R$-regular sequence and   the
Jacobian ring reappears as the cohomology of the associated   Koszul complex:
\[
\jac f= H^0( N^\bullet(\bff)), \, N=Re_1\oplus Re_2\oplus\cdots\oplus  R e_m,\, \phi(e_j)=\del f {x_j}.
\]
\item  (See D.\ Eisenbud~\cite[Sect. 7]{eis})
Let there be given  a commutative regular local ring $R$ with an $R$-regular sequence $\bff=(f_1,\dots,f_{m})$
and let $w$ be an element in the ideal $I=(f_1,\dots,f_{m})$, say  $
 w= \bff\cdot  \trp \bg$  where  $\bg=(\cdots, g_j,\cdots)$ is row vector in $ Re_1\oplus\cdots\oplus Re_m $.
This free $R$-module receives  a $\bZ/2$-grading by declaring $|e_j|=1$, $j=1,\dots,m$, $|r|=0$ for $r\in R$. 
This induces a  degree on wedge-products by reducing  mod $2$. The   \textbf{\emph{Koszul matrix factorization}} of $w= \bff\cdot \trp \bg$ 
   is defined as  \index{Koszul!matrix factorization}\label{page:KoszMf}
 \begin{equation}
 \set{\bff,\bg} : =\left( \xymatrix{
0 \ar[r] & (\Lambda^\bullet N)^0  \ar@/^/[r]^{\delta_\bff+d_{ \bg}}  &  (\Lambda^\bullet N)^1\ar@/^/@{..>}[l]^{\delta_\bff+d_{ \bg}  } \ar[r]_{\bff} & 
R/wR   
}\right).  \label{eqn:KoszMF}
 \end{equation} 
 It is indeed
a matrix factorization for $w$ with $\coker{ \bff}= R/(f_1,\dots,f_{m})$, since  $(\delta_\bff+d_{\bg})(\delta_\bff+d_{ \bg})=d_\bg  \comp \delta_\bff + \delta_\bff \comp d_\bg  =   w\cdot \id$ by \eqref{eqn:homotopy}.     
  \end{enumerate}
\end{exmple}

 \section{Matrix factorizations form a dg-category}\label{sec:dgcats}

\subsection{Some basic constructions}
Fix a commutative ring $R$ with unit $1$. Recall that  a complex $(X ^\bullet, d^\bullet)$ of $R$-modules is a   $\bZ$-graded $R$-module  
equipped with a degree $1$ derivation $d^\bullet: X^\bullet \to X^{\bullet +1}$,  
A  homogeneous element $x\in X ^\bullet$ by definition belongs to some graded piece $X^p$, and 
$p=\deg x$ is called the degree of $x$.   

\begin{dfn} The  category $\underline{C}(R)$   of $R$-complexes has for its objects  complexes of $R$-modules.\label{page:complexes}
A  morphism $f:X^\bullet \to Y^\bullet$  in  $\underline{C}(R)$  is a degree-preserving $R$-linear map respecting the differentials. 
\end{dfn}
Note that there are also maps  $X^\bullet \to Y^{\bullet+d}$  between complexes  of any degree $d\in\bZ$. These give  
$\underline{C}(R)$ an enriched structure, that of  a dg-category to be defined below   in \S~\ref{ssec:DGcats}.
Let me observe  here only  that  $\underline{C}(R)$  is a  so-called tensor category, that is, 
one has a    tensor product  of two complexes  defined by 
 \[
 (X^\bullet\otimes Y^\bullet)^k = \bigoplus_{a+b=k} X^a\otimes Y^b ,\quad 
 d_{X^\bullet\otimes Yv}(x\otimes y) = d_Xx \otimes y+(-1)^ax\otimes d_Y y,\, x\in X^a, y\in Y^b. 
  \]
The tensor product $f\otimes g :X^\bullet\otimes Y^\bullet \to (X')^\bullet\otimes  (Y')^\bullet $   is defined as
\begin{equation}
f\otimes g(x \otimes y)= (-1)^{pq}  f(x)\otimes g(y),\quad \deg g =p, \, \deg y=q. \label{eqn:Koszul}
\end{equation} 
This sign-rule is known as the \textbf{\emph{Koszul sign convention}}.\index{Koszul!sign convention}

\subsection{Introducing dg-categories} \label{ssec:DGcats}

A category $\underline{A}$ is a \textbf{\emph{dg-category (over $R$)}} if \index{category!dg- ---}\label{page:dg}
\begin{enumerate}
 \item For  all objects $X,Y$ of $\underline{A}$ the set  $\underline{A}(X,Y)$ of morphisms from $X$ to $Y$  is a  complex  of $R$-modules, i.e., $\underline{A}(X,Y)$ is an object in $\underline{C}(R)$.
 Moreover, $d(\id_X)=0$, where $\id_X\in \underline{A}^0(X,X)$.
 \item The composition $\underline{A}(Z,Y)\otimes_R  \underline{A}(X,Y) \mapright{\mu_{X,Y,Z} }\underline{A}(X,Z)$ 
 (tensor product as graded  $R$-modules)  is a morphism of complexes of 
 $R $-modules, that is,  $d(f\otimes g)$ maps to $d(f\comp g)$. Moreover, the composition
   is associative, i.e., 
   \[
   \mu_{W,Y,Z} \, \comp\,  (\id_{A(Y,Z)}\otimes\mu_{W,X,Y})=
   \mu_{W,X,Z}\, \comp\, (\mu_{X,Y,Z}\otimes \id _{A(W,Z)}).
   \]   
\end{enumerate}
The associated \textbf{\emph{homotopy category}} $[\underline{A}]$ has the same objects as $\underline{A}$ but the   \label{page:homcat}
morphisms    are the homotopy classes  of   morphisms  in  $  \underline{A}(X,Y) $.

 \begin{exmple}  \label{ex:dgexs}
 \begin{enumerate}[\bf 1.,leftmargin=\parindent+1em]
 	\item A differential graded $R$-algebra $A$ is a $\bZ$-graded associative $R$-algebra $A=\oplus_{j\in\bZ}  	A^j$
  with \emph{algebra} derivations $d$ of degree $1$, that is,
 $d: A^j \to A^{j+1}$, $d\comp d=0$ and $d(a b)= da \cdot  b + (-1)^j a \cdot db$ for all $a\in A^j, b\in A$ (the \emph{\textbf{graded Leibniz rule}}).
 This can be made into a $dg$-category $\ul A$ with one object, $A$, viewed as an $R$-module, whose   morphism are  the elements of the  algebra  $A$. Its structure as a differential graded algebra makes $\ul A$  a dg-category.
 Note that for this to be a dg-category, the Leibniz rule is not required.
  \item  Another basic example is obtained by  transforming   $\underline{C}(R)$ into a dg-category:
 The category  $\underline{C}_{dg}(R)$ has the same objects and morphisms as $\underline{C}(R)$.
 The extra structure comes from observing that  the direct sum 
 $\bigoplus_{d\in \bZ} \hom^d_R (X^\bullet,Y^{\bullet })$  of $R$-homomorphisms  $X^\bullet\to  Y^{\bullet+d}$  
 of all degrees   receives a   derivation  from    the derivations   on $X^\bullet$ and on $ Y^\bullet$: 
\[
 d(f)= d_Y\comp f- (-1)^n f\comp d_X \in \hom_R^{d+1}(X^\bullet ,Y^\bullet),\quad \forall  f\in \hom^d(X^\bullet ,Y^\bullet).
 \]
 In this way, $\hom^\bullet(X^\bullet,Y^\bullet)$ becomes a complex, the \textbf{\emph{hom-complex}} associated to $(X^\bullet,Y^\bullet)$.
Usually, in a   $dg$-category one just focusses solely on  their  hom-complexes. 
  \end{enumerate} 
 \end{exmple}

\subsection{The  dg-category    of  matrix factorizations of $w$}  
Recall that a matrix factorization  of $w$ is a free $\bZ/2$-graded $R$-module $X$ of finite rank equipped with an odd degree $R$-linear map $d:X\to X$ such that $d\comp d=w\cdot \id$. These are the
  objects of $\matf {R,w}$. If one allows arbitrary free $R$-modules, these are the objects of
  $\matf{R,w}^\infty $. \label{page:mtf}
As in  Example~\ref{ex:dgexs}.(2)  there are  associated hom-complexes $\hom_R^\bullet(X,Y)$.  
Taking into account the $\bZ/2$-grading, such complexes have 
essentially two components, $\hom_R^0(X,Y)$ and $\hom_R^1(X,Y)$ and two differentials
\[
	\xymatrix{
   	\hom_R^0(X,Y)  \ar@/^/[rr]^{d^0}   &&  \hom_R^1(X,Y)  \ar@/^/[ll]^{d^1}
	}
\]
given by $(d^0f )^m = d^Y \, f^m  -    f^{m+1} \,d^X $ and $(d^1 g )^m = d^Y  \,g^m  +   g^{m+1} \,d^X $. Now $X$ and 
$Y$ are not complexes, but  note that   $d^1\comp d^0 f=0$, for instance, 
\begin{align*}
			(d^1\comp d^0 f)^0  = (d^Y\comp d^Y\comp f^0 -d^Y\comp f^1\comp d^X)+  (d^Y\comp f^1\comp d^X - f^0 
			\comp d^X\comp d^X) =0 , 
\end{align*}  
since $d^X\comp d^X=d^Y\comp d^Y =w \cdot \id$. The resulting hom-complex being $\bZ/2$-graded is denoted $\hom_R^{\bZ/2} (X,Y)$.   
  In the remainder of these  notes, also  the  categorical notation $\matf {R,w}(X,Y)$ is employed, instead of $\hom_R^{\bZ/2} (X,Y)$.   Note that in the homotopy category $[\matf {R,w}]$ the complex  $\hom_R^{\bZ/2}(X,Y)$
 gets  replaced by $[X,Y]= H^0 \hom_R^{\bZ/2}(X,Y)$.
 
 \begin{exmple}[The Koszul matrix factorization revisited]\label{ex:koszAsdg}
 Consider the  Koszul matrix factorization   $ \set{\bff,\bg}$  of $w$ given by
  \eqref{eqn:KoszMF} which is associated to an $R$-regular sequence $\bff$
  of length  $m$ and for which  $w=\bff\cdot \trp \bg$.  Place it in the  dg-category $\matf{R,w}$
 which consists of  one object $\Lambda^\bullet   N$, where $N=\oplus_{j=1}^m  Re_j$, and  has 
 $ \matf {R,w}(\Lambda^\bullet     N,\Lambda^\bullet   N)= \endo_R(\set{\bff,\bg})$ as its hom-complex .
 Note that the latter  $R$-algebra itself  can also be considered as a ($\bZ/2$-graded) dg-category  with one object 
 and with endomorphism algebra $\endo_R(\set{\bff,\bg})$. As  such it is full subcategory of $\matf{R,w}$.
 \end{exmple}

 \subsection{Tensors  and matrix factorizations }  
 
The tensor product of two matrix factorizations is defined using signs as  for  complexes:

\begin{dfn}[\textbf{\emph{Tensor products of matrix factorizations}}] \label{dfn:tensor}
Let $(X, d^X)$ be a matrix factorization of $w $ and $(Y, d^Y)$ a matrix factorization of $w'$ over $R$.
Put 
\begin{align*}  
 \begin{matrix} 
     Z^0   =   X^0\otimes_R  Y^0 \oplus X^1\otimes_R  Y^1, &  \qquad\quad  Z^1   = X^0\otimes_R  Y^1\oplus X^1\otimes_R  Y^0
\end{matrix} \hspace{7em}
\\
\begin{matrix} 
d_0^Z(x_0\otimes y_0)  =   d_0^X(x_0)\otimes y_0 + x_0\otimes d^Y_0y_0,  &   d_0^Z( x_1\otimes y_1) =   d_1^X(x_1)\otimes y_1 - x_1\otimes d^Y_1y_1 \\
d_1^Z(x_0\otimes y_1) =  d_0^X(x_0)\otimes y_1 + x_0\otimes d^Y_1y_1, & d_1^Z( x_1\otimes y_0) =  d_1^X(x_1)\otimes y_0 - x_1\otimes d^Y_0y_0.
\end{matrix}
\end{align*}Then $(Z,d^Z)=(X,d^X)\otimes (Y ,d^Y)$ is a matrix factorization of $w_Z=-w \otimes 1+ 1\otimes w'$.
Indeed, the signs are such that $d_0^Zd_1^Z= d_1^Zd_0^Z= -w \otimes 1+ 1\otimes w'$.

In particular,  the tensor product does not preserve $\matf {R,w}$.
  \end{dfn}

 Observe that the hom-construction $ \hom_R^{\bZ/2}(X,Y)$ applies also  in the situation where $X$ gives a matrix factorization of $w\in R$ and 
 $Y$ gives a matrix factorization of $w'\in R$ and then it is the hom-complex
  of a matrix factorization of $-w+w'$. In particular,  $ (Y^0,Y^1)= (R, 0 )$ is a matrix factorization of $0$
 (with zero differentials) and then  $ X^\bullet :=\hom_R^{\bZ/2}(X,Y)$ is a matrix factorization of $-w$. Explicitly: for  $ g\in (X^\bullet )^1$ one has $dg(x)= g (dx)$ and for $f\in (X^\bullet )^0$ one has $df=0$ so that in view of the signs $d\comp d g(x)= -g (d\comp d (x))=
 -w g (x)$.
 In particular, if $w\not=0$, the category $\matf {R,w}$ is not stable under  duality.

  \section{Matrix factorizations as stabilizions} \label{sec:MFasStabs}

\subsection{Koszul matrix factorization as a stabilization}
As demonstrated in Lemma~\ref{lem:BasicMatFact}, for a regular local ring $(R,\m)$ and $w\in\m$, a maximal 
Cohen--Macaulay module  $M$ over $S=R/w$ gives a matrix factorization of $w$ 
whose cokernel equals $M$. The category of such $S$-modules can be ``stabilized'' if   homomorphisms
   $g,g':M \to M'$  are declared to be  identical if 
 $g'= t \comp  g \comp t'$, where $t\in \hom (M,M)$ and $t' \in \hom (M',M')$ factor over some free $S$-module.
This procedure gives the  
 \textbf{\emph{stable  category}} $\underline{CM}^{\rm stab} (S)$ of Cohen--Macaulay modules   over $S$. 
 It turms out that the above cokernel assignment functor (on homotopy level) 
\[
\coker : [\matf {R,w}] \to \underline{CM}^{\rm stab} (R/w),
\] 
is an  equivalence of categories. By definition, given a maximal Cohen--Macaulay module  $M$ over $S$,
the homotopy class of a corresponding  matrix factorization  is  called \textbf{\emph{the stabilization $M^{\rm stab}$ of $M$}}.\index{stabilization}
\label{page:Stab}

This generalizes to the situation of Example~\ref{ex:KoszIHS}.(2):

 \begin{prop}[\protect{\cite[Cor.2.7]{Dyck}}] \label{prop:dyck1} 
 Let $(R,\m)$ be a regular local ring, $I\subset \m $ an ideal generated by a regular sequence $\bff=(f_1,\dots,f_m)$.
 Suppose $w\in I $, set $L:= R/I$   and write $ w=  \bff\cdot \trp \bg$ as before.  
 \par
 Put  $S=R/w$ (so that  $L$ is an $S$-module).  
 Then the stabilization $ L^{\rm{stab}}$ of $L$  (as an $S$-module) is  the Koszul matrix factorization
 given in Example~\ref{ex:KoszIHS}.(2), i.e. 
 $ L^{\rm{stab}}=\set{\bff,\bg}\in \matf {R,w}$
  \end{prop}
   
 \begin{exmple} \label{ex:StabIHS}
  If the \ihs\ is given by a \textbf{\emph{weighted homogeneous 
  hypersurface}}  $w(x_1,\dots,x_m)=0$   of degree $d$,  where $x_i$ has weight $d_i$, $i=1,\dots ,m$,
  the Euler formula gives $\sum_j x_j d_j w_{x_j}=d \cdot w$. Replacing $x_j$ with $x'_j= d_j/d$,
   this shows  that the Koszul matrix factorization 
   $\set{\bx, \nabla w }$  of $w$  represents $k^{\rm stab}$  as an $R/w$-module    while  the factorization   $\set{\nabla w, \bx}$ represents 
  the stabilization of the Jacobian algebra $\jac w$  as an $(R/w)$-algebra.\index{weighted homogeneous polynomial}
 \end{exmple} 
   
\subsection{Technical interlude} 

The constructions in this subsection, which use the concept of stabilization,  will be used in an essential way in \S~\ref{ssec:CmptGens} 
and \S~\ref{ssec:diagconstr}. The main goal is to replace   Hochschild
cohomology    of the category $\matf{R,w} $  by the (classical)  Hochschild cohomology of   the   algebra 
$\widehat {M_{R,w}}$ which will be  introduced in Corollary~\ref{cor:MwR}. 
It requires passing to $\matf {R,w}^\infty$ where To\"en's results  \cite{toen}  in homotopy theory  
can be used.
I won't detail these techniques but only quote the results that Dyckerhoff obtains using these.  
 
\begin{lemma}[\protect{\cite[Lemma 4.2]{Dyck}}] 
\label{lemm:StabRole}
As above, let $(R,\m)$ be a local ring, $I\subset R$ an ideal generated by a  regular sequence and $w\in I$. 
\par
Set $S=R/w$ and let $L$ be an $S$-module whose stabilization $L^{\rm stab}$ 
belongs to  $\matf {R,w} ^\infty$.   
Let $X$ be an object of $\matf {R,w}$, and let $ A$ be the set  of morphisms  from $X$ to itself  in
 the category  $ \matf{R,w}^\infty$ considered as a ring under composition and let $A^o$ be the opposite ring.

\par
Then composition with $f\in  A$ gives $\matf{R,w}^\infty (X, L^{\rm stab})$ as well as $\hom^{\bZ/2}(X,L)$ 
the structure of an $A^o$-module. 
One has an isomorphism 
\[
[\matf{R,w}^\infty (X, L^{\rm stab})]  \mapright{\sim} [ \hom^{\bZ/2}(X,L) ]
\]  in the derived category of $A^o$-modules.
\end{lemma}

Now return  to the situation of Proposition~\ref{prop:dyck}. So $L=R/I$,  and $I$ is generated by
a regular sequence $\bff=\set{f_1,\dots,f_m}$,  and $L$ is to be considered as an $S=R/(w) $-module. 
Recall that   the Koszul resolution of $L$ associated to $\bff$ reads as follows:
\[
 0\to\Lambda^ k N \mapright{\delta_\bff}\Lambda^ {k-1} N\to\cdots\to\Lambda^ 2 N \mapright{\delta_\bff} N \mapright{\phi} L=R/I \to 0,
\]
where $\delta_\bff$ is defined  by Eqn.~\ref{eqn:Koszres} and $\phi$ is the natural map onto $\coker \delta_\bff$. Applying the preceding lemma, one finds:

\begin{corr}[\protect{\cite[Prop. 4.3]{Dyck}}]  For    a free $R$-module $N$   of rank $m$,  let
\[
h:\Lambda^\bullet  N \to \Lambda^0 N=R \mapright{\phi} R/I=L
\]
the composition of the  projection and the map $\phi$ resulting from the Koszul resolution of $L$. 
Then  for  all objects $X$ in $\matf {R,w}$ and $f\in \hom(X, L^{\rm stab}) $ the  map $f\mapsto h\comp f$  establishes a quasi-isomorphism
\[
\matf{R,w}^\infty  (X, L^{\rm stab})\xrightarrow[\simeq]{\rm{qiso}}  \hom^{\bZ/2}_R(X,L),
\]
where 
\[
L^{\rm stab}= ( \Lambda^\bullet  N, \delta_\bff+d_{\bg})
\]
 is  the Koszul matrix factorisation of $w\in I$ (cf. Eqn.~\eqref{eqn:KoszMF}).
 \label{cor:dyckstab}
\end{corr}

 \subsection{Compact  generators}
 \label{ssec:CmptGens}
 
 In this section   $R=k[x_1,\dots,x_m]$ but later, for technical reasons, 
 it will be replaced  by its completion  $\widetilde R=k[\![ x_1,\dots,x_{m}]\!]$. Since matrix-factorizations take place 
 in $R$, this is always possible.
 The goal of this section is to find $R$-algebras which as dg-algebra are homotopically the same as 
the    two homotopy categories   $[\matf{R,w}]$ and $[\matf{R,w}^\infty]$   and  that $w\in R$ has an \ihs\ at
the origin. Assuming that $w= \bg\cdot\bx$, $\bg=(g_1,\dots,g_{m})$, $\bx=(x_1,\dots,x_{m})$, the two main players are
\begin{eqnarray}
 E  &= & k^{\rm stab}=\set{\bg,\bx} \label{eqn:E} \\
 M_{R,w} &= & \matf{R,w} (E,E). \label{eqn:MRw}
\end{eqnarray}
The $R$-algebra $M_{R,w}$   will serve  as the building block for constructing the desired $R$-algebra.

The entire construction depends on  a crucial feature of these categories, namely that they are  \textbf{\emph{triangulated}}. 
See   Appendix A.3.14 of   \cite{CommAlg}  for more details on this concept. 
Subcategories of a triangulated category   stable under shifts, triangles, isomorphisms  and direct  sums  (coproducts) 
are called \emph{\textbf{thick subcategories}}. Certain objects  in such categories play the role of generators 
and in the present setting yield the desired $R$-algebras.
The  required  technical definitions are as follows. \index{category!triangulated ---}\index{category!compact generator}
\index{compact generator} \index{category!thick subcategory}

\begin{dfn} 
An object $X$   of a category $\underline{C} $  admitting arbitrary direct sums  is said to be \textbf{\emph{compact}} if 
$\hom(X,-)$ commutes with coproducts, i.e.,  
$$\hom (X,\coprod_{j\in J} Y_j)\simeq  \coprod_{j\in J} \hom(X,Y_j)
$$ for all objects $Y_j$ of  $\underline{C} $, $j\in J$, and 
$X$ is said to be a \textbf{\emph{compact generator of  $\underline{C} $}}  if      
  $X$ is compact and if the smallest thick subcategory  of    $\underline{C} $  containing    $X$ is the entire category $\underline{C} $.
\end{dfn}

Compact generators in this sense can only exist  within the category $[\matf  {R,w}^\infty]$  and not in  
$[\matf  {R,w}]$ since the latter does not admit infinite direct sums.
The following principle describes the functorial role of a compact generator.

\begin{thm}[\protect{\cite[Thm. 5.1]{Dyck}}] 
\label{thm:dycksmain} Let $\underline{C} $ be a triangulated 2-periodic dg-category admitting  arbitrary direct sums  and
admitting a  compact generator    $X$.
Let $[X^o\text{-mod}]$ be the localization (as dg-modules) of the category of $X^o$-modules (in the set of equivalences).
Then the functor \label{page:CHat}
\begin{align}
\label{eqn:FXfunct}
\underline{F}_X  : \underline{C}  \to   \widehat {\underline{C} }:=[X^o\text{\rm -mod}] ,\quad Y \mapsto  \hom_{\underline{C} }( X, Y) 
\end{align} 
induces an isomorphism in the homotopy category of 2-periodic dg-categories.
\par
Consider $B:=(X,\endo X)$ as a dg-category. Then $\underline{F}_X $ sends 
  the homotopy class of $ B  $ to  $\widehat  B$, that is,  $B$ considered as a  $B^o$-module.
 \end{thm}
 The concept of  $X^o$-module  requires some explanation. Normally the action of $X$ on and $X$-module is from the left. To indicate that the action
 is from the right, one speaks of $X^o$-modules, indicating that the action of  $X$ is "reversed". \label{page:Xo}

Compact generators do exits for the category $[\matf {R,w}^\infty]$
due to Dyckerhoff:

\begin{prop}[\protect{\cite[Thm. 4.1 and Corollary 4.12]{Dyck}}] \label{prop:dyck}
In the present situation   $E $ (cf. \eqref{eqn:E})  is   a compact generator
of  the homotopy category $[\matf  {R,w}^\infty  ]$. 
\end{prop}
 
It follows that  taking in Theorem~\ref{thm:dycksmain}    for $\underline{C} $  
the homotopy category of $ \matf{R,w}^\infty  $, the object $ E=k^{\rm stab} $ is  a compact generator. 
So one can form the category $\widehat{\matf{R,w}^\infty }$  as in that  theorem.
Note also that the objects of $\matf{R,w}$, being bounded  complexes, are compact. 
Recalling the formula~\eqref{eqn:MRw}, one then deduces from the theorem:
   
  \begin{corr} \label{cor:MwR} 
     The functor $\underline{F}_{E }$ induces an equivalence of categories
 \[
   [\matf{R,w}^\infty ]  \mapright{\,\,  \sim \,\,}   \widehat {M_{R,w}} .
   \]
   In other words, $\widehat{M_{R,w}}$  is a model for the derived category of $ \matf{R,w}^\infty $.
   \end{corr}
   
 One can describe  $\matf{R,w}$ in a similar fashion in case  the ring $R$ is a complete local ring.
 In the situation of a polynomial  \ihs\  $w$, one may  assume this  and then \cite[Thm. 5.7]{Dyck} implies:
    
    \begin{corr} \label{cor:StabInMatfCat} If $w=0$ determines an  \ihs\ (with singular point at $\mathbf 0$)
     the functor $\underline{F}_{E }$ (defined by \eqref{eqn:FXfunct}) induces an equivalence of categories 
   \[
   [\matf{R,w}] \mapright{\,\,  \sim \,\,}   \widehat {M_{R,w}},
   \]  
   i.e., in the derived category one can replace
    $\matf{R,w}$ by  $\widehat {M_{R,w}}$.
    \end{corr}
    
 Concluding, the algebra $\widehat {M_{R,w}}$ represents the derived category of  the category $\matf{R,w}$, by
 which the goal set at the beginning of this subsection now has been achieved.

 \subsection{The diagonal construction}  
 \label{ssec:diagconstr}
 
 A further crucial ingredient for calculating Hochschild cohomology comes from
the   diagonal construction explained in this subsection: see the proof of Theorem~\ref{thm:RoughHH}
 In this construction  tensor products of matrix factorizations for $\matf {R,w}$ and $\matf {R',w'}$ play a role 
 for  the special case where $R'=R = k [\! [ x_1,\dots, x_m]\!] $, specifically, one uses\label{page:Delta}
 \begin{align*}
\widetilde R &= R\otimes_k R =k[\![ y_1,\dots,y_m,z_1,\dots, z_m]\!],\quad y_j=x_j\otimes 1, z_j=1\otimes x_j \\
\widetilde w &= -w\otimes 1 +1\otimes w\\
\Delta & = \widetilde R/ I_\Delta,\quad  I_\Delta=(y_1-z_1,\dots,y_m-z_m) \quad \text{the "diagonal" of  $R$ in $\widetilde R$.}\\
\text{      Since } &    w\in\m =(x_1,\dots,x_m)\text{ it follows that }\widetilde w \in I_\Delta.
\end{align*} 
 If $X$ is an object of $\matf {R,w}$, then $X^* \otimes X$ is an object of $\matf {\widetilde R,\widetilde w}$.
  Recall  further that to be able to speak of "stabilization" in the category $\matf {R,w} $, one works over $S=R/(w)$
and over $\widetilde S= \widetilde R/\widetilde w$ in the category  $\matf {\widetilde R,\widetilde w}$.
The preceding considerations applies to this situation. Indeed, set 
\[
  F  :=    \Delta^{\rm stab}   \text{ as an   $\widetilde S$-module}   \text{  (an object of $\matf {\widetilde R,\widetilde w}$)}.
\]
Then using the notation of \eqref{eqn:E} and \eqref{eqn:MRw}, Corollary~\ref{cor:StabInMatfCat} in the setting of matrix factorizations over $\widetilde R$
 gives an equivalence of derived categories
\[
\underline  F_{E^*\otimes_k E}  : [\matf {\widetilde R,\widetilde w} ] \mapright{\sim} [M_{R,w}\otimes M_{R,w}^o  \text{\rm -mod}], \quad  Y \mapsto  \hom( E^*\otimes_k E , Y) 
\]
where one uses that  $E^*\otimes_k E$ is a compact generator of  $\matf{\widetilde R,\widetilde w}$.
  
\begin{prop} \label{prop:idfunct} The functor $\underline{F}_{E^*\otimes_k E} $  sends the stabilized diagonal
$F =  \Delta^{\rm stab}$    to  $M_{R,w}= \hom^{\bZ/2}_{R}(E,E)$ considered as an $M_{R,w}\otimes M_{R,w}^o$-module.
\end{prop}
\begin{proof}
The aim is to show that $\underline{F}_{E^*\otimes_k E} $  sends $F$ to $M_{R,w}$ considered as an $M_{R,w}^o\otimes M_{R,w}$-module.
Apply Lemma~\ref{lemm:StabRole}  to $X= E^*  \otimes_k E$  
 and $L=F$ which represents $\Delta$.   
 One finds quasi-isomorphisms (in the  category  $ \matf{\widetilde R,\widetilde w} $ of matrix factorizations)
 \begin{align*}
  \matf{\widetilde R,\widetilde w}(E^*  \otimes_k E,  F)&\simeq \hom^{\bZ/2}_{R\otimes_k R} (E^*  \otimes_k E,  R)\\
  &\simeq   \hom^{\bZ/2}_{R}(E,E) =M_{R,w}.
   \end{align*} 
Note that  $M_{R,w} $ under left and right composition is an $M_{R,w}\otimes M_{R,w}^o$-module:
$(f,g) h= f \comp h \comp g$ for all $f,g,h\in M_{R,w}$. Also $\matf{\widetilde R,\widetilde w}(E^*  \otimes_k E,  F)$ and $\hom^{\bZ/2}_{R\otimes_k R} (E^*  \otimes_k E,  R)$ are $M_{R,w}\otimes M_{R,w}^o$-modules and one can check that the isomorphisms preserve this structure.
\end{proof}

 Recall that  $\widetilde R  = R\otimes_k R =k[\![ y_1,\dots,y_m,z_1,\dots, z_m]\!]$
and that the polynomial $w(\bx) \in \m\subset k[x_1,\dots,x_m]$ yields   $w(\by)\in k[y_1,\dots,y_m]$ and hence
$\widetilde w  = -w\otimes 1 +1\otimes w\in I_\Delta\subset \widetilde R$ can be written as  
$\widetilde w= \sum \widetilde w_j (y_j-z_j)= 
  \widetilde \bw\cdot (\by-\bz) $.
Using this,  the  central result which will be used for  calculating Hochschild cohomology is as follows:

\begin{prop} \label{prop:stabdiagendo}  %
\begin{enumerate}[\qquad 1)]
\item   
The stabilized diagonal  $\Delta^{\rm stab}$, an object in  $\matf{\widetilde R,\widetilde w}$, is represented by the Koszul matrix factorization $\set{\widetilde \bw, \by-\bz}$.
\item One has $\widetilde w_j= w_{x_j}\mod I_\Delta$.
\item $\endo( \Delta^{\rm stab}) $ -- considered as a complex  --  is isomorphic to the Koszul  complex for the sequence $\widetilde w_1,\dots,\widetilde w_m$
modulo the diagonal ideal $I_\Delta$, considered as a $\bZ/2$-graded complex, that is, if
$\bw= (w_{x_1},\dots,w_{x_m})$, then $\endo (\Delta^{\rm stab })\simeq N^\bullet (\bw)$ (see Example~\ref{ex:KoszIHS}(1)).
\item $H^k(\endo (\Delta^{\rm stab})  )=0$
for $k$ odd and equal to  $ \jac w$ if $k$ is even.
\end{enumerate}
\end{prop}
\begin{proof}
1)   is clear and  2) is  left as an exercise.\\
3) Apply Corollary~\ref{cor:dyckstab} with $X=\widetilde R$, $L=R$ and remark that $R$ is an $\widetilde R/\widetilde w$-module whose stabilization is
$\Delta^{\rm stab}$. Hence $\endo{ \Delta^{\rm stab}} = \hom^{\bZ/2}_{\widetilde R}(\Delta^{\rm stab},R)$. Note that $R$ as $\bZ/2$-graded complex
has $R$ in even degrees and $0$ in odd degrees and so all derivatives are $0$. Moreover, $R$ is considered as an $\widetilde R$-module 
and so $\by-\bz $  maps to $0$  under any morphism $\Lambda ^{2j} N  \to R$, $N$ a free $R$-module of rank $m$. So in the complex  $\hom^{\bZ/2}_{\widetilde R}(\Delta^{\rm stab},R)$
only the derivatives from  $\widetilde \bw =\nabla w \mod I_\Delta$ survive  which gives the
   Koszul complex  for $\widetilde w_1,\dots,\widetilde w_m$ modulo the diagonal ideal $I_\Delta$, i.e. for the ideal generated by the partial derivatives of $w$.\\
   4) Follows from the above by applying Lemma~\ref{lem:KoszFree}. \end{proof}

 \begin{rmk} \label{rmk:algstruct} Using the diagonal construction one can define a product structure on $\matf{R,w}$ which under   the equivalence of categories $M_{R,w}=\matf{R,w} (E,E)\simeq [\matf{R,w}]$ corresponds to the $R$-algebra structure on $M_{R,w}$. So one might consider the homotopy category $ [\matf{R,w}]$ as a categorical incarnation
 of the $R$-algebra $M_{R,w}$.
\end{rmk}

 \section{Hochschild cohomology}
 
 \begin{footnotesize} In this section  I shall introduce Hochschild cohomology,  first for algebras, and then for categories.
 The aim is to understand Hochschild cohomology for the category of matrix factorizations over a    commutative ring $R$ with a unit.
 As before,  $R$ will be a   polynomial algebra over a field $k$, or its completion.
 \end{footnotesize}
   
\subsection{Hochschild cohomology for algebras} \label{ssec:hhalgs}
Let $A$ be any associative algebra over   $R$. So $A$ is  an $R$-module  equipped with an
associative product.  Now pass to   the $R$-module
\[
C^d(A)=\hom_R (A\otimes\cdots \otimes A \to A),\quad \text{($d$ factors)}.
\]
By convention, $C^0(A)=A$.
The modules $C^d(A)$ can be made into  a cohomological complex with derivations
$\delta^d: C^d(A) \to C^{d+1}(A)$ given by
\begin{align} 
     \delta^d f(a_0,\dots, a_d) &= a_0 \cdot  f(a_1,\dots, a_d)- \nonumber \\ \label{eqn:OnHcompl}
      \quad\quad &\sum_{i=0}^d (-1)^i f(a_0,\dots,a_ia_{i+1},\dots,a_d)+ (-1)^{d+1}f(a_0,\dots,a_d)\cdot a_d. 
       \end{align} 
Its cohomology is the \textbf{\emph{Hochschild cohomology}} of the algebra $A$:   \index{Hochschild!cohomology}\label{page:HHA}
 \[
 \hh d A= H^d(C^\bullet (A),\delta^\bullet),
 \]
 named after Hochschild's article  \cite{hochschild}. 
 As for ordinary cohomology, this group carries a graded cup-product structure coming from the product
on co-cycles $\gamma \in C^n(A$), $\gamma' \in C^m(A)$  given  by
\[
\gamma \cup\gamma'(a_1,\dots,a_{n+m})= (-1)^{nm } \gamma(a_1,\dots,a_n)\gamma'(a_{n +1},\dots,a_{n+m}),
\quad \forall\, a_1,\dots,a_{n+m}\in A.
 \]
 
 The Hochschild cohomology can be also be described in terms of the \textbf{\emph{enveloping algebra}} \label{page:envelope}
\[
A^e:= A\otimes_A A^o,\quad A^o=\text{opposite algebra of } A 
\]
as will be explained next.  This  is based on the observation 
that  the action of $A$ on $A$ by left multiplication makes $A$ into an $A$-module while multiplication on the right
gives $A$  the structure of an $A^o$-module.
There is indeed a  complex of free  $A^e$-modules that computes Hochschild cohomology,   the so-called \label{page:BarCmplx}
\textbf{\emph{bar-complex}} $C^{\rm bar}_\bullet  (A)=A^{\otimes {\bullet+2}}$, a homological complex  starting in degree $0$ 
with $A\otimes A$ and with derivations given by 
\[
d_n (a_0\otimes\cdots\otimes a_{n+1})= \sum_{i=0}^n (-1)^i a_0\otimes\cdots\otimes a_ia_{i+1}\otimes\cdots
\otimes a_{n+1}.
\]
The modules $C^{\rm bar}_n(A)$ are free $A^e$-modules under the operation $(a\otimes b)\cdot (a_0\otimes 
\cdots a_{n+1} )= a \cdot  a_0\otimes\cdots a_{n+1}\cdot  b$ since  $C^{\rm bar}_n(A)\simeq A^e\otimes A^{\otimes n}  \simeq\oplus_j (A^e \otimes 1) \otimes e_j$, where the $e_j$ form a $k$-basis of $A^{\otimes n}$.
By definition, the associated cohomological complex is 
\[ 
C^\bullet (A,A)= \hom_{A^e}(C^{\rm bar}_\bullet   A,A). 
\]
There is an isomorphism of $k$-vector spaces $C^n(A) \to C^n(A,A)$ given by
$f \mapsto \set{ a_0\otimes\cdots\otimes a_m \to a_0 f(a_1 \otimes\cdots\otimes a_{n} )a_m } $
whose inverse is $g \mapsto \set {a_1 \otimes\cdots\otimes a_{n} \mapsto g(1\otimes a_1 \otimes\cdots\otimes a_{n} \otimes 1)}$,
as one easily verifies.
Hence  
\begin{align}
\label{eqn:hhA}
\hh d A = H^d (C^\bullet A)= H^d(C^\bullet (A,A)).
\end{align}

The bar-resolution is   a free resolution of $A$ as an $A^e$-module
by extending it to the right by the multiplication map $A\otimes_k A \to A$. The Ext-groups
are thus the cohomology groups of the complex $\hom_{A^e}(C^{\rm bar}_\bullet   A,A)$.\index{category!derived ---}
Next, pass to the derived category\footnote{See for example Appendix A.3.14  in  \cite{CommAlg}.} 
$\sfD(\ul A^e)$ which is built on complexes of $A$-bimodules such as the bar-complex, or its dual $C^\bullet(A,A)$.
In the derived language this gives  $H^d(R \hom_{\sfD(A^e)}(A,A)) \simeq \hom_{\sfD(A^e)}(A,A[d])$.
Summarizing, one has:
 
\begin{prop} $\hh d A\simeq \ext^d_{A^e}(A,A) $ which in turn is isomorphic to   
$H^d(R \hom_{\sfD(A^e)}(A,A)) \simeq \hom_{\sfD(A^e)}(A,A[d])$.
\end{prop}

 \subsection{Hochschild cohomology for dg-categories} \label{ssec:catHH}

 B. Keller \cite{keller} has introduced  an analog of the bar-resolution for any    dg-category $\underline{A}$ 
 which serves as a means to define
  Hochschild cohomology of $\underline{A}$. In order to carry this out, the first task is to define tensor products
 of dg-categories in order to define the analog of $A^e$.
 
 \begin{itemize}
\item The  tensor product $\underline{A}\otimes\underline B$ of dg-categories:     its objects  are pairs $(x,y)$ of objects
  $x$ of $\underline{A}$ and $y$ of $\underline B$ and its  morphisms are given by
  $\underline{A}\otimes\underline B((x,y),(x',y'))= \underline{A}(x,x')\otimes \underline B(y,y')$ as dg-modules;
\item the dg-category  $\underline{A}^o$, the one  opposite to $\underline{A}$,  has the same objects as  $\underline{A}$ but 
  $\underline{A}^o(x,y)= \underline{A}(y,x)$. 
\item
the enveloping dg-category  is \label{page:OppAndExtCat}
 $\underline{A}^e:= \underline{A}\otimes \underline{A}^o$.
   \end{itemize}
 
 One can attempt to define    Hochschild cohomology  by imitating what has been done for algebras: 
\begin{eqnarray*}
 \hh d {\underline{A}} =  H^d(R \hom_{\sfD(\underline{A}^e)}(\underline{A},\underline{A})) \simeq \hom_{\sfD(\underline{A}^e)}(\underline{A},\underline{A}[d]).
 \end{eqnarray*} 
 The problem is then to find a substitute for the  bar-complex  which should represent
   $ R  \hom_{\sfD(\underline{A}^e)}(\underline{A},\underline{A})$.   Any such complex
   is called a  \textbf{\emph{Hochschild complex}}.
There is indeed a  categorical version
of the bar-complex as explained in \cite{keller} but  this complex  usually is unsuitable for concrete calculations.
 There exists a  more suitable Hochschild complex  via  the diagonal construction, as now will be explained. But first, some more categorical constructions are \index{Hochschild!complex}
   needed. 
 
 \begin{enumerate}[(i)]
  
\item A dg-functor $\underline F:\underline{A}\to\underline B$ between two dg-categories $\underline{A},\underline B$ consists of a map $x\mapsto\underline  F(x)$ from objects in $\underline{A}$ 
to objects in $\underline B$, and for any two objects $x,y$, a $k$-linear morphism $\underline{A}(x,y) \to \underline B(\underline F(x),\underline F(y))$  preserving the identity
and satisfying the usual associativity condition.

\item Given a dg-category $\underline{A}$, a left  $\underline{A}$-module  consists of a functor $\underline M: \underline{A}   \to \underline C_{\rm dg}(k)$. So for  objects $x$ of $\underline{A}$ the image
$\underline M(x)$ is a $k$-complex and  for any  two objects $x,y$, there is a  $k$-linear morphism $\underline{A}(x,y) \to  \underline C_{\rm dg}(\underline M(x),\underline M(y)) = \hom (\underline M(x),\underline M(y)) $.
In other words, this  gives  morphisms of complexes $\underline{A}(x,y)\otimes \underline M(x)\to \underline M(y)$
which  describes  the action of $\underline{A}$ on $\underline M$.
 
\item An  \textbf{\emph{$\underline{A}$-bimodule}}  $M$ is  a  dg-functor  $M:\underline{A}^e \to  \underline C_{\rm dg}(k)$.
\end{enumerate}

\begin{exmples}  \textbf{1.} The \textbf{\emph{identity  $\underline{A}$-bimodule}}    
\index{identity!bimodule}\index{identity!functor}\index{diagonal bimodule}
or  the \textbf{\emph{diagonal bimodule}} $\Delta_{\underline{A}}: \underline{A}^e \to \underline C_{\rm dg}(k)$ of $\underline{A}$  is defined by $\Delta_{\underline{A}} (x,y)=\underline{A}(x,y) $ on objects $(x,y)$ of $\underline{A}\otimes\underline{A}^o$
  and 
  \[
  \Delta_{\underline{A}} ( \underline{A}(x,y)\otimes  \underline{A}(y',x'))= \hom(\underline{A}(x,y),   \underline{A}(y',x'))  
  \]
  on morphisms of $\underline{A}^e$.  This is a left $\underline{A}^e$-module with action 
  $(\underline{A}^e( (x,y) , (x',y')) \otimes  \underline{A}(x,y)   \to \underline{A}(y',x') $.  In the case of an $R$-algebra considered as
  category, $\Delta(A)$ is the algebra $A$ considered as an $A^e$-bimodule. \label{page:IdFunct}
\\
\textbf{2.} A special case of \textbf{\emph{the identity functor.}}
Let $X $ be a    variety  over the field $k$ and let $\delta: X \into X\times X$ be the diagonal embedding with image 
$\Delta$. Let $p,q:   X \times X \to X$ be the two projections.  Note that $\delta_*\cO_X= \cO_\Delta$ and so for
 a locally free sheaf $\cE$ on $X$ one has $p^*\cE\otimes \delta_*\cO_X =q^*\cE  \otimes \delta_*\cO_X$. It follows that there are canonical isomorphisms
 \[
 q_* (p^* \cE\otimes_{O_{X\times X} }  \delta_*\cO_X)\simeq  q_*( q^* \cE\otimes_{O_{X\times X}}  \delta_*\cO_X)= \cE\otimes_{O_X}  (q_*\comp  \delta_*\cO_X)\simeq  \cE.
 \]
 So the functor on the category of locally free $\cO_X$-sheaves given by
 \[
 \cE \mapsto q_* (p^* \cE\otimes_{O_{X\times X} }    \cO_{\Delta})\]
  represents the identity functor. In this sense $\Delta$ "is" the identity functor on  the category of locally free sheaves on $X$. 
  This functor can be extended to the dg category of complexes of locally free sheaves on $X$, or to the category of matrix factorizations 
  in the sense of \S~\ref{ssec:matfGeom}.

  The same construction for pairs  $(X,Y)$ of  $k$-varieties, $\cE$ a   coherent $\cO_X$-module, 
 and with the structure sheaf of the diagonal replaced with any coherent $\cO_{X\times Y}$-sheaf $\cK$ 
 defines the \textbf{\emph{Fourier-Mukai transform
 of $\cE$ with kernel}}  $\cK$, a functor on the category of  coherent $\cO_X$-modules to the category of coherent $\cO_Y$-modules. \index{Fourier--Mukai transform}
 
 The diagonal also allows to define Hochschild cohomology for the scheme $X$ as 
 \begin{equation}
 \label{eqn:hcohX}
  \hh d X := H^d(X\times X,  R\cH om_{\cO_{X\times X}} (\cO_\Delta,\cO_\Delta)).
 \end{equation} 
 
   \end{exmples}

Equation~\eqref{eqn:hcohX} can be seen as an example of a general result due to
  B.  Toën  \cite[Cor. 8.1]{toen} stating that  the usual bar complex   is    homotopic  to the endomorphism complex of the identity bimodule  
\begin{eqnarray}
 \label{eqn:HHcoh}
  \endo_{ \underline{A}^e }
( \Delta_{\underline{A}} )=  \hom_{\underline{A}^e }  (  \Delta_{\underline{A}}, \Delta_{\underline{A}}) ,
\end{eqnarray}
and hence serves as a Hochschild complex. In particular, equation~\eqref{eqn:hcohX} shows that any complex representing
$R\cH om_{\cO_{X\times X}} (\cO_\Delta,\cO_\Delta)$ is a Hochschild complex for $\cO_X$ (identified with $\cO_\Delta$).

 By making use of    the diagonal construction in \S~\ref{ssec:diagconstr} the preceding observations    
 can be applied to the category $\matf{R,w}$ of  matrix factorizations of   an \ihs\ at $\bo$ given by a polynomial  
$w\in   R=\bC  [   x_1, \dots, x_{m} ]$. Indeed, by Proposition~\ref{prop:idfunct}  the identity functor is represented by the stabilized diagonal.
Its endomorphism algebra as well as  its cohomology  has been calculated in Proposition~\ref{prop:stabdiagendo}. The results thus reads

\begin{thm} \label{thm:RoughHH}    The Koszul complex on  
the derivatives $\displaystyle\set{ w_{x_1} ,\dots\ w_{x_m }}$ of $w\in R=\bC[x_1,\dots,x_m]$, 
viewed as a $\bZ/2$-graded complex serves as a Hochschild complex for $\matf{R,w}$ and hence 
\[
\hh d{\matf{R,w}}=\begin{cases}
						\jac {w}=  \bC[x_1,\dots,x_{m}]/ (w_{x_1} ,\dots\ w_{x_m })  &  \text{ if  $d$ is even}\\
						0 &  \text{ if  $d$ is odd.} 
			\end{cases}
 \]
\end{thm}

 \begin{rmk} \label{rmk:prodstruct} As explained in Remark~\ref{rmk:algstruct}, one can equip  the homotopy category $[\matf{R,w}]$
 with an $R$-algebra structure through the quasi-isomorphism $\matf{R,w} (E,E)\simeq \matf{R,w} $.
 The Hochschild complex respects the algebra structure and so this structure  survives on the level of its cohomology. 
 Clearly, this is reflected in the above theorem:  
 the jacobian ring $\jac {w}$ is an $R=\bC[x_1,\dots,x_m]$-algebra.
 \end{rmk}
 
  \section{The equivariant case}
  \label{sec:EqMatFac} 
 
\subsection{Equivariant matrix factorizations}

Assume  that   the  \ihs\   given by $\set{\set{w=0},\bo} $  admits symmetries in the weak sense that
  $w \in \bC[x_1,\dots, x_m ]$   is a semi-invariant with respect to  a  group $\Gamma$ of linear transformations of $\bC^m$, that is,
   there is a character $\chi :\Gamma\to \bC^\times$ with 
\[
\gamma(w)= \chi(\gamma) w,\quad \forall \gamma\in \Gamma.
\]
A  $\Gamma$-character $\chi$ can be viewed as a rank $1$ free $R$ module with obvious $\Gamma$-action,   
and so, if $V$ admits a $\Gamma$-action , also $V(\chi):=V\otimes_R \chi$ admits a canonical $\Gamma$-action.

\index{matrix factorization!equivariant ---}

\begin{dfn} Let $(R,\m)$, $\Gamma$, $\chi$  and  $w\in \m$ as above.  
\begin{enumerate}[(i)]

\item A \textbf{\emph{$\Gamma$-equivariant matrix factorization}} of $w$ consists of a pair 
$X^0, X^1$ of free $R$-modules of finite rank equipped with an action of $\Gamma$,  together with equivariant $R$-linear morphisms
$  X^0\mapright{d_0} X^1(\chi)$ and $  X^1\mapright{d_1} X^0$ such that  $d_1\comp d_0 = w\cdot \id$ (as above).   

\item A morphism $f : (X ,d)\to (Y ,d')$ of matrix factorizations is a $\bZ/2$-graded $\Gamma$-equivariant map $f $
such that $f  \comp d'=d\comp f $.
\end{enumerate}
\end{dfn}

The resulting dg-category is denoted $\matf{w,\Gamma,\chi}$.\label{page:EqvMatf}
There is  an important difference with the non-equivariant case: this category is not $\bZ/2$-graded.
This is caused by the action of the character. Instead, the $\Gamma$-equivariant hom-complexes are
\begin{align*}
\begin{split}
\hom^{2k}_\Gamma(X,Y) & =  \hom   (X^0,Y^0(\chi^{\otimes  k}))^\Gamma\oplus \hom   (X^1,Y^1(\chi^{\otimes k} ))^\Gamma,\\
\hom^{2k+1}_\Gamma(X,Y) &=  \hom   (X^0,Y^1(\chi^{\otimes(k+1)}))^\Gamma\oplus \hom   (X^1,Y^0(\chi^{\otimes k}))^\Gamma,
 \end{split}
 \end{align*} 
  where  the differentials   are defined as in the non-graded case and where  $
  M^\Gamma$
stands for   the submodule of $\Gamma$-invariants  of   $M$, i.e. $ M^\Gamma=\sett{x\in M}{\gamma(x)=x \text{ for all } \gamma \in \Gamma}  $.

\begin{exmple} 
 \label{ex:main}
Assume  that   $\Gamma$   acts on a commutative ring $R$ and let $N$ be a free $\Gamma$-module of finite  rank.
Starting from  $\bff\in N^\Gamma$    and a $\Gamma$-equivariant map $\phi: N\to R$
such that $\phi(\bff)\in R$ is $\Gamma$-invariant, the two Koszul sequences given in Section~\ref{sec:Koszmf} are obviously  $\Gamma$-equivariant.
This yields  a $\Gamma$-equivariant
 matrix factorization for $w=\phi(\bff)$. One can also find a $\Gamma$-equivariant Koszul matrix factorization $\set{\bg,\bx}$   for $\bC$
 (with trivial $\Gamma$-action)  using a $\Gamma$-invariant  \ihs\  given by $ w=\sum g_i   x_i $.
Below  I'll restrict myself to the following situation:
 
\begin{itemize} 
\item \begin{itemize}
		\item The ring $R$ is the polynomial ring $R=\bC[x_1,\dots,x_m]$;
		\item $\Gamma$ is an finite extension of $\bC^* $  acting diagonally on  $ \bC^{m}$, i.e. there  are characters 
		\[
		\hspace{2.7em}\chi_i: \Gamma\to \bC^*  , \quad i=1,\dots, m  \text{ with  } \gamma(x_i)=\chi_i(\gamma) x_i   \text{ for all  }\gamma\in \Gamma ; 
		\]
		\item There is a character 
		\[
		\hspace{-1.7em}\chi: \Gamma\to \bC^*\text{ such that }  G:= \ker \chi  \text{ is a finite group.}
		\]
	\end{itemize}
\item \begin{itemize}
		\item The polynomial $w\in R=\bC[x_1,\dots,x_m]$ belongs to the $\chi$-character subspace 
		$ R_\chi =\sett{w\in R}{\gamma(w)=\chi(\gamma) w} , \gamma\in \Gamma$
		so that $w$ is invariant under $G$,
		\item $w=0$ has  an \ihs\ at $\mbold 0$.
	\end{itemize}
\end{itemize}
\end{exmple}

\subsection{Hochschild cohomology} \label{ssec:hheq}

In the situation of Example~\ref{ex:main}, the techniques that have been used in  Section~\ref{ssec:catHH} can be  extended to the 
$\Gamma$-equivariant  case. It is a special case of a more general situation studied at length in \cite{balfavkatz}. 

The calculation in loc.\ cit.\ is very technical, but the main ideas are already present in the non-equivariant setting. There are some 
important special features of the equivariant setting   to keep in mind:
\begin{enumerate}[label=\bf{\arabic*.},nosep]
\item In the equivariant setting characters of the group $\Gamma$  play a central role which, as
explain above,  implies  that  the hom-complexes of the matrix factorizations are not periodical, but graded by the characters and so
are the  Hochschild cohomology groups; 
\item as in the non-equivariant case, Jacobian rings 
come up, but now involve only part of the variables associated to a character space.
\end{enumerate}
The expression for equivariant Hochschild cohomology  involves  an extra variable $x_0$,
 and so one introduces
\begin{itemize}[nosep,left= \parindent]
\item $V:=\bC x_0 \oplus\cdots\oplus  \bC x_m$;

\item The $\Gamma$-action extends to $V$ (and hence to $ \bC[x_0,\dots,x_m]$) through the character
\begin{equation}
\label{eqn:OnChis}
\chi_0:= \chi \otimes\prod_{i=1}^{n+1}\chi_i ^{-1}, \quad \gamma(x_0)= \chi_0(\gamma) x_0 \text{ for all } \gamma\in \Gamma;
 \end{equation} 
\end{itemize}
\noindent Moreover, for each $\gamma\in \Gamma$ set
\begin{itemize}[left= \parindent] 
\item $V^\gamma= \sett{v\in V}{\gamma(v)=v}=\bigoplus_{i\in I^\gamma}\bC x_i $, where  $I^\gamma $ is the collection of integers $i$ in  $ I:=\set{  0,\dots, n+1} $ for which  $\gamma(x_i)=x_i $.

\item $w^\gamma$, the restriction of the polynomial $w$ to the polynomial subalgebra of $\bC[x_0,\dots,x_m]$ spanned by the $\gamma$-invariant variables.

\item $V_\gamma=\bigoplus_{j\in I\setminus I^\gamma} \bC x_j  $, a $\Gamma$-invariant complement of $V^\gamma$ in $V$. 
\end{itemize}

 Tensor-products and duals  of $\Gamma$-modules  are $\Gamma$-modules. In what follows, the dual of $\bC[x_0,\dots,x_m]$
is identified with the polynomial ring  $\bC[x_0^*,\dots,x_m^*]$ where the $x_j^*$ are given degree $-1$.
Hence      $ \bC[x_0,\dots,x_m,
x_0^*,\dots,x_m^*]$ is a $\Gamma$-module as well.

In the present situation the following $\Gamma$-submodules of   $ \bC[x_0,\dots,x_m,
x_0^*,\dots,x_m^*]$    play a role:
\begin{itemize}[left= \parindent, labelsep=.2em] 
\item   
The Jacobian ring  $\jac{w_{ \gamma}} $.
\item The character space  $\det (V_\gamma)=\Lambda^{\dim V_\gamma}V_\gamma$ with character $\kappa_\gamma:= \prod_{i\in I_\gamma}  \chi_i$,
\item  For any $\Gamma$-invariant subring $S $ of  $\bC[x_0,\dots,x_m,
x_0^*,\dots,x_m^*]$, the  character space  $S_\chi:= \sett{ x\in S  }{ \gamma(x)=\chi(x) \cdot x} $.
\end{itemize}
Recall  from Section~\ref{sec:Koszmf} that the Koszul sequence associated to the regular sequence $\bw =(w_ {x_1},\dots,w_{x_m})$
is denoted by $N(\bw )$,  and similarly for the regular sequence $\bw^\gamma$ associated to $w^\gamma$.
Using these conventions, the result from \cite{balfavkatz} in this case  reads (see \cite[Eqn. (5.2]{LekiliUeda})
\begin{eqnarray*}
 \hh t {\matf {w,\Gamma,\chi}} = \bigoplus_{\substack{\gamma \in \ker \chi, \ l \geq 0 
                   \\ t - \# I^\gamma = 2u }}
  \lb
  H^{-2l}(N(\bw^\gamma) \otimes\kappa_\gamma \rb_{(u+\ell)\chi } 
  \bigoplus  & \\
                  &  
\hspace{-9em} \displaystyle \bigoplus_{\substack{\gamma \in \ker \chi, \ l \geq 0 
                                          \\ t - \# I^\gamma  = 2u+1}}
  \lb H^{-2l-1}(N(\bw^\gamma) \otimes   \kappa_\gamma  \rb_{(u+\ell+1) \chi }.
 \end{eqnarray*}

Since  $w=0$ (and hence also $w^\gamma=0$) is an isolated  \ihs\ at $\mathbf 0$, Example~\ref{ex:KoszIHS}.1 tells that the cohomology of the Koszul sequence 
$\bw$ (and also for $\bw^\gamma$) is concentrated in degree $0$ and so only the terms  with $\ell=0$ contribute, i.e.,
the terms involving the  $H^0(N(\bw^\gamma))= \jac {w^\gamma}$.
Unraveling the above formula then yields:\index{Hochschild!cohomology (equivariant matrix factorization)}

\begin{prop} \label{prop:hheqmatfac} Let $u\in \bZ$, then
\begin{enumerate}
\item The contributions to $\hh {2u+\#I^\gamma}  { \matf{w,\Gamma,\chi }}$ consist  of two types of monomials in the ring $\bC[x_0,\dots,x_m,
x_0^*,\dots,x_m^*]$:
\begin{enumerate}
\item In case $\gamma(x_0)\not= x_0 $, every monomial  in $S_{\chi^{\otimes u}}$, where $S= \jac{w^\gamma} \otimes \kappa^*_\gamma$;
\item If  $\gamma(x_0) =\ x_0 $, every monomial in $(S\otimes \bC[x_0] )_{\chi^{\otimes u}}$. 
\end{enumerate}
\item The contributions to $\hh {2u+1+\#I ^\gamma}  { \matf{w,\Gamma,\chi } }$ 
consist  of every monomial of the form  $  x_0^* \otimes (S\otimes \bC[x_0] )_{\chi^{\otimes u}}$ provided $\gamma(x_0) =\ x_0 $.
 
\end{enumerate}
\end{prop}

 \chapter{Bigrading on symplectic cohomology as a   contact-invariant}  \label{lect:bigrading}
\section*{Introduction}
 
 A  Gerstenhaber   algebra, introduced in  Section~\ref{sec:Gerst}, gives a bigrading on Hoch\-schild 
  cohomology,  as explained  in \eqref{eqn:HHbigrad}.  It turns out that a  Gerstenhaber structure    exists 
  on the symplectic cohomology  of  the symplectic 
completion of a Liouville domain with $b_1=0$ and   with trivializable  tangent bundle, 
  as   is the case for the Milnor fiber of an \ihs. Under favorable conditions, enumerated in 
Theorem~\ref{thm:ContInvOfLnk}, this structure is a
contact invariant of the link. It follows in particular (cf. Corollary ~\ref{cor:contactinv}) that 
 for isolated cDV singularities $(X,x)$ the resulting grading  on each of the vector spaces $ \hh  d {X,x}$, $d<0$,  is a contact invariant.

\section{Gerstenhaber algebras}
\label{sec:Gerst}
A graded complex vector space $\geg^*=\oplus_{k\in \bZ} \geg^k $,  
is a \textbf{\emph{Gerstenhaber algebra}} if it comes equipped with \index{Gerstenhaber!algebra}

\begin{itemize}
\item a degree-preserving associative and graded commutative
product $\cdot$, that is, $a\cdot b= (-1)^{|a| \cdot |b|}  b\cdot a$, where $|a|,|b|,\dots$, denotes the degree of $a,b,\dots$.
\item a  Lie-algebra bracket  $[-,-]$ of degree $-1$:
\begin{itemize}
\item
$[ \geg^k, \geg^\ell] \subset \geg^{k+\ell-1}$;
\item $[a,b]= -(-1)^{(|a| -1)\cdot (|b| -1)}[b,a]$ (i.e. $[-,-]$ is graded anti-symmetric);
\item  $[a,[b,c]] = [[a,b],c] + (-1)^{(|a|-1)(|b|-1)}[b,[a,c]] $ (the Jacobi identity).
\end{itemize} 
 \item the two products are compatible:  $[a,b\cdot c] = [a,b]c + (-1)^{(|a|-1)|b|} b\cdot [a,c] $  (the Poisson identity). 
\end{itemize} 

The subspace $\geg ^1$ is a Lie-algebra over $\bC$. Moreover,
since $[\geg^1\, ,\geg^k]\subset \geg^k$, one has a degree preserving representation of $\geg^1$ on $\geg$.
Assume now that $\geg^1$ is a finite dimensional  Lie algebra. 
Then, from the usual theory of Lie-algebras (see e.g. \cite[\S 15.3]{hum}), 
there is a Cartan subalgebra $\gh\subset \geg^1$, \index{Cartan subalgebra}that is a nilpotent  subalgebra which equals its  own normalizer.
Cartan subalgebras are unique up to conjugacy. In case  $\geg^1$  is semisimple, a Cartan subalgebra is  the same as a maximal abelian
subalgebra.  

\begin{exmple} Consider the traceless matrices $\gs\gel(2)$ of $2\times2$-matrices. The $1$-dimensional vector space 
$\gh$ of diagonal matrices is a Cartan subalgebra. The functional $\alpha: \gh\to \bC$ sending $H=\begin{pmatrix}
1 & 0\\ 0& -1
\end{pmatrix}$ to $2$ gives  an eigenvalue for the adjoint action of $H$ on $\gs\gel(2)$ with eigenvector
$\begin{pmatrix}
0& 1 \\ 1& 0
\end{pmatrix} $. The functional $\alpha$ is called a root of  $\gs\gel(2)$ . The only other root is $-\alpha$.
\end{exmple}

Since $\geg ^1$ cannot be assumed to be semisimple, the usual approach with root systems does not apply here.
It is still true that for a finite dimensional representation $\rho: \gh \to \endo V$    the Jordan decompositions
for the  non-zero elements  $x\in\gh$ are compatible and lead  to a common generalized eigenspace decomposition
\[
V= \bigoplus _\lambda V_\lambda, \quad  V_\lambda=\sett{ v\in V  }{ (\rho(x) - \lambda(x) \id ) ^{n_\lambda} v=0,\, \text{ for all } x\in \gh},
\]
where $n_\lambda$ is some  positive integer. Moreover, the set of eigenvalues $\lambda(x)$  
for a fixed $\lambda$ defines a linear function on $\gh$.
\par
In the semisimple case the $V_\lambda$ are genuine  eigenspaces and are directly related to the roots. This  
provides $V$ with a multi-index grading as follows. Choose a basis $\set{\alpha_1,\dots,\alpha_r}$, $r=\dim \gh$
for the roots. Then the eigenvalues occurring in a $\gh$-representation space $V$ are integral linear combinations 
 $\lambda=\sum \lambda_i \alpha_i$ to which one associates
  the multi-index $\vec{\lambda}= (\lambda_1,\dots,\lambda_r)\in \bZ^r$.
This gives the aimed for grading $V=\oplus V_{\vec \lambda}$. 
In this situation, the adjoint representation of $\gh$ on $\geg^1$ gives the  so-called Cartan decomposition $
\geg^1=\gh  \oplus_{\alpha\not=0} \geg^{1,\alpha},\quad \dim (\geg^{1,\alpha})=1, \alpha\not=0$,
where now any  non-zero $\alpha$ is called a root.
\par
In  the general case the $(\lambda=0)$-eigenspace for the adjoint action of $\gh$ on $\geg^1$ still equals the Cartan subalgebra. 
However, instead of a root basis, one can only use  a $\bC$-basis and the multi-index becomes a complex multi-index.
In the applications below,   $\dim \gh=1$ and  there is  an honest    finite complex grading on each of the $\geg^k$ resulting in the bigrading 
\begin{equation}
\label{eqn:FirstBigrading}
\geg=\bigoplus_{k\in\bZ,\ell \in \bC } \geg^{k, \ell }.
\end{equation} 
This bigrading depends on the basis  $\alpha$ of $\gh$. Changing it by a complex multiple $\mu\cdot  \alpha$ replaces 
the second grading $\ell$ by $\mu \ell$. One calls it a \textbf{\emph{rescaled grading}}.
 
 Recall from \S~\ref{ssec:hhalgs} that the Hochschild cohomology $\hh * A$ of an associative
algebra $A$  is constructed from a complex $C^\bullet(A)$ and if $Z^k(A)$ is the sub-algebra
of the $k$-cocyles, there is a a Lie-algebra structure which comes from a
 Lie bracket  
\[
Z^k(A)\times Z^\ell(A) \to Z^{k+\ell-1}(A)
\]
generalizing the Lie bracket on $Z^1(A)$, the so-called
\textbf{\emph{Gerstenhaber bracket}}. See \cite{gerstenhaber} for the (involved) definition.\index{Gerstenhaber!bracket}
This bracket together with the cup product structure  gives Hochschild cohomology 
$ \hh * A$ the structure of a Gerstenhaber algebra.

\begin{exmple}
Since    Remark~\ref{rmk:prodstruct}  states that the Hochschild cohomology 
of the category $[\matf{R,w}]$ of matrix factorizations for $w$ can be  computed  from a Hochschild complex of $R$-algebras, 
it   has a natural structure of a Gerstenhaber algebra.
 
 Observe that  since $R=\bC [x_1,\dots,x_m]$, the  $R$-algebras that occur here are graded and so $\hh * {\matf{R,w}}$ as well as 
 $ \hh * {\matf {w,\Gamma,\chi}}$ receive  an extra grading. This grading is an \textbf{integral} grading.
\end{exmple}

The above example motivates to consider 
the abstract situation of a    \textbf{\emph{graded}} associative algebra,  $A$,  that is $A=\oplus_{\ell\in\bZ} A^\ell$.  As
in the example,  the spaces $C^k(A)$  then receive an extra grading.  The one on
 $C^1(A)=\hom(A,A)$ is to be  interpreted as an operator $E: A  \to  A $
having  integral eigenvalues $\ell$ on the  eigenspace $ A^\ell$. Such an operator can be seen to be
 a derivation on $A $   and hence  defines a class $[E]\in \hh 1 A$. 
 The operator $E$ acts on $Z^k(A)$ through the Gerstenhaber-bracket and preserves $B^k(A)$. Hence one gets
 a second grading on Hochschild  cohomology.
Usually one shifts the grading on $A$ to achieve that the bigrading adds up to the degree  of the 
Hochschild cohomology by setting
\begin{equation}
\label{eqn:HHbigrad}
 Z^{k,\ell}(A):= \sett{ c\in  Z^{k }(A  )} { [E, c]= \ell  c}[-\ell] , \quad \hh  {k,\ell} A = Z^{k,\ell}(A)/B^{k,\ell}(A),
\end{equation} 
where $B^{k,\ell}(A)= Z^{k,\ell}(A)\cap B^k(A)$.
It follows from kernel  of $[E,-]$ on $\hh 1 A$  is exactly $\hh  {1,0} A$.  If this space happens to be $1$-dimensional
it obviously is a Cartan subalgebra of $\hh 1 A$ and by unicity of such a subalgebra, one deduces:

\begin{prop} \label{prop:AbstrBigr} Suppose $\dim {\hh  {1,0} A}=1$. The  adjoint action of the Cartan subalgebra $\hh  {1,0} A\subset \hh 1 A$ 
on Hochschild cohomology  yields an integral    bigrading   \eqref{eqn:HHbigrad} which   satisfies
 \[
\hh {m} A =\bigoplus_{k+\ell =m } \hh  {k, \ell } A .
\]
\end{prop}

 \begin{exmple}
\label{exm:gas}
 $\hh * {\matf{R,w}}$ as well as  $ \hh * {\matf {w,\Gamma,\chi}}$ come from Hochschild complexes of  
 $R$-algebras with $R=\bC[x_1,\dots,x_m]$.
\end{exmple}

\section{How Gerstenhaber algebras lead to contact invariants}
\label{sec:SHasGest}

Recall  that by Theorem~\ref{thm:fund}  the symplectic cohomology of   a Liouville domain $W$ with $c_1(W)=0$ has a graded product. More is true:

\begin{thm}
Let $(W,\omega)$ be a Liouville domain for which  $TW$ is trivializable. and such that
$b_1(\widehat W)=0$, where  $\widehat W$ is the  symplectic completion of $W$.
Then $\sh * W$ admits a Lie-algebra bracket of degree $-1$ and, together with the graded product $\sh * W$,  forms a  Gerstenhaber algebra.
\end{thm}

There is an excellent overview of the construction of the Lie bracket  in \cite[\S 4.3]{EvansLekili}.
The Gerstenhaber structure then provides a Cartan subalgebra   $\gh \subset \sh 1 W$ which is the generalized $0$-eigenspace 
$\sh{1,0} W$ of the  adjoint action of $\gh$ on $\sh 1 W$. Assuming that
\begin{equation}\label{eqn:SH10}
\dim {\sh {1,0} W}=1,
\end{equation} 
the Gerstenhaber structure gives $\sh * W$ a bigraded structure as in  \eqref{eqn:FirstBigrading}:
\[
\sh * W= \bigoplus_{k\in\bZ,\ell \in \bC } \sh {k, \ell }W.
\]
It is worthwhile to observe that  here one has $\sh {k, \ell }W\subset \sh {k  }W$, contrary to the bigrading
in Hochschild cohomology discussed in the previous subsection  where  $\hh  {k, \ell } A\subset \hh {k+\ell} A$ 
  (see Proposition~\ref{prop:AbstrBigr}).

 One applies this to contact manifolds $S$ which are symplectically fillable, say $S=\partial W$, $W$ a Liouville
 domain and $\widehat W$ its symplectic completion.
 In  \S~\ref{ssec:OnContactInv} it  has been explained that the groups $\sh k W$  for $k<0$ are
  contact invariants  under the assumption that $(S,\xi)$ is index-positive.   Lemma 4.3  in \cite{EvansLekili} states something more precise, namely, if in addition, 
 every closed Reed orbit  $\gamma$  has  Conley--Zehnder index $\ge  \max(5-n,n-1)$, then  for a suitable almost complex
  structure on $W$ these orbits stay away from the cylindrical end of $W$, that is, close to $S=\partial  W$.
 Together with some supplementary conditions, 
 this   implies  that  then  the Gerstenhaber algebra structure on $\sh {<0}  W$  is  a
 contact invariant  for  $ S$: 
 
 \begin{thm}[\protect{\cite[Cor. 4.5]{EvansLekili}}] \label{thm:ContInvOfLnk}
 Let $S$ be as above and let $W,W'$ be Liouville domains of (real) dimension $2n$ with $\partial W=\partial  { W'}=S$.  Assume
 \begin{enumerate}[(1)]
\item The Conley--Zehnder index of every closed Reeb orbit $\gamma$ satisfies $\cz \gamma\ge \max(5-n,n-1)$;
\item $c_1(W)=c_1(W')=0$;
\item $W$ and $W'$ admit Morse functions all of whose critical points have index $\not=1$.
\end{enumerate}
Then 
\begin{enumerate}[(a)]
\item there is an isomorphism of Lie algebras $f^1: \sh 1 W \mapright{\sim}  \sh 1 {W'}$, 
\item for each $d < 0$  there is an isomorphism $f^d:  \sh d W \mapright{\sim}  \sh d {W'}$ which intertwines the induced representations given by the adjoint representation given by the bracket operation of $\sh 1 W$
on  $\sh d W$, respectively of $\sh 1 {W'}$ on  $\sh d {W'}$.
 That is, for each $d < 0$, one has  a commutative diagram:
\[
\xymatrix{
\sh 1 W    \ar[rr]_{\text{ad}}  \ar[d]_{f^1} && \geg\gel (\sh d W) \ar[d]_{f^d}  \\
\sh 1 {W'} \ar[rr]_{\text{ad}}                    && \geg\gel (\sh d {W'} ).
}
\]
\end{enumerate}
 \end{thm}
 
 This can be applied to links. Note that in case two  Milnor fibers 
 (for different singularities) are  symplectomorphic, its is not clear that the induced contact structures on the boundary are contactomorphic. The above result gives conditions which make it possible to read 
 this off from Reeb orbits
 near the cylindrical ends. In dimension $3$ one deduces:
 
 \begin{corr} Let $\set{f=0}\subset \bC^{4}$ have  a normal   terminal  \ihs\   at the origin.   
 Then the bigraded symplectic cohomology    in negative degrees  
 of its  Milnor fiber $\mf f$ is a contact invariant of the link. 
 \label{cor:contactinv}
 \end{corr}
 \begin{proof}
 By  Proposition~\ref{prop:onMindiscr} the minimal discrepancy equals $1$. By Mclean's theorem
 ~\ref{thm:mclean} the Conley--Zehnder index for every closed Reeb orbit is at least $2=\max(5-3,3-1)$.
 Since the Milnor fiber is a parallellizable complex manifold, one has  $c_1(\mf f)=0$. Also, since $\mf f$ 
 is diffeomorphic to a handlebody obtained from the $6$-disc by attaching $\mu$ handles of index $3$, by Corollary~\ref{cor:OnMorseFuncts} 
 there is Morse function $\mf f \to \bR$ which has only index $0$ and $3$.   So all conditions are satisfied to apply the preceding  theorem.
 \end{proof}

\chapter{Symplectic   cohomology for  invertible matrix singularities} 
	\label{lect:Hochschild}

\section*{Introduction}

In Section~\ref{sec:SympIsHH}  it will be shown that  symplectic cohomology for the Milnor fiber of
large classes of  invertible matrix singularities 
 is the same as Hochschild cohomology for the category of equivariant matrix factorizations.
 But first, in \S~\ref{sec:prescript}, I shall
 specialize the prescription given in Section~\ref{ssec:hheq}  to the special case 
 of  invertible matrix singularities.  
 
By Corollary~\ref{cor:contactinv}   the bigraded symplectic cohomology    in negative degrees  
 of the Milnor fiber  of a $3$-dimensional normal terminal \ihs\ is a contact invariant of the link. 
In  Section~\ref{sec:SympIsHH} it is   shown that   for cDV-singularities $\set{w_A=0}$ 
of invertible matrix type  the Gerstenhaber structure on symplectic cohomology 
can be transported to $\hh * {A,\Gamma_A}$ in such a way  that it preserves the property of being a contact invariant of the link. This makes this bigrading  often computable in these cases.

In Section~\ref{sec:diagonal} and \ref{sec:diagexmple},  following \cite[\S2.4, \S3.1]{EvansLekili}, I shall explain how to   calculate the Hochschild cohomology with its Gerstenhaber structure   for 
 diagonal matrix-singularities in dimension $3$. As just explained, this also gives contact invariants for
 the links.
 These calculations give   an indication of  how to proceed in the other cases treated in  \cite{EvansLekili}.

\section{General prescription}
\label{sec:prescript}

Recall that an invertible matrix $A=(a_{ij})\in \gl  \bC {n+1}$ defines the  polynomial   
$w_A(\mbold x):=\sum_k   x_1^{a_{k,1}}  x_2^{a_{k,2 }} \cdots  x_{n+1}^{a_{k,n+1}}$ 
which   is an   invertible polynomial \ihs\
if the hypersurface   $w_A=0$  of $\bC^{n+1}$   has an isolated singularity at $\mbold 0$.
The entries of   $A$  define the group 
\[
\Gamma_A:= \sett{  (t_0,\dots,t_{n+1})\in (\bC^*)^{n+2}}{t_1^{a_k,1}\cdots t_{n+1}^{a_k,n+1} = t_0\cdots t_{n+1}, k=1,\dots,n+1 }. 
\]
Since $A$ is invertible, this is  in fact  a  finite group extension of $\bC^*$ which admits
the canonical character
\[
\chi_A  :   \quad \Gamma_A \mapright{\quad}  \bC^*,\quad 
 		  \quad \bt := (t_0,\dots,t_{n+1}) \mapsto t_0 \cdots t_{n+1} 
\]
whose kernel is the finite group  \label{page:GA}
\begin{eqnarray*}
   G_A  = \sett{ \bt\in  (\bC^\times)^{n+2}}{   t_1^{a_{k,1}}  t_2^{a_{k,2 }} \cdots  t_{n+1}^{a_{k.n+1}} =1, \, k=1,\dots,n+1,\, t_0= (t_1\cdots t_{n+1})^{-1}  } .
\end{eqnarray*} 

Now one can begin to specialize the description of $\hh * {\matf{w,\gamma\chi}}$
given in \S~\ref{ssec:hheq} in case 
 $w=w_A, \Gamma=\Gamma_A,\chi=\chi_A$.
 First note that  $(\bC^\ast)^{n+1}$    acts naturally on the polynomial ring $\bC[x_0,x_1,\dots,x_{n+1}]$ and  
 on   
 \[
 \widetilde R:= \bC[x_0,\dots,x_{n+1},x_0^* ,\dots, x_{n+1}^{* } ]
 \]
 by coordinate-wise multiplication:
 \[
  (t_0,\dots,t_{n+1}) \cdot x_j= t_jx_j ,\quad (t_0,\dots,t_{n+1}) \cdot x^*_j= t_j^{-1}x_j ^*.
  \]
 This induces an action of    $\Gamma_A$  on   $\widetilde R$ through further characters \label{page:ChiJ}
$\chi_j$, $j=0,\dots,n+1$,  which on $(t_0,\dots,t_{n+1})$  take the value $t_j$, that is
\[
\gamma(x_j) =\chi_j(\gamma) \cdot x_j ,\quad \gamma(x^ *_j) =\chi^{-1}_j(\gamma) \cdot x^*_j .
\]
 One clearly has:

 \begin{lemma*} $w_A$ is a  
 semi-invariant for the  $\Gamma_A$-action with character $\chi_A$ and $w_A$ is invariant under the action of $G_A$.
 \end{lemma*}

The individual variables  $x_j$ may or may not be invariant under the action of $\gamma\in G_A$, 
 and one accordingly divides  the indexing set $I=\set{1,\dots,n+1}$
 in two disjoint subsets $ I^{\gamma}$ and $I_{\gamma}$, 
 \begin{center}
  $i\in I^{\gamma} \iff x_i$ is fixed under the action of $\gamma$, \\
  $i\in I_{\gamma} \iff x_i$ is not fixed under the action of $\gamma$.
 \end{center}
 The polynomial $w_A^\gamma$ is the trace of $w_A$ in the $\gamma$-invariant polynomial ring.
In other words,  $w_A^\gamma$  is  obtained  from $w_A$ upon  setting all $x_j$,  $j\in I_\gamma$, to zero: 
  \[
 w_A^\gamma= w_A|_{\set{ x_j=0, \text{ for all } j\in I_\gamma}}.
 \]
 Note that $ w_A^\gamma$ only involves the $\gamma$-invariant variables.
Since $  t_0  =  t_0  \cdots t_  {n+1 } \cdot (t_1 \cdots t_  {n+1 })^{-1}$, 
the characters $\chi_A$ and $\chi_j$, $j=0,\dots,n+1$ satisfy the relation
\[ 
\chi_0= \chi_A\otimes \prod_{i=1}^{n+1} \chi_i^{-1},
\]
   one finds oneself exactly in the situation of Section~\ref{ssec:hheq}.  In the present situation  one has
\[
\kappa_\gamma= \prod_{j\in I_\gamma}  x_j^*.
\]
Since the contributions in the Hochschild cohomology come from monomials $\bm\in \widetilde R$,
  it is convenient to use the 
 corresponding monomial characters  $ \chi_\bm$ of the full torus $(\bC^*)^{n+2}$ given by \label{page:ChiM}
  \[
  \chi_\bm : (\bC^*)^{n+2} \to \bC^*,\quad \bt \mapsto t_0^{b_0} \cdots t_{n+1}^{b_{n+1}}, \text{ where }  b_j=\deg_{x_j}(\bm)-\deg_{x_j^{-1}}(\bm).
  \]
 Proposition~\ref{prop:hheqmatfac} motivates the  concept   of a $\gamma$-monomial:
 \index{$\gamma$-monomials}
 
 \begin{dfn} To  $\gamma\in G_A$ and $w_A$ one associates a set   $M_\gamma  :=  A_{\gamma} \cup B_{\gamma} \cup C_{\gamma}$
 of monomials in $\widetilde R$, the \textbf{\emph{$\gamma$-monomials}}, where

 \begin{description}
\item[(=case 1(b) of Prop.~\ref{prop:hheqmatfac}))] 
The set   $A_{\gamma}$  is empty if    $\gamma(x_0)\not=x_0$ 
and if  $\gamma(x_0)=x_0$  one has
\[
A_{\gamma}=\sett
{ x_0^{b_0}\cdot P \cdot \prod_{i\in I_{\gamma}} x^*_i }{ b_0 \geq 0  \text{ and }  P  \text{ monomial in } \in \jac {w_A^\gamma} },  
 \]
a collection of  monomials of   involving $x_0$,
$x_j$, $j\in I^\gamma$, $x_i^*$, $i\in I_\gamma$.  \label{page:AGamma}
 
\item[(=case 2 of Prop.~\ref{prop:hheqmatfac})]  The set   $B_{\gamma}$  is empty if  
  $\gamma(x_0)\not=x_0$ and if  $\gamma(x_0)=x_0$ one has
\[
B_\gamma =\sett {x_0^{b_0}\cdot P \cdot  x^*_0 \cdot \prod_{i\in I_{\gamma}} x^*_i}
{ b_0 \geq 0  \text{ and }  P \text{ monomial in }   \in \jac {w_A^\gamma}}, 
\]
a collection of   monomials   involving $x_0$, $x_0^*$, 
$x_j$, $j\in I^\gamma$, $x_i^*$, $i\in I_\gamma$. \label{page:BGamma}

\item[(=case 1a of Prop.~\ref{prop:hheqmatfac})] The set   $C_{\gamma}$  is empty if   
 $\gamma(x_0)  =x_0$ and if  $\gamma(x_0)\not =x_0$ one has
\[
C_{\gamma}=\sett
{P\cdot x^*_0\cdot \prod_{i\in I_{\gamma}} x^*_i }{ \text{ and }  P \text{ monomial in } \in   \jac {w_A^\gamma} },
\]
a collection of   monomials   involving  $x_0^*$, 
$x_j$, $j\in I^\gamma$, $x_i^*$, $i\in I_\gamma$. \label{page:CGamma}
 \end{description}

The condition  that a $\gamma$-monomial $\bm$ belongs to the $\chi^{\otimes u}_A$-character space translates as
$\chi_{\bm}=\chi_A^{\otimes u}$, suggesting  the following
 
  \begin{dfn}
A pair $(\gamma,\bm)$, consisting of $\gamma \in \text{ker}(\chi_A)$ and a $\gamma$-monomial $\bm$, 
 is said to be a \textbf{\emph{compatible  pair of weight $u\in \bZ$}}  if      $\chi_{\bm}=\chi_A^{\otimes u}$.
\end{dfn}

To determine the degree of such monomials in the Hochschild cohomology, let me provisionally introduce 
the \emph{\textbf{negativity}}\index{negativity (of $\gamma$-polynomial}

    of a $ \gamma$-monomial $\bm$   as    the number of  variables $x_0^*,\dots, x_{n+1}^*$ appearing in $\bm$. 
 \end{dfn}
  
 So, if $x_0$ is  not  fixed by $\gamma$,  the  only $\gamma$-monomials in $C_{\gamma}$ are those  involving $x_0^*$ and no powers of $x_0$.
In case   $\gamma(x_0) =x_0$,  there are two types: The $A$-types  which  
 possibly  involve a power of $x_0$ but not of $x_0^*$ while the $B$-types involve $x_0^*$ and  possibly  a power of $x_0$.
 The negativity  of a $\gamma$-monomial  of type $A_\gamma$  equals  $ \# I_\gamma $,  the total number of $x_j$,
$ j\in [1,\dots,n+1]$  that are not invariant  under  $\gamma$. 
 The other types have negativity  $ \#I_\gamma +1$. From  Proposition~\ref{prop:hheqmatfac}
 it then follows that a compatible pair $(\gamma,\bm)$ of weight $u$ and negativity $h$
 contributes to  $\hh{2u+h} {A,\Gamma_A}  $.

 \begin{exmple} \label{exm:gammapol} The group $G_A$ is finite. If $|G_A|=k$ and if  the subgroup $H $ that fixes each of  the  variables $x_1,\dots,x_{n+1}$ has order $\ell$, then $\#I_\gamma= n+1$ for $\gamma \in G_A\setminus H$.
  Thus   $(x_0\cdots x_{n+1})^*$ is  a  $\gamma$-monomial    of type $B$  $\iff \gamma(x_0)=x_0$, and otherwise  is of type $C$.  
  Hence one always has
 $(k-\ell)$ $\gamma$-monomials  of negativity $ n+2 $.
 \end{exmple}
 
 \begin{rmk} \label{rmk:gammapol} $\gamma$-monomials  can involve $x_j$,  $j\ge 1$ but  the exponent of such $x_j$ is bounded by 
 $\dim \jac {w_A}$. This implies that $\gamma$-monomials of  type $C$ have bounded total degree. 
Since $x_0$ can have arbitrary high exponent for   types $A$ and   types $B$,
 the total degree of such $\gamma$-polynomials can be arbitrarily high.
 \end{rmk}
 
  The group $\Gamma_A$ and its canonical character $\chi_A$ being defined by $A$,
 as in Section~\ref{sec:newresults} of the introductory chapter, I shall use
simplified notation:
\begin{equation}
 \hh * {A,\Gamma_A}:=\hh * {\matf {w_A,\Gamma_A,\chi_A}}.\label{eqn:SimpNot}
\end{equation}  
Summing up one  has:

\begin{thm}[\protect{\cite[Thm. 2.14]{EvansLekili}}] $\hh *{A,\Gamma_A}$ is the (possibly infinite dimensional) $\bC$-vector space  with basis the compatible $(\gamma,\bm)$-pairs. 
A compatible pair $(\gamma,\bm)$ of weight $u$ and negativity $h$  contributes to  $\hh{2u+h} {A,\Gamma_A}  $.
 \end{thm}

 \section{Relating symplectic cohomology to Hochschild cohomology}
 
 \label{sec:SympIsHH}

As mentioned in  the introduction, calculating  symplectic cohomology is in general difficult. 
In \cite{EvansLekili}  it is explained how one might reduce the calculation for  $n$-dimensional isolated 
singularities $w_A=0$ associated to  invertible matrices $A$,   to a purely algebraic one.
This explanation is highly technical; for the benefit of a general readership I shall give here a simplified  account skipping all technical    details.

There are two related categories that play the central role here. They are  associated to a perturbation of  $w_A$  which  only has ordinary double point singularities, a so-called \textbf{\emph{Morsification}} of $w_A$. \index{singularity!Morsification}
\index{Morsification (of singularity)}  Such perturbations  arise e.g.  in the semi-universal unfolding,  which has been briefly discussed in Section~\ref{ssec:ConstrSmallRes}.
To the resulting   family of  hypersurfaces acquiring at most ordinary double points  
one can apply the technique of Lefschetz:  each double point gives a 
vanishing cycle and a vanishing thimble in the total space of the family. 
It turns out that this association leads to   two corresponding so-called $A_\infty$-algebras $\cA$ (for the vanishing cycles)
and $\cB$ (for the vanishing  thimbles) which only depend  on the function $w_A$.\footnote{ Here one does not need to know what $A_\infty$-algebras or $A_\infty$-categories are. The reader can learn about these structures for instance in S. Ganatra's  thesis~\cite{Ganatra}.}  
As outlined in Section~\ref{ssec:catHH}, these algebras can be considered as categories as well, say  $
\underline{\cA}$ and $ \underline{\cB}$. 
Below a   "duality" between invertible matrix-singularities plays a role where   $A$  is replaced with its transpose $\trp A$.

In order to understand the formulation of these conjectures,  observe  that $\Gamma_A$ 
also fixes the polynomial  $x_{A':}=  x_A+(x_0\cdots x_{n+1}) $,
one can thus consider the category $\matf{w_{A',\Gamma,\chi}} $.
The following conjectures play  role in this story:
 
\begin{conj} \label{conj:central} \begin{enumerate}[\qquad (A)] 
\item The homotopy category $[\matf{w_ A ,\Gamma_A,\chi_A}]$ is  equivalent  to the homotopy category  $[\underline{\cA}]$ (associated to the vanishing thimbles  of a morsification of $w_{\trp A}$).
Moreover,  the algebra $\cA$ is \textbf{\emph{formal}} in the sense that its cohomology algebra $H^*(\cA)$  is quasi-isomorphic to $\cA$.
\item  The
homotopy category $[\matf{w_{A',\Gamma_A,\chi_A}} ]$  is  equivalent  to the homotopy category $[\underline{\cB }]$ (associated to  the vanishing thimbles of a morsification of $w_{\trp A}$).
\end{enumerate}

\end{conj}

Recalling the simplified notation~\eqref{eqn:SimpNot}, the announced relation is as follows:
\begin{thm} Assume that $w_A$ has non-zero amplitude (i.e., $w_A$ is not of log-Calabi type, see Example~\ref{ex:ihs}.5). If moreover
\begin{enumerate}
\item $\hh 2  { A,\Gamma_A}=0$;
\item  either Conjecture~\ref{conj:central}.(A) or \ref{conj:central}(B) holds,
\end{enumerate}
then $\hh *  { A,\Gamma_A}  \simeq \sh *  {\mf  {w_{\trp A}}}$ as Gerstenhaber  algebras.

\end{thm}
\begin{proof}[Sketch of the proof:]   
Assumption (1) implies two  crucial results used in  the proof:
\begin{itemize}
\item  linking  Hochschild cohomology of the category $\underline\cB$ to symplectic cohomology
of the Milnor fiber of $w_{\trp A}$:
\begin{equation}
\label{eqn:step1} \hh * {\underline \cB} \simeq \sh *  {\mf {w_{\trp A}}}
\quad \text{ (as Gerstenhaber algebras)}.
\end{equation}
This is a consequence of ~\cite[Thm. 6.4]{LekiliUeda} together with  \cite[Thm. 1.1]{Ganatra}. 
\item linking $\underline \cA$ and $\underline \cB $:
\[
\sf A \simeq \sf B,\quad \sf B = H^*(\underline\cB) .
\]
This is a consequence of  \cite[Eq.(1.0) and Section 2]{LekiliUeda}.
Consequently   $ \hh *  {A,\Gamma_A}  = \hh * {\sf B} $.
 Assumption (1)  then implies that also $\cB$ is formal so that  
\begin{equation}
\label{eqn:OnAandB}
\hh * {\sf A}\simeq \hh * {\sf B}=\hh * {\underline{\cB}}.
 \end{equation} 
\end{itemize}

Let me first assume that Conjecture~\ref{conj:central}(A) holds.   
Hence \begin{equation*}
 \hh *  {A,\Gamma_A} \simeq \hh * {\sf A} , \quad \sf A=H^*(\underline\cA).
\end{equation*} 
By \eqref{eqn:OnAandB}
  it thus follows that 
\begin{equation}
\label{eqn:step3} \hh *  {A,\Gamma_A}\simeq  \hh * {\underline{\cB}}.
\end{equation} 
Moreover, the isomorphisms preserve the Gerstenhaber-algebra structure. Equations~\eqref{eqn:step1} and  \eqref{eqn:step3}  yield   $\hh *  {A,\Gamma_A}  = \sh *  {\mf {w_{\trp A}}}$
as Gerstenhaber algebras which proves the theorem.

Next, assume that Conjecture~\ref{conj:central}(B) holds. 
By \cite[Theorem 2.15]{EvansLekili}  assumption (1)  implies that 
the two homotopy categories $[\matf{w_A,\Gamma_A,\chi_A} ]$ and 
$[\matf{w_{A'},\Gamma_A,\chi_A} ]$,   are equivalent. 
Hence   Conjecture~\ref{conj:central}.(B) together with \eqref{eqn:step1} then imply  the theorem.
 \end{proof}

The present status of the conjectures is as follows:

\begin{prop} \label{prop:ConjABTrue}
Conjecture~\ref{conj:central}(A) holds if 
\begin{enumerate}
\item $A$ is diagonal (cf. \cite{Futaki});
\item $A$ is block diagonal and its blocks are either $1$-by-$1$ or $2$-by-$2$ equal to $\begin{pmatrix}
 2&1\\0&k
\end{pmatrix}$  (cf. \cite{futueda});
\item $A= f(x_1,x_2)+\sum_{j= 3}^{n+1} x_j^2$  (cf. \cite{habersmith}).
\end{enumerate}
Conjecture~\ref{conj:central}(B) holds if
 $A$ is associated to an $A$-$D$-$E$-singularity (any dimension).  (cf. \cite{gamma,LekiliUeda,lekili2})
\end{prop}

\section{The diagonal case}
\label{sec:diagonal}

\subsection*{On the $\gamma$-monomials}
Assuming  that $A=\text{diag}(a_1,\dots,a_{n+1})$,  I shall explain how to find compatible pairs of $\gamma$-monomials.
The first task consists of comparing $\chi_A$ and the restriction   of the character $\chi_\bm$ to $\Gamma_A$.
Recall that  
\[
G_A= \ker (\chi_A)=  \sett{ \bt\in  (\bC^\times)^{n+2}}{   t_k^{a_k   } =1, \, k=1,\dots,n+1,\text{ and } t_0= (t_1\cdots t_{n+1})^{-1} },
\]
which is isomorphic to the product   $\mbold \mu_{a_1} \times\cdots \mbold \mu_{a_{n+1}} $ of  $(n+1)$ cyclic groups, each  generated a primitive  root of  unity.
The comparison will be done using the group
\[
\widetilde G_A:= \sett{  (t_0,\dots, t_{n+1}) \in  (\bC^\times)^{n+2}}{ t_j^{a_j}=1,\quad t=1,\dots, n+1}
\]
and the homomorphism\label{page:AuxGroup}
\begin{equation}
\label{eqn:T}
T: \widetilde G_A\to \Gamma_A,\quad ( t_0,\dots, t_{n+1}) \mapsto (  t_0^m \cdot( t_1  \cdots t_{n+1})^{-1}, t_0^{\ell/a_1} t_1,\dots, t_0^{\ell/a_{n+1}} t_{n+1}),
  \end{equation} 
where $\ell:= \text{lcm}(a_1,\dots,a_{n+1})$ and $m:= \ell- \sum_{i=1}^{n+1} \ell/a_i$. Notice that 
\[
\chi_A\comp T (t_0,\dots, t_{n+1})=(t_0\cdots t_{n+1})^\ell 
\]
which  gives a commutative diagram
\[
\xymatrix{
1 \ar[r]      &  G_A  \ar@{=}[d]            \ar[r]                                                    &\Gamma_A    \ar[r]_{\quad \chi_A\quad }                    &\bC^\times  \ar[r]                                & 1\\
1\ar[r]        & \mbold \mu_{a_1}\times\cdots\times \mbold\mu_{a_{n+1}} \ar[r]   & \widetilde G_A \ar[u]_{T} \ar[r]_{\quad \chi_A \comp   T\quad } & \bC^\times  \ar[r] \ar[u]_{t\mapsto t^\ell} & 1.
}
\]
To check if a $\gamma$-monomial $\bm$
is compatible and determine its weight, one has a simple procedure.  
This uses the reduced exponents of the monomial $\bm$ where $x_j^{p _j }(x^*_j)^{q_j}$
has reduced exponent $p_j-q_j$. In other words, write $x_j^{-1}$ instead of $x_j^*$ and allow  negative exponents. 
(Note that for a $\gamma$-monomial one has $q_j=0$ or $q_j=1$.)
For simplicity I  make the substitution $x_j^*=x_j^{-1}$ and use the  and use the 
reduced form $x_0^{b_0} \cdots x_{n+1}^{b_{n+1}}$. This creates an ambiguitiy  only for $x_0$
since for $j\ge 1$, the variables $x_j$ and $x_j^*$ never occur simultaneously in $\gamma$-monomials.
 
\begin{lemma} \label{lemm:ChinRem}
Let $\bm= x_0^{b_0} \cdots x_{n+1}^{b_{n+1}}$ be a $\gamma$-monomial   in reduced form. 
Then the character $\chi_\bm$ is a
power of $\chi_A$ if and only if  integers $m_1,\dots, m_{n+1}$ exist such that $b_i= b_0-m_ia_i$ for $i=1,\dots, n+1$ and then
$\chi_\bm= \chi_A^{\otimes u}$ with $u= b_0-\sum m_i$.
\end{lemma}
\begin{proof} Observe that $\chi_\bm(\bt)=t_0^{b_0}\cdots t_{n+1}^{b_{n+1}}$ and so, using Eqn.~\eqref{eqn:T} one finds
\begin{equation}
\label{eqn:TwoChars}
\chi_\bm \comp T (\bt)=  t_0^{c_0} \cdots  t_{n+1}^{c_{n+1}},\quad c_0=mb_0+\sum_{ i=1}^{n+1}  b_i \frac \ell {a_i},\quad c_i= b_i-b_0,\, i=1,\dots, n+1. 
\end{equation}
I claim that this  is a power of $\chi_A$ if  $c_i \equiv 0\bmod a_i$ for $i=1,\dots, n+1$. In fact, writing 
 \[
  b_i=   b_0   - m_i a_i  , \quad i=1,\dots, n+1,
\]
one finds $c_0= \ell (b_0-\sum m_i)$ and so 
\[
\chi_\bm \comp T (\bt)=( (t_0\cdots  t_{n+1})^\ell)^{b_0-\sum m_i}  =(\chi_A\comp T (\bt) )^{b_0-\sum m_i} \text{ for all } \bt\in \widetilde G_A. 
\] 
Conversely, if $\chi_\bm \comp T (\bt)=\chi_A(T(\bt))^u= ( (t_0\cdots  t_{n+1})^\ell)^u$, then \eqref{eqn:TwoChars} implies that $b_i-b_0=\ell u$, but $\ell$ is a multiple of $a_i$ for all $i=1,\dots,n+1$ and so $b_i-b_0\equiv 0\bmod a_i$.
\end{proof}
 
As a consequence, if $(\gamma,\bm = x_0^{b_0}\cdots x_{n+1}^{b+1})$ is  a compatible pair, then its weight equals $u=b_0-\sum m_i$. Hence the sole exponent 
$b_0$  determines the other exponents $b_j$ as well as the weight by   solving the congruences 
\begin{equation}
 \label{eqn:Congr}
 b_i\equiv b_0 \bmod a_i , \quad i=1,\dots,n+1 .
\end{equation} 
Now by Remark~\ref{rmk:gammapol}, the appearance of $x_j^{b_j}$, $j\ge 1$ is governed by the Jacobian ring of the polynomial $w_A$. 
In the present  situation all $w_A^\gamma=\sum_{j \in I^\gamma} x_j^{a_j} $ are diagonal and  
\[
\jac{w_A^\gamma}  = \bC \cdot 1\oplus  \bigoplus_{\overrightarrow {k_j}=\overrightarrow 1}^{\overrightarrow {a_j-1}}  \bC\cdot \prod_{j \in I^\gamma}   x_j^{k_j}   .
\]

\begin{exmples} \label{exs:gammamons}
1. I claim that $\dim \hh n{\Gamma_A}\ge  \prod_{i=1}^{n+1} (a_i-1)$. This can be seen by the above procedure, setting $b_0=-1$. Then the  congruences \eqref{eqn:Congr} have
 a solution $b_i=-1$, $m_i=0$, $i=1,\dots,n+1$. So $x_0^{-1}\cdots x_{n+1}^{-1}$ is    a $\gamma$-polynomial for all  
$\gamma\in G_A$ that fix no variable $x_j$, $j\ge 1$,  i.e,  
$\gamma\in  (\mbold\mu_{a_1}\setminus \set{1} )\times \cdots\times  (\mbold\mu_{a_{n+1}}\setminus \set{1})$. Hence there are $\prod_{i=1}^{n+1} (a_i-1)$ such   compatible pairs $(\gamma,  x_0^{-1}\cdots x_{n+1}^{-1})$ of weight  $u= -1$ and negativity $(n+2)$
  which all  contribute  to $\hh n{\Gamma_A}=\hh  {-2+n+2}{\Gamma_A}$.
 So its dimension is at least  $\prod_{i=1}^{n+1} (a_i-1)$.
 \\
 2. Likewise, setting $b_0=0$,  one finds that  $b_i= 0$, $m_i=0$, $i=1,\dots,n+1$. 
Hence the reduced form of $\bm$ equals the constant polynomial $1$ and so  $\bm=1$,  an $A_\gamma$-type polynomial,
or $\bm=x_0x_0^*$, a $B_\gamma$-polynomial of negativity $1$ and weight $0$. Hence 
 $\hh 1 {\Gamma_A}$ is at least one-dimensional.     If $n=3$, the calculations in
  Section~\ref{sec:diagexmple} show that this is the only class in $\hh {1}{\Gamma_A}$.
  By Theorem~\ref{thm:scaling} below, since the non-zero class  is represented  by $x_0x_0^{-1}$, it belongs to  $\hh {1,0}{\Gamma_A}$ precisely because
  $\dim \hh {1,0}{\Gamma_A}=1$.  This is a central observation which makes it possible to calculate the  contact invariants of the examples of the  links discussed in the last section.
     \\
  3. In a similar way, for a fixed  $b_0 >0$, solving  the  congruences \eqref{eqn:Congr},
  one  checks whether $\bm(b_0):=x_0^{b_0} \cdots x_{n+1}^{b_{n+1}}$
  is an $A_\gamma$-monomial for some  $\gamma\in G_A$.
  By construction it yields the  compatible pairs and their weights and hence to which degree in Hochschild cohomology the monomial contributes.
 \end{exmples}

Elaborating the last example, let me outline a practical way to find a $\gamma$-monomial $\bm_A(b_0)$ of $A$-type  whose $x_0$-exponent is a given positive integer $b_0>0$.
Taking a look at the $A$-monomials, one sees that for each $i\in I_\gamma$ the variable $x_i^{-1}$ appears. So then $b_i=-1$ and this leads to the congruence
$b_0+1\equiv 0 \bmod a_i$. Pick out all $i$ for which this is possible. Then these form $I_\gamma$.
So now the complementary set  $I^\gamma$ is known and one determines the monomials  which span $\jac{w_A^\gamma}$.
Now $d_0 \bmod a_i$ gives a unique remainder $b_i<a_i$ and the required $\gamma$-monomial  becomes
\[
\bm_A(b_0):= x_0^{b_0} \cdot \prod_{j \in I^\gamma}   x_j^{b_j} \cdot \prod_{i\in I_\gamma} x_i^{-1},
\]
where one sets $x_j^{a_j-1} =1$.
The associated $B$-monomial then is $\bm_B(b_0)=\bm_A(b_0) x_0 x_0^{-1}$. Notice  that if indeed $\gamma\in \Gamma_A$ exists with 
$\gamma(x_0)=x_0$ and  $\gamma(x_i)=x_i\iff  i\in I^\gamma$, 
the pairs $(\gamma,\bm_A(b_0))$  and $(\gamma,\bm_B(b_0))$
yield  compatible pairs. 

\begin{exmple} Consider $w:=x_1^2+x_2^3+x_3^5+x_4^7$ and $b_0=38$.
Then $b_0+1=39\equiv 0 \bmod 3$ but $39\not \equiv 0\bmod 2,5,7$ and so $I_\gamma=\set{2}$.
Hence  $\jac{w_{ \gamma}}$ is spanned by $x_3^a x_4 ^b$ with $a=0,\dots,3$, $b=1,\dots,5$.
Since $b_0-b_3\equiv 3- b_3\equiv 0\bmod 5$, one has $b_3=3$ and likewise $38- b_4 \equiv 0 \bmod 7$ gives $b_4=3$.
Hence $\bm_A(38)= x_0^{38} x_3^3x_4^3 x_2^{-1}$.
\end{exmple}

\subsection*{Computing the second grading}
Let me now  describe (without proof)
  the bigrading on $ \hh * {A,\Gamma_A}$   under  the identification 
\[ \hh * {\matf {w_A}}= \hh * {A,\Gamma_A}.
\]
This requires tracing through all of the identifications from Section~\ref{sec:SympIsHH}. This is quite involved. The final result is as follows:

\begin{thm}[\protect{\cite[Lemma 4.6]{EvansLekili}}]  
\label{thm:scaling} Let $w_A=0$  be an $n$-dimensional isolated singularity
  associated to  an invertible matrix $A$.     Suppose that  assumption \eqref{eqn:SH10} holds.    Then
 the bigrading  on the side of symplectic  cohomology   on the Milnor fibre $\mf {w_A}$  given by the representation
\[
\sh 1 {w_A} \rightarrow \bigoplus_d \mathfrak{gl}(\sh d {\mf {w_A}} ) 
\]
is scale equivalent to the bigrading by the total exponent of $x_0$ on Hochschild cohomology:
\[
\hh d  {A,\Gamma_A} =\bigoplus_q \hh {d-q,q}  {A,\Gamma_A},
\]
where  a $\gamma$-monomial $m$ whose total exponent of $x_0$ in $m$ is  $b_0$
contributes to the bigraded piece $\hh {d-nb_0,nb_0}{A,\Gamma_A}  $. 
\end{thm}

 The bigrading on $\bigoplus_{d<0} \sh d {w_A}$ (which by Theorem~\ref{thm:ContInvOfLnk}) is a contact invariant of the link) 
 can therefore be computed in terms of the $x_0$-powers of the contributing $\gamma$-monomials. This information can be used to distinguish non-isomorphic contact structures on the link of the Milnor fibre as will
 be explained below in Section~\ref{ssec:bigrading} for the example of diagonal CDV-singularities.

 \section{A diagonal cDV-example} 
 \label{sec:diagexmple} 
I now give the full details for  calculations which leads to the results stated in Example~\ref{exm:A1k}. Here  $(a_1 ,a_2, a_3,a_4)=(2,2,2,2k)$. 

\subsection{The symplectic cohomology of the Milnor fiber  in terms of the Hochschild cohomology}
Recall the statement:
\begin{thm*}
\[
\dim (\sh d  {A_1(2k) })= \begin{cases}
				0 & d\ge4\\
				2k-1 & d=3\\
				0 & d=2\\
				1 & d\le 1
  			\end{cases} 
\]
\end{thm*}

To show this, first note that $
G_A= \mbold\mu_1\times \mbold\mu_1\times \mbold\mu_1\times\mbold\mu_{2k}$
with action  given by\footnote{Here $\rho_k$ denotes a primitive $k$-th root of unity.} 
\begin{align*}
( (-1)^{b_1},(-1)^{b_2},(-1)^{b_3},(\rho_{2k})^{b_4} )&\cdot  (x_0,x_1,x_2,x_3,x_4)\\
& \hspace{-4em}=((-1)^{b_1+b_2+b_3}  \rho_{2k} ^{b_4} x_0, (-1)^{b_1} x_1, (-1)^{b_2} x_2,(-1)^{b_3}x_3,
(\rho_{2k})^{b_4} x_4). 
\end{align*}
Let me make this more explicit in  the following tables. There  I use the  convention  that $\alpha_i$ is the generator of the $i$-th cyclic factor of $G_A)$ and 
$\alpha_{ij}=\alpha_i\times \alpha_j$. Hence $\alpha_1,\alpha_2,\alpha_3$ have order $2$ while $\alpha_4=\rho_{2k}$, and in Table~\ref{tab:1}  only exponents $1,\dots,2k-1$
are allowed to occur. 
 Observe that $\alpha_4^k =\rho_{2k}^k=-1$ so that $(-1,-1,-1,  \alpha_{4} ^k)$ as well as $1_4$, $(-1,-1,1,1), (-1, 1,-1,1)$ and $ ( 1,-1,-1,1)$ fix  $x_0$. This leads to Table~\ref{tab:1} and Table~\ref{tab:2} 
   \begin{table}[htp]
 \caption{$\gamma$ with $\gamma(x_0)=x_0$}\label{tab:1}
\begin{center} 
\begin{small}
\begin{tabular}{|c|c|c|c|c|c|c|}
\hline
$\gamma$ & $1_3 \times  1$ &   $ \alpha_{i }\times 1$ & $\alpha_{i j}\times \alpha_4 ^k$ &   $ (-1)_3\times \alpha_4 ^k $
\\
\hline
other   $x_k$   fixed  &  $k= 1,2,3,4$  &     $k\not= i , k=4$ & $k\not=i,j$ &   none
\\
\hline 
$|I^\gamma|$ & $4$ & $3$ & $1$ &   $0$
\\
\hline
$\dim \jac{w_A^\gamma }$ &  $2k-2$ & $2k-2$ &   $0$ & $0$  
\\
\hline
$A_\gamma$-polyn. &$x_0^{b_0} x_4^{\ell}$ & $x_0^{b_0} x_4^{\ell}x_i^{-1}$ &  $x_0^{b_0} x_i^{-1}x_j^{-1}x_4^{-1}$ &  $x_0^{b_0}  \prod_{i=1}^4 x_i^{-1}    $ \\
\hline
\end{tabular}
\end{small} 
\end{center}

\end{table}
  
 \begin{table}[htp]
 \caption{$\gamma$ with $\gamma(x_0 )\not=x_0$}\label{tab:2}
\begin{center}
\begin{small}
\begin{tabular}{|c|c|c|c|c|c|c|c|}
\hline
$\gamma$ &$1_3 \times  \alpha_4 ^\ell$& $\alpha_i \times  \alpha_4^\ell$ & $ \alpha_{ij}\times  \alpha_4^\ell$  
 & $ (-1)_3\times  \alpha_4^\ell $ & $ \alpha_{ij}\times 1$ & $ (-1)_3\times  1 $
\\
\hline
 $\gamma(x_k)=x_k $&$k= 1,2,3$& $k\not=i $ &   $k\not=i,j$ & none  & $k\not=i$ & $k\not=i,j$ 
\\
\hline 
$|I^\gamma|$ & $3$ & $2$ & $1$ & $0$        &$2$ & $1$
\\
\hline
$\dim \jac{w_A^\gamma }$   &0&0& 0&0&   $2k-2$ &$2k-2$
\\
\hline
  $C_\gamma$-polyn.  &  $  x_4^{-1}$& $  x_i^{-1}x_4^{-1}$ & $  x_i^{-1}x_j^{-1} x_4^{-1}$  & $  \prod_{i=1}^4 x_i^{-1}$ 
  &  $  x_i^{-1}  x_j^{-1} x_4^{\ell}$ & $  \prod_{i=1}^3 x_i^{-1} x_4^{\ell }$ \\
\hline
\end{tabular}
\end{small}
\end{center}
\end{table}
 
The next task is to find compatible pairs. This necessitates solving the equations $b_0\equiv b_i \bmod 2$ for $i=1,2,3$ and $b_0\equiv  b_4 \bmod 2k$
for any given integer $b_0>0$.  Now, by the euclidean algorithm, one may write
\[ 
b_0 = 2p k+ 2q+r, 
\quad p\ge 0,\quad \quad 0\le q\le k-1, \quad r=0,1.
\]
One has to play this off  against  the $x_i$, $ 1,2,3$, having exponents $b_i=0$ or $b_i=-1$.
\par 
Let me first find the $A_\gamma$-polynomials.
\begin{description}
\item[ Case 1] $b_0$ is even. Then $r=0$ and $b_1=b_2=b_3=0$ and  $b_4= 2q$, $0\le q\le k-1$.
Recalling Lemma~\ref{lemm:ChinRem}, 
one finds $m_1=m_2=m_3=pk+q$, $m_4=p$ and so the weight is $u=b_0-\sum m_i= -((k+1) p+q)$.  
Such polynomials  $x_0^{b_0} x_4^{2q}$ are $A_\gamma$-polynomials for $\gamma=1$. These contribute each  to $\hh {2u}{A_1(2k)}=\hh {-2((k+1) p+q)}{A_1(2k)}$.
\item[Case 2] $b_0$ is odd (so $r=1$). One sees that $b_1=b_2=b_3=-1$, $b_4=  2q+1$, $0\le q\le k-1 $.
Also, $m_1=m_2=m_3=pk+q+1$,  $m_4=p$ if $q\le k-2$ and $m_4=-1$ if $q=k-1$. Only the last possibility gives an $A_\gamma$-polynomial,
namely  $x_0^{b_0} x_1^{-1}x_2^{-1} x_3^{-1}x_4^{-1}$ for $\gamma=  (-1)_3\times \alpha_4 ^k $
whose   weight is   $u= -((k+1) p +k+2 )$. It contributes to $\hh{2u+4}{A_1(2k)}=\hh{-2 ( (k+1) p+k)}{A_1(2k)}$.
\end{description}
\par

This shows that   $\dim \hh m{A_1(2k)}=1$ for  each   non-positive   even  degree $m$, since   $2((k+1) p+q)$, $q=0,\dots,k$  runs over  all possible 
even numbers (or $0$)  because   $p  $ can be any   non-negative integer.
 \par
 Next, multiplying  the above $A_\gamma$-polynomials with $x_0^{-1}$  gives $B_\gamma$-polynomials which 
  gives the odd negative degrees. For $b_0=0$ one only finds $1\in \hh 0 {A_1(2k)}$. 
  The only other $B_\gamma$-polynomial (or $C_\gamma$-polynomial) is $(x_0\cdots x_{4})^{-1}$ which contributes one dimension to $\hh 2{A_1(2k)}$  for  each
  group element of the form  $\gamma=(-1_3\times \alpha_4^\ell)$, $\ell=1,\dots,2k-1$. 
  See Example~\ref{exs:gammamons}.1.   The other $C_\gamma$-polynomials do not give contributions.

 \subsection{The bigrading and contact invariants} \label{ssec:bigrading}
 
 Recall  that the link of $A_1(2k)$ is diffeomorphic to $S^2\times S^3$ and its contact
 structure is denoted   \index{contact structure!on $S^2\times S^3$} 
 $\alpha_{1,k}$.  The goal is to show that the second grading distinguishes these contact structures.
 The crucial tool is Theorem~\ref{thm:scaling} which asserts that a   $\gamma$-monomial  
 $ x^{b_0}_0x_1^{b_1}x_2^{b_2}x_3^{b_3}x_4^{b_4}$ contributes to $\hh {d-3b_0,3b_0}{A,\Gamma_A}  $.
  So in this case the unique contribution to  degree $-2$ is:
\[
\begin{cases}x_0x_1^{-1} x_2^{-1} x_3^{-1} x_4^{-1}\text{ bidegree } (-5,3) &\text{ for }\alpha_{1,1},\\ 
		    x_0^2x_4^2\text{ bidegree } (-8,6)  &\text{ for }\alpha_{1,k},\ (k\geq 2). 
		    \end{cases}
    \] 
 This already distinguishes $\alpha_{1,1}$  from everything else.
 To make   comparison of the various contact structures easier, one can rescale the second degree for degrees $d<-2$
 coming from  the contribution of  $ x_0^{b_0}\cdots x_4^{b_4} $ to be 
$ 4b_0 $ for $\alpha_{1,1}$  and  $  2b_0 $ for  $\alpha_{1,k}$, $k\ge 2$.
Then  the unique contribution in degree $-4$  is:
\[
\begin{cases} x_0^4 \text{ in bidegree }   (-20,16)& \text{ for } \alpha_{1,1}, \\ 
		   x_0^3x_1^{-1} x_2^{-1} x_3^{-1} x_4^{-1} \text{ in bidegree }  (-10,6) &\text{ for } \alpha_{1,2},\\ 
	           x_0^4x_4^4 \text{ in bidegree }   (-12,8)&\text{ for }\alpha_{1,k}, (k\geq 3).
    \end{cases}
  \] 
   This   distinguishes    $\alpha_{1,2}$ from the other $\alpha_{1,k}$, 
$k\neq 2$.  To distinguish $ \alpha_{1,k}$ from $\alpha_{1,K}$ with
$2\leq k < K$, observe that the unique contribution to  degree $-2k$  
is $x_0^{2k-1}x_1^{-1} x_2^{-1} x_3^{-1} x_4^{-1}$ in bidegree $(-6k+2,4k-2)$ for $\alpha_{1,k}$,
respectively  $x_0^{2k} x_4^{2k}$ of bidegree $(-6k,4k)$ for $ \alpha_{1,K}$. The result is summarized in the table below
which shows that  $\alpha_{1,k}$, and $ \alpha_{1,j}$ are not contactomorphic if $k\not=j$. 

To interpret  the table, recall that
$\dim \sh d {A_{2,k}}=1$ for $d<0$ so that in all cases there is one generator in each bidegree so that indeed
the second degree (which is printed in red here)  distinguishes the contact structures $\alpha_{1,k}$
among each other.

\begin{table}[htp]
\caption{\small Bidegrees $(d-p,{\color{red} p})$  contributing to $\sh d {A_{2,k}}$, $d<0$, $d$ even.}
\begin{center}
\begin{small}
\begin{tabular}{|c|c|}
\hline 
$k$ & $(d,{\color{red}p} )$\\
\hline
$1$ & $(-2,{\color{red}3})$ \\
$2$ &     $(-2,{\color{red}6})$ \\
\hline
$1$&  $(-4,{\color{red}16})$\\
$2$& $(-4, {\color{red}6})$ \\
$k\ge 3$ & $(-4,  {\color{red}8})$ \\
\hline
$3\le k\le K-1$ & $(-2k,{\color{red}4k-2})$\\
$K$ &   $(-2k, {\color{red}4k})$\\
\hline
\end{tabular}
\end{small}
\end{center}
\label{tab:A1k}
\end{table}

 \bibliographystyle{alpha}

 \printindex 

\end{document}